\documentclass[10pt]{amsart}
\usepackage[utf8]{inputenc}
\usepackage[T1]{fontenc}
\usepackage[english]{babel}

\usepackage{color}
\usepackage{hyperref}
\usepackage{amsthm}
\usepackage{amsmath}
\usepackage{amssymb}
\usepackage{mathrsfs}
\usepackage{bbold}
\usepackage{graphicx}
\usepackage[top=2cm, bottom=2cm, left=2.4cm, right=2.4cm]{geometry}
\usepackage{listings}
\usepackage{mathrsfs}
\usepackage{enumerate}
\usepackage{pict2e}
\usepackage{stmaryrd}
\usepackage{mathtools}
\usepackage{mathabx}

\usepackage[all,cmtip]{xy}

\newtheorem{defi}{Definition}[section]
\newtheorem{theo}[defi]{Theorem}
\newtheorem{propo}[defi]{Proposition}

\newtheorem{lem}[defi]{Lemma}
\newtheorem{corol}[defi]{Corollary}
\newtheorem{rem}[defi]{Remark}

\DeclareMathOperator{\diff}{d}

\DeclareMathOperator{\ug}{\boldsymbol{u}}
\DeclareMathOperator{\vg}{\boldsymbol{v}}

\DeclareMathOperator{\phig}{\boldsymbol{\phi}}

\DeclareMathOperator{\Lag}{\mathscr{L}}
\DeclareMathOperator{\HDNLS}{H_{ \rm DNLS}}

\DeclareMathOperator{\Span}{Span}

\DeclareMathOperator{\sinc}{sinc}
\DeclareMathOperator{\domega}{d\omega}
\DeclareMathOperator{\dx}{dx}
\DeclareMathOperator{\ds}{ds}

\DeclareMathOperator{\supp}{supp}

\DeclareMathOperator{\id}{id}

\DeclareMathOperator{\xO}{y}

\newcommand\unupoop[3]{\mathrel{\mathop{#1}\limits_{#2}^{#3} }  }

\begin{document}

\title[Discrete traveling waves for DNLS]{Existence and stability of traveling waves for \\ discrete nonlinear Schr\"odinger equations over long times}

\author{Joackim Bernier}
\address{Univ Rennes, CNRS, IRMAR - UMR 6625, F-35000 Rennes}
\email{Joackim.Bernier@ens-rennes.fr}

\author{Erwan Faou}
\address{INRIA, Univ Rennes, CNRS, IRMAR - UMR 6625, F-35000 Rennes}
\email{Erwan.Faou@inria.fr}

\maketitle

\begin{abstract}
We consider the problem of existence and stability of solitary traveling waves for the one dimensional discrete non linear Schr\"odinger equation (DNLS) with cubic nonlinearity, near the continuous limit.
We construct a family of solutions close to the continuous traveling waves and prove their stability over long times. Applying a modulation method, we also show that we can describe the dynamics near these discrete traveling waves over long times.  
\end{abstract}

\tableofcontents

\section{Introduction}
\subsection{Motivations and main results}
We study existence and stability of solitary traveling waves for the discrete nonlinear Schr\"odinger equation (DNLS) on a grid $h\mathbb{Z}$ of stepsize $h>0$ and with a cubic focusing non linearity. This equation is a differential equation on $\mathbb{C}^{h\mathbb{Z}}$ defined by (see \cite{MR2742565} for details about its derivation)
\begin{equation}
\label{def_DNLS}
\forall g\in h\mathbb{Z}, \quad \ i\partial_t \ug_g = \frac{\ug_{g+h}  - 2 \ug_g + \ug_{g-h}}{h^2} + |\ug_g|^2 \ug_g.
\end{equation}
We focus on this equation near its continuous limit (as $h$ goes to $0$), called non linear Schr\"odinger equation (NLS), defined as the following partial differential equation 
\begin{equation}
\label{def_NLS}
\forall x\in \mathbb{R}, \quad\ i\partial_t u(x) =\partial_x^2 u(x) + |u(x)|^2 u(x).
\end{equation}
We study solutions of DNLS \eqref{def_DNLS} with a behavior close to the continuous traveling waves of NLS \eqref{def_NLS}. Such solitons $u$ are  global solutions of NLS with speed of oscillation $\xi_1$ and speed of advection $\xi_2$, satisfying 
\begin{equation}
\label{def_CTW_abs}
\forall t_0 \in \mathbb{R},\quad  \forall t\in \mathbb{R},\quad\forall x\in \mathbb{R},\quad \ u(t_0 + t,x)=e^{i\xi_1 t} u(t_0,x-\xi_2 t).
\end{equation}
The parameter $\xi = (\xi_1,\xi_2)$ characterizes travelling waves up to gauge transform $u(x) \mapsto e^{i \gamma} u(x)$ and advection $u(x) \mapsto u(x - \xO)$. For NLS they are given explicitly  by their values at time $t=0$
\begin{equation}
\label{def_CTW}
 \forall x\in \mathbb{R},\quad \ \psi_\xi(x) = e^{\frac12 i x \xi_2 } \frac{  \sqrt{2} m_\xi }{\cosh (m_{\xi}x)} \quad \textrm{ with } \quad m_{\xi}=\sqrt{ \xi_1 - \left(\frac{\xi_2}2 \right )^2}.  
\end{equation}
for speed of oscillation $\xi_1$ and speed of advection $\xi_2$ satisfying
\begin{equation}
\label{ine_speed}
\xi_1 > \left( \frac{\xi_2}2 \right)^2 .
\end{equation}

On a grid, the notion of traveling wave is not as clear as on a line, and we cannot  define traveling waves for DNLS as easily as those of NLS by \eqref{def_CTW_abs}. The difficulty comes from the definition of the advection. Indeed, the canonical advection on a grid is only defined when the distance to cross is a multiple of the stepsize $h$.
Of course, we could find some reasonable extensions of \eqref{def_CTW_abs} in the discrete case. For example, a possible definition of discrete traveling waves  could be for solution $\ug$ to DNLS to satisfy
 \begin{equation}
\label{def_DTW}
\forall t_0 \in \mathbb{R}, \quad \forall n\in \mathbb{Z}, \quad \forall g\in h\mathbb{Z}, \ \quad  \ug_g(t_0 + n\tau)= e^{i \xi_1 n\tau} \ug_{g-nh}(t_0) \quad \textrm{ with } \quad \xi_2 \tau = h, 
\end{equation}
for some speeds $\xi_1,\xi_2 \in \mathbb{R}$. Even if this definition seems to be the most natural, it is not the only one possible. For example, we could replace $h$ by $2h$ in this definition or to do things even more complicated, and no canonical choice appears obvious. 
 There is at least one class of solutions that can be defined without ambiguity, the standing waves (i.e. when $\xi_2=0$) which are solutions of the form 
\begin{equation}
\label{def_DSW}
 \forall t_0\in \mathbb{R},\forall t\in \mathbb{R}, \ \ug(t_0+t) = e^{i\xi_1 t}\ug(t).
\end{equation}
for some speed of oscillation $\xi_1\in \mathbb{R}$.


 We define the discrete $L^2$ and $H^1$ norms as follows: for $\vg \in \mathbb{C}^{h\mathbb{Z}}$,
\[ \|\vg\|_{L^2(h\mathbb{Z})}^2 = h\sum_{g \in h\mathbb{Z}} |\vg_g|^2 \quad \textrm{ and } \quad \| \vg \|_{H^1(h\mathbb{Z})}^2 = h\sum_{g \in h\mathbb{Z}} \left| \frac{\vg_g - \vg_{g-h}}{h} \right|^2 + \|\vg\|_{L^2(h\mathbb{Z})}^2. \]
Of course, these norms are equivalents but not uniformly with respect to $h$.  Since  we focus on the continuous limit (i.e. when $h$ goes to $0$), uniformity with respect to $h$ is crucial. 


The discrete $L^2$ norm,  $\| \cdot \|_{L^2(h\mathbb{Z})}^2$ is a constant of the motion of DNLS associated, through Noether Theorem (see, for example, \cite{MR3443560} for details about this Theorem), to its invariance under gauge transform action. As $L^2(h\mathbb{Z})$ is an algebra we can deduce  by Cauchy-Lipschitz Theorem that DNLS is globally well-posed in $L^2(h\mathbb{Z})$. Moreover, DNLS is a Hamiltonian system associated with the  Hamiltonian  
\begin{equation}
\label{def_HDNLS}
\HDNLS(\ug) = \frac{h}2 \sum_{g \in h\mathbb{Z}} \left| \frac{\ug_{g+h} - \ug_g}h \right|^2 - \frac{h}4 \sum_{g \in h\mathbb{Z}} |\ug_g|^4. 
\end{equation}
As we can guess from its expression, this Hamiltonian is very useful to establish some estimates of coercivity with the discrete $H^1$ norm, uniformly with respect to $h$.   

The continuous traveling waves of NLS defined by \eqref{def_CTW} verify a property of stability called orbital stability. If for a given time a solution of NLS is close enough of a traveling wave, then it stays close of this traveling wave for all times, up to an advection and a gauge transform. This property has been first proven by Cazenave and Lions in 1982 in \cite{MR677997}  by a compactness method  and in 1986 by Weinstein in \cite{MR820338} with what we call nowadays the {\it energy-momentum \rm} method. This second method is more quantitative than the first one, and the estimates of stability we give in this article are all based on it. It has been developed by Grillakis, Shatah and Strauss in 1987 in \cite{MR901236} and \cite{MR1081647} (see also \cite{MR3443560} for a very clear presentation of this method).
\begin{theo} Cazenave and Lions \cite{MR677997}, Weinstein \cite{MR820338} \\
\label{Thm_COS}
For each couple of speed $\xi \in \mathbb{R}^2$, such that $\xi_1 > \left( \frac{\xi_2}2 \right)^2$, there exists a constant $c>0$, such that for all solutions $u$ of NLS \eqref{def_NLS} with $\| u(0) - \psi_\xi \|_{H^1(\mathbb{R})} < c$ and $\| u(0)  \|_{L^2(\mathbb{R})}^2 = \|  \psi_\xi \|_{L^2(\mathbb{R})}^2$, for all time $t\in \mathbb{R}$, there exist $\xO,\gamma\in \mathbb{R}$ such that
\[ c \| u(t) - e^{i\gamma}\psi_\xi(x-\xO) \|_{H^1(\mathbb{R})} \leq \| u(0) - \psi_\xi \|_{H^1(\mathbb{R})}. \]
\end{theo}
This result does not give any information on the exact position of the solution.  To remedy this problem, \it modulational stability \rm  methods have been developed, which allows to follow very precisely this solution (see \cite{MR783974} or \cite{MR2094474}). 

If we try to apply energy-momentum method to construct orbitally stable traveling waves for DNLS, the main difficulty comes from the definition of the advection on the grid. We discuss this problem in detail in section \ref{Sec_2}. 
However this problem is easily solved when considering standing waves (i.e. $\xi_2=0$) with symmetric perturbations for which the solution, remaining symmetric for all times, cannot move. 
In $2010$, Bambusi and Penati proved in \cite{MR2576378} the existence of standing waves of DNLS  looking like those of NLS. In fact, they constructed two kinds of standing waves. Each ones are real valued and symmetric but the first ones, called {\it Sievers-Takeno modes \rm} or {\it onsite \rm}, are centered in $0$ whereas the second ones, called {\it Page modes \rm} or {\it off-site, \rm} are centered in $\frac{h}2$. In $2013$, in \cite{MR3018143}, Bambusi, Faou and Gr\'ebert, studying fully discrete approximation in time and space of NLS standing waves, gave some results of their orbital stability.
The construction of these standing waves is also realized in a $2016$ paper of Jenkinson and Weinstein (see \cite{MR3460749}), with another kind of approximations. If we focus only on the \it onsite \rm standing waves, we summarized a piece of these results in the following theorem.
\begin{theo} Existence and orbital stability of standing waves \\
\label{Thm_DSW}
For all $\xi_1>0$, there exists $h_0,C,c>0$ such that for all $h<h_0$, there exists a unique $\phig_{\xi_1}^h \in H^1(h\mathbb{Z};\mathbb{R})$ symmetric, centered in $0$,  and $\zeta_1 \in \mathbb{R}$, such that 
\begin{itemize}
\item $e^{i\zeta_1 t}\phig_{\xi_1}^h$ is a solution of DNLS,
\item $|\zeta_1 - \xi_1|+\| \phig^{h}_{\xi_1} - \psi_{(\xi_1,0) \ | h\mathbb{Z}} \|_{H^1(h\mathbb{Z})} \leq C h^2$,
\item If $\ug$ is a solution of DNLS such that $\ug(0)$ is symmetric, centered in $0$, and $\|\ug(0) - \phig_{\xi_1}^h \|_{H^1(h\mathbb{Z})}<c$  then for all $t\in \mathbb{R}$, there exists $\gamma\in \mathbb{R}$ such that
\[ \| \ug(t) - e^{i\gamma} \phig_{\xi_1}^h \|_{H^1(h\mathbb{Z})} \leq C \|\ug(0) - \phig_{\xi_1}^h \|_{H^1(h\mathbb{Z})} . \]
\end{itemize}
\end{theo}

Note that the same theorem holds, for the \it off site \rm standing waves. We just need to write "symmetric, centered in $\frac{h}2$" instead of "symmetric, centered in $0$" and "$\psi_{(\xi_1,0)}(.-\frac{h}2)_{| h\mathbb{Z}} $" instead of "$\psi_{(\xi_1,0) \ | h\mathbb{Z}} $".

Usually, it is enough to prove existence and orbital stability of NLS standing waves to get some orbitally stable traveling wave. Indeed, NLS is invariant by \it Galilean transformation \rm, defined by
\[  u(t,x) \mapsto e^{i \frac{v}{2} (x-v t)+ i \left( \frac{v}{2}\right)^2 t } u(t,x- v t).\]
However, it seems there is no such transformation for DNLS. So we cannot apply the same strategy.

The second reason why existence of orbitally stable traveling waves for DNLS seems very uncertain is more experimental. If we assume that DNLS admits a moving traveling wave (i.e. $\xi_2\neq 0$) that is orbitally stable and looking  like a continuous traveling wave, $\psi_\xi$, then the solution of DNLS generated by the discretization of $\psi_\xi$ on $h\mathbb{Z}$, should look like $\psi_\xi$ for all times, up to an advection and a gauge transform. But there are some reasonable numerical simulations for which it is not what is observed (see \cite{MR3460749}). In fact, the speed of this solution seems going to $0$ as $t$ goes to infinity. In the literature, this phenomenon is usually called Peierls-Nabarro barrier (see  \cite{MR3460749}, \cite{MR2742565} and \cite{MR2365587}). A rigorous proof of this phenomenon seems to be an open problem. However, it is really difficult to observe when $h$ is small enough (in fact, stability for exponentially long times is expected, see  \cite{MR2365587}).

Before stating our main results, let us first formulate an easy corollary of them, showing that there exists quasi-traveling waves to DNLS close to the continuous limit, for times of order $\mathcal{O}(h^{-2})$, preventing the phenomena described above to appear before this time scale.

\begin{theo}
\label{thm_result_easy}
 For all $\varepsilon >0$ and for all $\xi\in \mathbb{R}^2$ such that $\xi_1 > \left(  \frac{\xi_2}2 \right)^2$, there exist $h_0,C,T_0>0$ such that 
 $$
T_0 = \infty \quad \mbox{when} \quad \xi_2 = 0\qquad\mbox{and}\qquad T_0 \to \infty \quad\mbox{when the speed}\quad \xi_2 \to 0, 
 $$
 and such that 
 if $h<h_0$, $\xO_0,\gamma_0\in \mathbb{R}$ and $\ug$ is the solution of DNLS such that
\[ \forall g \in h\mathbb{Z}, \ \quad   \ug_g(0) = e^{i\gamma_0} \psi_{\xi}(g-\xO_0) , \]
then, there exist $\gamma,\xO\in C^1(\mathbb{R})$ satisfying $\gamma(0) = \gamma_0$ and $\xO(0) = \xO_0$ such that, for all $t\geq 0$, 
\[ \forall t \leq T_0 h^{-2+\varepsilon}, \qquad \sup_{g \in h \mathbb{Z}} \left| \ug_g(t) - e^{i\gamma(t)} \psi_{\xi}(g-\xO(t)) \right| \leq C h^2\]
and 
\[ \forall t \leq T_0 h^{-2+\varepsilon}, \qquad |\dot \gamma (t) - \xi_1 | + |\dot \xO (t) - \xi_2 | \leq C h^2.\]
\end{theo}

The proof of Theorem \ref{thm_result_easy} is a straightforward application of Theorem \ref{thm_stab_reg} (or Theorem \ref{Thm_DTW} if $\xi_2=0$). It would be possible to write the same result with the discrete $H^1$ norm instead of the $L^{\infty}$ norm. 

 To obtain this result, the strategy is to construct a  function close to the continuous solitary wave $\psi_\xi$ for given parameters $\xi = (\xi_1,\xi_2)$, which define solitary waves of a modified version of DNLS essentially defined by removing the aliasing terms. This typically gives bound for time scales of order $\mathcal{O}(h^{-1})$ for orbital stability in $H^1(h\mathbb{Z})$. Moreover as the aliasing terms are small for regular functions, we can combine this analysis with a result of control of discrete Sobolev norms of DNLS to reach the time scale $\mathcal{O}(h^{-2})$. We give now the details of our results. 
The first one is a result of existence and stability in $H^1$ of discrete traveling waves for times of order $h^{-1}$.

\begin{theo}
\label{Thm_DTW}
Let $\Omega$ be a relatively compact open subset of $\left\{ \xi \in \mathbb{R}^2 \ | \ \xi_1>\left( \frac{\xi_2}2 \right)^2 \right\}$.\\
There exist $h_0,\kappa, r,\ell>0$ such that for all $h<h_0$, for all $\xi \in \Omega$, there exists $\eta_\xi^h \in H^{\infty}( \mathbb{R})$ with
\begin{equation}
\label{err_cons_thm}
  \| \eta^{h}_{\xi} - \psi_{\xi} \|_{H^1(\mathbb{R})} \leq \kappa h^2, 
\end{equation}
satisfying the following property.

If $\vg \in H^1(h\mathbb{Z})$ is an approximation of $\eta_\xi^h$ up to a gauge transform or an advection, i.e.
\[ \exists \gamma_0,\xO_0 \in \mathbb{R}, \  \quad \| \vg - (e^{i\gamma_0}\eta_\xi^h(\cdot-\xO_0))_{| h\mathbb{Z}} \|_{H^1(h\mathbb{Z})} \leq r, \]
then there exist $\gamma,\xO \in C^1(\mathbb{R})$ with $\gamma(0) = \gamma_0$ and $\xO(0) = \xO_0$ such that if $T>0$ and $\ug$, the solution of DNLS with $\ug(0)=\vg$, satisfy
\begin{equation}
\label{boot}
  \forall t \in (0,T), \quad  \ \delta(t) := \| \ug(t) - (e^{i\gamma(t)}\eta_\xi^h(\cdot-\xO(t)))_{| h\mathbb{Z}} \|_{H^1(h\mathbb{Z})} \leq r,
\end{equation}
then we have for all $t\in (0,T),$
\begin{equation}
\label{modulation_control}
 | \dot{ \gamma }(t) -\xi_1 | +  | \dot{ \xO }(t) -\xi_2 | \leq \kappa \ (\delta(0)+\delta(t) + e^{-\frac{\ell}h }),
\end{equation}
and 
\begin{equation}
\label{stab_H1}
\delta(t) \leq \kappa \ e^{\frac{h |\xi_2| t}{\ell^2}} ( \delta(0) + e^{-\frac{\ell}h } ).
\end{equation}
\end{theo}

The functions $\eta_\xi^h$ are constructed in the third section and estimates  \eqref{modulation_control} and \eqref{stab_H1} are proven in the fourth section. Now, we discuss this result. We focus on inequalities  \eqref{modulation_control} and \eqref{stab_H1}.
\begin{itemize}
\item If we remove the exponential terms, it is a result stronger than the classical inequality of orbital stability (see Theorem \ref{Thm_COS}) as it includes a result of modulation.
\item The exponential terms "$e^{-\frac{\ell}h }$" means that any discretization of $\eta_\xi^h$ is not exactly a traveling wave of DNLS. 
\item The time dependent exponential term means that the estimate of stability holds while $t |\xi_2|$ is smaller than $h^{-1}$. In particular, if we focus on standing waves (i.e. $\xi_2 = 0$), we get an estimate of stability for all times. Since our perturbation does not need to be symmetric, it is an extension of the previous results (see Theorem \ref{Thm_DSW}).
\item If $\ug(0)$ is a discretization of $\eta_\xi^h$ (i.e. if $\delta(0)=0$) then the estimate of stability holds longer. Indeed, while $t |\xi_2|$ is smaller than $\frac{\ell^3}{h^2}$ (up to a multiplicative constant) , then the bootstrap \eqref{boot} condition is satisfied. In particular, we deduce of the second inequality  that at the end of this time, $\ug$ has crossed the distance $\frac{\ell^3}{h^2}$ (up to a multiplicative constant), still looking like $\eta_\xi^h$.
\end{itemize}

Now, we discuss some consequences and applications of the proof of Theorem \eqref{Thm_DTW}. These extensions are linked to the two relevant exponents for $h$ in this theorem.

First, there is a control of $\eta^{h}_{\xi} - \psi_{\xi}$ by $\mathcal{O}(h^2)$ (see \eqref{err_cons_thm}). This error is a consistency error. It is due to the approximation of the second derivative by a finite difference formula of order $2$. Such an estimate depends on the finite difference operator used to approximate second derivative in space. For example, if we consider the generalization of DNLS \eqref{def_DNLS} called Discret Self-Trapping equation (DST, see \cite{MR805707})
\begin{equation}
\label{def_DST}
\forall g\in h\mathbb{Z},\quad  \ i\partial_t \ug_g = \frac{1}{h^2} \sum_{k \in \mathbb{Z}} a_k \ug_{g-kh}  + |\ug_g|^2 \ug_g,
\end{equation}
where $(a_k)_{k\in \mathbb{Z}}\in L^{1}(\mathbb{Z};\mathbb{R})$ is a symmetric sequence (i.e. $a_k=a_{-k}$ for all $k$), consistent of order $2n$, $n\in \mathbb{N}^*$,
\begin{equation}
\label{def_cons}
 \forall u \in H^{\infty}(\mathbb{R}), \quad \ \frac1{h^2}\sum_{k \in \mathbb{Z}} a_k u(hk) \unupoop{=}{h\to 0}{} \partial_x^2 u(0) + \mathcal{O}(h^{2n+2}), 
\end{equation}
and satisfying the estimate of stability
\begin{equation}
\label{def_stab}
 \exists \alpha>0,\quad \forall \omega \in (0,\pi), \quad  \ -\sum_{k \in \mathbb{Z}} a_k \cos(k\omega) \geq \alpha \omega^2   
\end{equation}
then Theorem \eqref{Thm_DTW} holds for DST and we can replace \eqref{err_cons_thm} by $\| \eta^{h}_{\xi} - \psi_{\xi} \|_{H^1(\mathbb{R})} \leq \kappa h^{2n}$. In particular, this extension includes usual pseudo spectral method and the usual high order discrete second derivatives (see \cite{finites_diff} for details about these formulas) whose non-zero terms are given by
\[ a_{\pm k} =  \frac{2(-1)^{k+1}}{k^2} \frac{C_{2n}^{n-k}}{C_{2n}^{n}}  ,\quad  \textrm{ if } \quad 0<k<n, \quad  \textrm{ and } \quad a_0 = - 2\sum_{j=1}^n \frac1{j^2}.   \]


Second, there is the right exponential term $e^{\frac{h |\xi_2| t}{\ell^2}}$ giving the stability estimates for times of order $h^{-1}$. As the error terms come mainly from aliasing effects, the control of stability for times larger than $\frac1h$ essentially relies on a control of higher Sobolev norms for long times uniformly with respect to $h$. 
More precisely, we define the discrete homogeneous Sobolev norm $ \| \cdot \|_{\dot H^n(h\mathbb{Z})}$ by
\begin{equation}
\label{def_high_Sobolev}
\| \ug \|_{\dot H^n(h\mathbb{Z})}^2 = \langle (-\Delta_h)^n \ug , \ug \rangle_{L^2(h\mathbb{Z})}, \quad \textrm{ with } \quad  (\Delta_h \ug)_g =  \frac{\ug_{g+h}  - 2 \ug_g + \ug_{g-h}}{h^2},
\end{equation}
and the Sobolev norm by
\[  \| \ug \|_{H^n(h\mathbb{Z})}^2 = \sum_{k =0}^n \| \ug \|_{\dot H^k(h\mathbb{Z})}^2. \]
Then we have the following version of Theorem \ref{Thm_DTW} (see Remark \ref{rem_if_reg} for its proof).
\begin{theo} In Theorem \ref{Thm_DTW}, the inequality \eqref{stab_H1} can be replaced by
\label{thm_stab_Hn}
\begin{equation}
\label{stab_Hn}
   \forall n\in \mathbb{N}^*, \ \delta(t) \leq \kappa \ \left(  \delta(0) + e^{-\frac{\ell}h}+\sqrt{t|\xi_2|} h^{n - \frac12} \sup_{0<s<t} \| \ug(s) \|_{\dot{H}^n(h\mathbb{Z})} \right).
\end{equation}
\end{theo}

 With such an estimate, we see that to obtain stability over exponentially long times, it would be enough to prove a control of the growth of the homogeneous Sobolev norm of the type $C t^{\alpha}$, with $\alpha$ independent of $n$ and $h$ and $C$ independent of $h$. Note that for the continuous case, it is indeed the case for the solutions of NLS for which the $H^s$ norms are uniformly bounded in times by using integrability arguments (see for example \cite{MR2996998}). Note that such bounds hold for linear Schr\"odinger equation with a smooth potential in $t$ and $x$ (see \cite{MR1753490}).


For DNLS, it is possible to obtain polynomial control of the growth of Sobolev norms by using the \it higher modified energy \rm    
 method. The following result was obtained in \cite{growth_sobolev}  by the first author: 
\begin{theo}[Growth of discrete Sobolev norms, see \cite{growth_sobolev}] 
\label{thm_growth}
  For all $n\in \mathbb{N}^*$, there exists $C>0$, such that for all $h>0$,  if $\ug$ is a solution of DNLS  then for all $t\in \mathbb{R}$
  \begin{equation}
  \label{ineq_growth}
 \| \ug(t) \|_{\dot{H}^{n}(h\mathbb{Z})} \leq C \left[ \| \ug(0) \|_{\dot{H}^{n}(h\mathbb{Z})} +   M_{\ug(0)}^{\frac{2n+1}3}   +  |t|^{\frac{n-1}2} M_{\ug(0)}^{\frac{4n-1}3}  \right] ,  
\end{equation}
where
\[  M_{\ug(0)} = \|\ug(0) \|_{\dot{H}^1(h\mathbb{Z})} +\|\ug(0) \|_{L^2(h\mathbb{Z})}^3  .\]
\end{theo}
The exponents of the $\ug(0)$ norms are natural and correspond to an homogeneous estimate preserved by scalings in $h$. 
As a corollary of Theorem \eqref{thm_stab_Hn} and Theorem \ref{thm_growth}, we get an extension of Theorem \ref{thm_result_easy} for smooth perturbations of $\eta_\xi^h$. It is a result of stability for times of order $h^{-2}$ for such perturbations.
\begin{theo}
\label{thm_stab_reg} Let $\Omega$ be a relatively compact open subset of $\left\{ \xi \in \mathbb{R}^2 \ | \ \xi_1>\left( \frac{\xi_2}2 \right)^2 \right\}$ and $h_0,\kappa, r,\ell>0$ be the constants given in Theorem \ref{Thm_DTW}.

For all $\varepsilon,s>0$, there exists $n\in \mathbb{N}^*$ such that for all $\rho>0$, there exist $C,T_0>0$ with
 \begin{equation}
 \label{rossini}
T_0 = \infty \quad \mbox{when} \quad \xi_2 = 0\qquad\mbox{and}\qquad T_0 \to \infty \quad\mbox{when the speed}\quad \xi_2 \to 0, 
 \end{equation}
 and $h_1\in (0,h_0)$, such that for all $h<h_1$, $\xi\in \Omega$ and for all $v \in H^n(\mathbb{R})$, if 
\[ \| v \|_{\dot{H}^n(\mathbb{R})} \leq \rho \quad \textrm{ and } \quad \|  \psi_\xi - v  \|_{H^1(\mathbb{R})} \leq \frac{r}{2(1+\kappa)}\]
then any solution $\ug$ of DNLS such that 
\[   \exists \xO_0,\gamma_0\in \mathbb{R},\quad  \forall g\in h\mathbb{Z},\quad  \ \ug_g(0) = e^{i\gamma_0}v(g - \xO_0)\]
satisfies, for all $t\geq 0$ such that ,
\begin{equation}
\label{eq_delta_hs}
\forall\, t \leq T_0 h^{-2+\varepsilon},\qquad 
  \| \ug(t) - (e^{i\gamma(t)}\eta_\xi^h(\cdot-\xO(t)))_{| h\mathbb{Z}} \|_{H^1(h\mathbb{Z})} \leq C \left( \|  \eta_\xi^h - v \|_{H^1(\mathbb{R})}  + h^{s} \right) 
\end{equation}
where $\gamma,\xO \in C^1(\mathbb{R})$ satisfy $\gamma(0) = \gamma_0$, $\xO(0) = \xO_0$ and
\begin{equation}
\label{eq_mod_hs}
\forall\, t \leq T_0 h^{-2+\varepsilon},\qquad 
 | \dot{ \gamma }(t) -\xi_1 | +  | \dot{ \xO }(t) -\xi_2 | \leq C \left( \|  \eta_\xi^h - v  \|_{H^1(\mathbb{R})}  + h^{s} \right).
\end{equation}
\end{theo}
This Theorem is proven in Appendix (see Section \ref{sec_app_proof}). Note that if we can prove a control on the growth of high Sobolev norms by $\mathcal{O}(t^{\alpha (n-1)})$ with $\alpha<\frac12$, then we would adapt Theorem \ref{thm_stab_reg} to reach a stability time of order $h^{-\alpha+\varepsilon}$.

\subsection{Notations} Sometimes some notations could be ambiguous, so in this subsection we clarify them.
\begin{itemize}

\item In all this paper, we consider $\mathbb{C}$ as an $\mathbb{R}$ Euclidian space of dimension $2$ equipped with the scalar product "$\cdot$" defined by
\[ \forall z_1,z_2\in \mathbb{C}, \ z_1\cdot z_2 = \Re( z_1 \overline{z_2}) = \Re z_1\Re z_2 + \Im z_1 \Im z_2. \]
Consequently, $L^2(\mathbb{R};\mathbb{C})$ scalar product is defined by
\[ \forall u_1,u_2 \in L^2(\mathbb{R};\mathbb{C}), \ \langle u_1 , u_2 \rangle_{L^2(\mathbb{R})} = \int u_1(x)\cdot u_2(x) \dx  .\]
In particular, we consider all the Fr\'echet differentials as $\mathbb{R}$ linear applications. 
\item If $u: \mathbb{R} \to \mathbb{C}$ is a real function and $h>0$, we define the \it discrete seconde derivative \rm of $u$ by
\[ \forall x \in \mathbb{R}, \ \Delta_h u (x) = \frac{ u(x+h) + u(x-h) - 2 u(x) }{h^2}. \]

\item We define the cardinal sine function on $\mathbb{R}$ by $\sinc(x) := \dfrac{\sin(x)}x$.

\item As usual when we consider second derivative, we identify the continuous bilinear forms with the operators from the space to its topological dual space. More precisely, if $E$ is a normed vector space and $b$ is a continuous bilinear form on $E$, we identify $b$ with the operator $\widetilde{b}:E\to E'$ defined by $b(x,y) = \widetilde{b}(x,y)$, $x,y\in E$. Consequently, it makes sense to try to invert $b$.

\item If $M\in M_n(\mathbb{R})$ is a square matrix of length $n$ then $\| M\|_{p}$ is the matrix norm of $M$ associated to the $\ell^p$ norm on $\mathbb{R}^n$. Similarly, if $\xi \in \mathbb{R}^2$, $|\xi| := \sqrt{\xi_1^2 + \xi_2^2}$ is the $\ell^2$ norm of $\xi$.

\item If $\mathscr{E}$ is a set then $\mathbb{1}_{\mathscr{E}}$ is the characteristic function of $\mathscr{E}$.

\end{itemize}

\subsection*{Acknoledgements} The authors are glad to thank Dario Bambusi, Beno\^it Gr\'ebert and Alberto Maspero for their helpful comments and discussions during the preparation of this work. 

\section{Aliasing generating inhomogeneity}

\label{Sec_2}
In this section, we explain why DNLS can be interpreted as an inhomogeneous equation on $\mathbb{R}$ and why we cannot apply directly the  \it energy-momentum \rm method to get stable traveling waves. This section is also an introduction to most of the tools used in the this paper.

The \it energy-momentum method \rm is a way to construct orbitally stable equilibria of a Hamiltonian system, relatively to a Lie group action. It has been used by Weinstein in \cite{MR820338} to prove the orbital stability of the traveling waves of NLS. Then it has been developed, in the general context of Hamiltonian systems by Grillakis, Shatah, Strauss in \cite{MR901236},\cite{MR1081647}. A clear and rigorous presentation of the method and its formalism in a general setting is given in the paper \cite{MR3443560} of De Bi\`evre, Genoud, and Rota Nodari.

A crucial part of this method is based on Noether theorem, requiring to identify invariant Lie group actions with Hamiltonian flows.  
For DNLS, the Lie group actions are defined by gauge transform $\ug \mapsto e^{i \gamma} \ug$ and discrete advection $\ug \mapsto (\ug_{g+a})_{g \in h \mathbb{Z}}$. The gauge transform is clearly the flow of the Hamiltonian $\| \ug \|_{L^2(h\mathbb{Z})}^2$ but the discrete advection is only defined for a countable set of values $a \in h\mathbb{Z}$ and cannot naturally be associated with a Hamiltonian.

First, we need to extend the advection for any values $a \in \mathbb{R}$ and then try to identify this extension with the flow of an Hamiltonian.
Then we are going to see that the Hamiltonian of DNLS (see \eqref{def_HDNLS}) is not invariant by this advection, and that the error is driven by aliasing terms.

\subsection{Shannon's advection}

There are   natural ways to define an advection, denoted by  $\tau_a$,  on the grid $h\mathbb{Z}$.  For a given interpolation operator $\mathcal{I}_h : L^2(h\mathbb{Z}) \to L^2(\mathbb{R})$ we can carry the advection on $\mathbb{R}$ to the grid $h\mathbb{Z}$ by making the following diagram  commute 
\begin{equation}
\label{def_Shannon_advection}
\xymatrixcolsep{5pc} \xymatrix{
L^2(h\mathbb{Z}) \ar[d]_{   \mathcal{I}_h  } \ar[r]^{ \tau_a } & L^2(h\mathbb{Z}) \ar[d]_{ \mathcal{I}_h }\\
L^2(\mathbb{R}) \ar[r]^{u \mapsto  u(\, \cdot\, -a)}        & L^2(\mathbb{R)} }
\end{equation}
In general, this construction does not work, as the advection of an interpolation is not necessary an interpolation (see, for example with a finite element interpolation). However, there exists a classical interpolation called \it Shannon interpolation \rm for which this construction can be applied.
Let us define the \it discrete Fourier transform \rm $\mathcal{F}_h$ and Fourier Plancherel transform $\mathscr{F}$
\begin{equation}
\label{def_discrete_FT}
\mathcal{F}_h : \left\{  \begin{array}{llll} L^2(h\mathbb{Z}) & \to & L^2(\mathbb{R} / \frac{2\pi}h \mathbb{Z}) \\
																\ug & \mapsto &\displaystyle \omega \mapsto h\sum_{g \in h\mathbb{Z}} \ug_g e^{ i g \omega}
\end{array} \right. \quad\mbox{and}\quad
\qquad \mathscr{F} : \left\{  \begin{array}{llll} L^2(\mathbb{R}) & \to & L^2(\mathbb{R}) \\
																u & \mapsto & \displaystyle \omega\mapsto \int_{\mathbb{R}} u(x) e^{i x \omega} \dx
\end{array}  \right. 
\end{equation}
where the last integral is defined by extending the operator defined on $L^1(\mathbb{R})\cap L^2(\mathbb{R})$. We also use the notation $\widehat{u} = \mathscr{F} u$. 
The \it Shannon interpolation \rm, denoted  by $\mathcal{I}_h$, is defined through the following diagram
\begin{equation}
\label{def_Shannon_interpolation}
\xymatrixcolsep{5pc} \xymatrix{  L^2(h\mathbb{Z})  \ar[r]^{ \mathcal{F}_h } \ar@/_1pc/[rrr]_{\mathcal{I}_h} & L^2(\mathbb{R} / \frac{2\pi}h \mathbb{Z})  \ar[r]^{ u \mapsto \mathbb{1}_{(-\frac{\pi}h,\frac{\pi}h)} u }  & 
 L^2(\mathbb{R})  \ar[r]^{ \mathscr{F}^{-1} } &  L^2(\mathbb{R}) } .
\end{equation}

With this construction, this interpolation clearly enjoys some useful properties.
\begin{propo} $\mathcal{I}_h$ is an isometry between $L^2(h\mathbb{Z})$ and its image in $L^2(\mathbb{R})$. This image is denoted $BL^2_h$. It is the 
subspace of $L^2(\mathbb{R})$ whose Fourier transform support is a subset of $[-\frac{\pi}h,\frac{\pi}h]$, i.e.
\[ BL^2_h = \{ u \in L^2(\mathbb{R}) \ | \ {\rm Supp } \ \widehat{u} \subset [-\frac{\pi}h,\frac{\pi}h]\}. \]
Moreover, the \it Shannon advection \rm $\tau_a$ is well defined through \eqref{def_Shannon_advection}. \end{propo}
\begin{proof} We just need to verify that the advection of a Shannon interpolation is an interpolation. So let $u\in BL^2_h$. Since we have
\[ \forall \omega \in \mathbb{R}, \ \quad  \widehat{ u(\cdot-a)}(\omega) = e^{-i\omega a} \widehat{u}(\omega), \]
it is clear that ${\rm Supp } \ \widehat{u(\cdot-a)} = {\rm Supp } \ \widehat{u}$. Consequently, we have proven that $u(\cdot-a) \in BL^2_h$.
\end{proof}

Since Fourier transform support of Shannon interpolations is bounded, $BL_2^h$ functions are very regular functions (they are entire function). Consequently, when we deal with $BL_2^h$ functions we will not justify the algebraic calculations.
 
We now check that this advection  is generated by a Hamiltonian flow. Introducing some formalism, since Shannon interpolation is a $\mathbb{C}$ linear isometry, we prove in the following Lemma that it is a symplectomorphism between $\left( L^2(h\mathbb{Z};\mathbb{C}), \langle i. , . \rangle_{L^2(h\mathbb{Z};\mathbb{C})} \right)$ and $\left( BL^2_h, \langle i. , . \rangle_{L^2(\mathbb{R};\mathbb{C})} \right)$ preserving the Hamiltonian structure. \begin{lem} 
\label{lem_symplecto}
Let $I$ be an open subset of $\mathbb{R}$, $\ug\in C^1(I; L^2(h\mathbb{Z};\mathbb{C}))$ and $H \in C^1(L^2(h\mathbb{Z};\mathbb{C});\mathbb{R})$. Defining $u = \mathcal{I}_h \ug$, the following properties are equivalents
\begin{equation}
\label{stat_1}
\forall t\in I,\quad  \forall \vg \in L^2(h\mathbb{Z};\mathbb{C}), \quad \ \langle i\partial_t \ug(t) , \vg \rangle_{L^2(h\mathbb{Z})} = \diff H(\ug(t))(\vg),
\end{equation}
and
\begin{equation}
\label{stat_2}
 \forall t\in I, \quad \forall v \in BL^2_h,\quad  \ \langle i\partial_t u(t) , v \rangle_{L^2(R)} = \diff (H\circ I_h^{-1})(u(t))(v).  
\end{equation}
\end{lem}
\begin{proof}
Assume \eqref{stat_1} and $v\in BL^2_h$. Since $I_h$ is bijective, there exists $\vg \in L^2(h\mathbb{Z};\mathbb{C})$ such that 
$v = \mathcal{I}_h \vg$. So we have
\[ \diff (H\circ I_h^{-1})(u(t))(v) = \diff (H\circ I_h^{-1})(u(t))(\mathcal{I}_h \vg) =  \diff H(\ug (t))( \vg) = \langle i\partial_t \ug(t) , \vg \rangle_{L^2(h\mathbb{Z})}.  \]
However, we have
\[ \langle i\partial_t u(t) , v \rangle_{L^2(R)} = \langle \mathcal{I}_h^* i \mathcal{I}_h\partial_t \ug(t) , \vg \rangle_{L^2(h\mathbb{Z})},\]
where $\mathcal{I}_h^*$ is the adjoint operator of $\mathcal{I}_h$. But $\mathcal{I}_h$ is $\mathbb{C}$ linear so we have
\[  i \mathcal{I}_h\partial_t \ug(t) =  \mathcal{I}_h i\partial_t \ug(t).\]
Furthermore, it is an isometry so we have $\mathcal{I}_h^* = \mathcal{I}_h^{-1}.$ Consequently, we get
\[ \langle i\partial_t u(t) , v \rangle_{L^2(R)} =  \langle i\partial_t \ug(t) , \vg \rangle_{L^2(h\mathbb{Z})}.\]
So we have proven \eqref{stat_2}.
Conversely, we can prove that \eqref{stat_1} is a consequence of \eqref{stat_2} using the same equalities.
\end{proof}

Applying Lemma \ref{lem_symplecto} to identify Shannon advection with a Hamiltonian flow, we just need to identify the canonical advection on $BL^2_h$.
\begin{lem}
\label{lem_momentum} Let $\mathcal{M}:BL^2_h \to \mathbb{R}$ be the momentum defined by
\[ \forall u \in BL^2_h, \quad  \ \mathcal{M}(u) =  \langle i\partial_x u,u \rangle_{L^2(\mathbb{R})}.\]
If $u\in C^1(\mathbb{R};BL^2_h )$ then the following properties are equivalent
\begin{equation}
\label{stat_3}
\forall t\in \mathbb{R},\quad  \ u(t,x) = u(0,x+2t),
\end{equation}
and
\begin{equation}
\label{stat_4}
 \forall t\in \mathbb{R}, \quad \forall v \in BL^2_h, \quad \ \langle i\partial_t u(t) , v \rangle_{L^2(R)} = \diff \mathcal{M}(u(t))(v).  
\end{equation}
\end{lem}
\begin{proof}
Assume \eqref{stat_4} and let $t\in \mathbb{R}$, $v\in BL^2_h$. We have
\[  \langle \partial_t u(t) , v \rangle_{L^2(R)} = \langle i\partial_t u(t) , iv \rangle_{L^2(R)} = \diff \mathcal{M}(u(t))(iv) = 2 \langle i\partial_x u(t) , iv \rangle_{L^2(R)} =  2 \langle \partial_x u(t) , v \rangle_{L^2(R)}.\]
So since $(BL^2_h, \| \cdot \|_{L^2(\mathbb{R})})$ is a Hilbert space, we have
\[ \forall t,x\in \mathbb{R}, \quad  \ \partial_t u(t,x) = 2\partial_x u(t,x) .\]
Consequently, we have $u(t,x) = u(0,x+2t)$. The converse is obvious.
\end{proof}

Applying Lemma \ref{lem_symplecto} and Lemma \ref{lem_momentum}, we deduce that Shannon's advection the flow of the Hamiltonian $-\frac12\mathcal{M}\circ I_h^{-1}$.

\subsection{The aliasing error}

In this subsection, we show that the DNLS Hamiltonian is not invariant by Shannon's advection. We recall some classical properties of Shannon interpolation, see for example \cite{MR3156669} for more details. 

\begin{propo} 
\label{prop_seq_is_rest}
If $\ug \in L^2(h\mathbb{Z})$ then $\mathcal{I}_h \ug_{| \ h \mathbb{Z}} = \ug$. \end{propo}
This proposition is just a corollary of the following decomposition, where the series converges in $L^{\infty}(\mathbb{R}) \cap L^2(\mathbb{R})$,
\[ \forall x \in \mathbb{R},\quad  \ \mathcal{I}_h \ug(x) =   \sum_{g\in h\mathbb{Z}} \ug_g \sinc(\pi\ \frac{x-g}h).\]
\begin{corol}
\label{corol_strong}
 The Shannon interpolation of $\ug$ is the only function in $ L^2(\mathbb{R})$ with Fourier transform support included in  $[-\frac{\pi}{h},\frac{\pi}h]$ and whose values on $h\mathbb{Z}$ are those of $\ug$.
\end{corol}

Now, we detail a classical property of Shannon interpolation that is crucial in this paper.
\begin{propo}
\label{prop_per}
 If $u\in H^1(\mathbb{R})$ then $\ug :=u_{| \ h\mathbb{Z}} \in L^2(h\mathbb{Z})$ and for all $\omega \in (-\frac{\pi}h,\frac{\pi}h)$ we have
 \begin{equation}
 \label{aliasing_series}
\widehat{\mathcal{I}_h \ug}(\omega) = \sum_{k \in \mathbb{Z}} \widehat{u}(\omega + \frac{2\pi}h k).  
\end{equation}
 \end{propo}
\begin{proof} First observe that the series \eqref{aliasing_series} converges in $L^2(-\frac{\pi}h,\frac{\pi}h)$. Indeed, using Cauchy Schwarz inequality, we have
\begin{align*}
 \sum_{k \in \mathbb{Z}\setminus \{0\}} \| \widehat{u}(\omega + \frac{2\pi}h k) \|_{L^2(-\frac{\pi}h,\frac{\pi}h)} &\leq \sum_{k \in \mathbb{Z} \setminus \{0\}}  \| \widehat{\partial_x u}(\omega + \frac{2\pi}h k) \|_{L^2(-\frac{\pi}h,\frac{\pi}h)}  \frac{h}{ |2k-1|\pi} \\
 &\leq  \sqrt{2\pi} \| \partial_x u \|_{L^2(\mathbb{R})}  \sqrt{ \sum_{k \in \mathbb{Z}\setminus \{0\}}  \frac{h^2}{ (2k-1)^2\pi^2} } .
\end{align*}
Now define $v\in BL^2_h$ through its Fourier transform
\[ \widehat{v}(\omega) = \mathbb{1}_{(-\frac{\pi}h,\frac{\pi}h)}  \sum_{k \in \mathbb{Z}} \widehat{u}(\omega + \frac{2\pi}h k).\]
If we prove that the values of $v$ on $h\mathbb{Z}$ are the same as the values of $u$ then we conclude the proof with Corollary \ref{corol_strong}.  Using   inverse Fourier transform formula and continuity of Fourier Plancherel transform, we get for $j\in \mathbb{Z}$,
\begin{align*}
v(hj) &= \frac1{2\pi} \int_{\mathbb{R}} \widehat{v}(\omega) e^{-i\omega h j} {\rm d\omega} = \frac1{2\pi}  \sum_{k \in \mathbb{Z}}  \int_{-\frac{\pi}h- \frac{2\pi}h k}^{\frac{\pi}h-  \frac{2\pi}h k} \widehat{u}(\omega) e^{-i(\omega- \frac{2\pi}h k ) h j} {\rm d\omega} \\
		&=  \frac1{2\pi}  \sum_{k \in \mathbb{Z}}  \int_{-\frac{\pi}h- \frac{2\pi}h k}^{\frac{\pi}h-  \frac{2\pi}h k} \widehat{u}(\omega) e^{-i\omega h j} {\rm d\omega} =  \frac1{2\pi}  \int_{\mathbb{R}}  \widehat{u}(\omega) e^{-i\omega h j}  {\rm d\omega} = u(hj).
\end{align*}
\end{proof}
We now express the DNLS Hamiltonian in terms of Shannon interpolation: 
\begin{lem} For all $\ug \in L^2(h\mathbb{Z})$, let $u=\mathcal{I}_h \ug$,  then we have
\label{lem_HDNLS}
\begin{equation}
\label{lem_eq_HDNLS}
\HDNLS(\ug) = \frac12 \int_{\mathbb{R}}  \left| \frac{u(x+h) - u(x)}h \right|^2 \dx - \frac14 \int_{\mathbb{R}} \left( 1+2\cos(\frac{2\pi x}h)\right)|u(x)|^4 \dx.
\end{equation}
\end{lem}
\begin{proof}
Since the Shannon interpolation $\mathcal{I}_h$ is an isometry between $L^2(h\mathbb{Z};\mathbb{C})$ and $L^2(\mathbb{R};\mathbb{C})$, we have
\[  h\sum_{g \in h\mathbb{Z}}   \left| \frac{\ug_{g+h}- \ug_{g}}h \right|^2  = \int_{\mathbb{R}}  \left| \frac{u(x+h) - u(x)}h \right|^2 \dx.\]
Now we calculate the nonlinear part. First, we use the same argument of isometry to prove that
\begin{equation}
\label{I_HDNLS}
h\sum_{g \in h\mathbb{Z}} |{\ug_g}|^4 = \langle \ug , |\ug|^2 \ug \rangle_{L^2(h\mathbb{Z})} =  \langle u , \mathcal{I}_h(|\ug|^2 \ug) \rangle_{L^2(\mathbb{R})}. 
\end{equation}
But we deduce from Proposition \ref{prop_per} that for $\omega\in \mathbb{R}$
\[  \mathscr{F} \mathcal{I}_h(|\ug|^2 \ug)(\omega)  =  \mathbb{1}_{(-\frac{\pi}h,\frac{\pi}h)}(\omega) \sum_{k \in \mathbb{Z}} \widehat{|u|^2 u}(\omega +\frac{2\pi}h k). \]
However, since $u\in BL^2_h$, we have
\[ \supp \widehat{|u|^2 u} \subset \supp \widehat{u} + \supp \widehat{u}  + \supp \widehat{\bar{u}} \subset \left[ -\frac{3\pi}h,\frac{3\pi}{h} \right].\]
Consequently, if $k \notin  \{-1,0,1\}$ the term in the sum is zero. Furthermore, it is clear that for any $v\in L^2(\mathbb{R})$, $\gamma\in \mathbb{R}$, $ \widehat{v}(\cdot + \gamma) = \widehat{e^{i \gamma x} v}$.
So we have
\[  \mathscr{F} \mathcal{I}_h(|\ug|^2 \ug)(\omega) =  \mathbb{1}_{(-\frac{\pi}h,\frac{\pi}h)}(\omega) \mathscr{F} \left[  \left( 1+2\cos(\frac{2\pi x}h)\right) |u|^2 u \right](\omega).   \]
We conclude by plugging this relation in \eqref{I_HDNLS}. 
\end{proof}

We this Lemma \ref{lem_HDNLS}, we can observe that $\HDNLS$ is not invariant by advection. This default of invariance is due to an inhomogeneity generated by  aliasing errors.

\subsection{The flow of DNLS in the space of the Shannon interpolations}
\label{stop_discrete}
Thanks to Shannon interpolation, we identify functions defined on a grid with functions of $BL_2^h$. We will now see that it is equivalent to consider the flow of DNLS on a grid, or consider the Hamiltonian flow on $BL_2^h$ associated with the Hamiltonian 
\begin{equation}
\label{def_HDNLS_cont}
\forall u \in BL^2_h, \ H_{\rm DNLS}^h(u):=\frac12 \int_{\mathbb{R}}  \left| \frac{u(x+h) - u(x)}h \right|^2 \dx - \frac14 \int_{\mathbb{R}} \left( 1+2\cos(\frac{2\pi x}h)\right)|u(x)|^4 \dx.
\end{equation}
Applying Lemma \ref{lem_symplecto}, we obtain: 
\begin{lem} 
\label{lem_DNLS_BL2}
Let $h>0$, $\ug \in C^1(\mathbb{R}; L^2(\mathbb{R}))$ and $u = \mathcal{I}_h(\ug)$. Then $\ug$ is a solution of DNLS (see \eqref{def_DNLS}) if and only if
\[ \forall t\in \mathbb{R},\quad \forall v\in BL^2_h, \quad \ \langle i \partial_t u(t) , v \rangle_{L^2(\mathbb{R})}  = \diff H_{\rm DNLS}^h(u(t))(v).\]
\end{lem}

We conclude with the following result showing that discrete Sobolev norms are equivalent to continuous Sobolev norms on $BL_2^h$:  
\begin{lem}
\label{lem_compare_norms}
Let $\ug\in L^2(h\mathbb{Z})$ and $u=\mathcal{I}_h \ug \in BL^2_h$. Then we have
\[ \frac{2}{\pi} \| u \|_{H^1(\mathbb{R})}  \leq \| \ug \|_{H^1(h\mathbb{Z})} \leq \| u \|_{H^1(\mathbb{R})}. \]
\end{lem}
\begin{proof}
By construction, we know that  $\| \ug \|_{L^2(h\mathbb{Z})} = \| u \|_{L^2(\mathbb{R})} $. So we just need to focus on the other part of the $H^1(h\mathbb{Z})$ norm. Indeed, applying Shannon isometry and Fourier Plancherel isometry, we have
\begin{align*}
  \| \ug \|_{\dot{H^1}(h\mathbb{Z})}^2 &= \sum_{g \in h\mathbb{Z}} \left| \frac{ \ug_{g+h} - \ug_g}{h}\right|^2 = \int_{\mathbb{R}} \left| \frac{ u(x+h) - u(x)}{h}\right|^2 \dx 
 												    = \frac1{2\pi}  \int_{\mathbb{R}} \frac{4}{h^2} \sin^2 \left( \frac{\omega h}2\right) |\widehat{u}(\omega)|^2 \domega 
												   \\&  =  \frac1{2\pi}  \int_{-\frac{\pi}h}^{\frac{\pi}h} \sinc^2 \left( \frac{\omega h}2\right) |\omega \widehat{u}(\omega)|^2 \domega 
												    \in  \frac1{2\pi}\int_{-\frac{\pi}h}^{\frac{\pi}h}  |\omega \widehat{u}(\omega)|^2 \domega \ [\sinc^2(\frac{\pi}2),1]  
												    = \| \partial_x u \|_{L^2(\mathbb{R})}^2 \  \left[ \left( \frac2{\pi} \right)^2,1\right].
\end{align*}
\end{proof}
Similarly, we can prove that for high order homogeneous Sobolev norms (see \eqref{def_high_Sobolev}), we have for all $\ug \in L^2(h\mathbb{Z};\mathbb{C})$
and $u = \mathcal{I}_h(\ug)$, 
\begin{equation}
\label{eq_dis_cont_sobolev}
 \left( \frac{2}{\pi} \right)^{n}   \| u \|_{\dot{H}^n(\mathbb{R})} \leq \| \ug \|_{\dot{H}^n(h\mathbb{Z})} \leq \| u \|_{\dot{H}^n(\mathbb{R})}. 
\end{equation}

\section{Traveling waves of the homogeneous Hamiltonian}

In the previous subsection, we have seen that the Hamiltonian of DNLS is not invariant by Shannon's advection. This default of invariance is due to an inhomogeneity generated by an aliasing error (the highly oscillatory terms in \eqref{def_HDNLS_cont}), preventing a faire use of energy-momentum method to get stable traveling waves. Let us introduce the following perturbation of the DNLS Hamiltonian, obtained by  removing these  aliasing terms: 
\begin{equation}
\label{def_HDNLS/AC}
\forall u \in BL^2_h,\quad  \ H_h(u) = \frac12 \int_{\mathbb{R}}  \left| \frac{u(x+h) - u(x)}h \right|^2 \dx - \frac14 \| u\|_{L^4(\mathbb{R})}^4.
\end{equation}
This new Hamiltonian is clearly invariant by gauge and advection transform, and we will be able to apply the energy-momentum method. Moreover, for smooth function, it is very close to the DNLS Hamiltonian.

In the first subsection, we construct, with a perturbative method, critical points of Lagrange functions associated with \eqref{def_HDNLS/AC}. These critical points are the functions $\eta_\xi^h$ of Theorem \ref{Thm_DTW}. They are traveling waves for the dynamic associated to this homogeneous Hamiltonian. In the second subsection, we focus on their regularity and their orbital stability.

In all this section, we only consider speeds $\xi$ in $\Omega$, a relatively compact open subset of $\left\{ \xi \in \mathbb{R}^2 \ | \ \xi_1>\left( \frac{\xi_2}2 \right)^2 \right\}$.

\subsection{Construction of the traveling waves}
\label{sub_31}
 Let us introduce the Lagrange function $\Lag_\xi^h : BL^2_h \to \mathbb{R}$ defined  by
\begin{equation}
\label{def_LDNLS/AC}
\forall u \in BL^2_h, \quad  \ \Lag^{h}_\xi(u) = H_h(u) + \frac{\xi_1}2 \| u \|_{L^2(\mathbb{R})}^2 + \frac{\xi_2}2 \langle i \partial_x   u, u \rangle_{L^2(\mathbb{R})} .
\end{equation}
 We prove in the following lemma that traveling waves generated by $H_h$ are critical points of $\Lag_\xi^h$. 
\begin{propo}
\label{prop_what_TW_DNLSA}
Let $\xi \in \mathbb{R}^2$, $h>0$ and $u \in C^1(\mathbb{R};BL^2_h)$ be such that
\[ \forall t\in \mathbb{R},\quad \forall x\in \mathbb{R},\quad  \ u(t,x) = e^{i\xi_1 t} u(0,x-\xi_2t). \]
Then the following properties are equivalents
\begin{equation}
\label{stat5}
\forall t\in \mathbb{R},\quad  \forall v\in BL^2_h,\quad  \ \langle i\partial_t u(t),v \rangle_{L^2(\mathbb{R})} = \diff H_h(u(t))(v), 
\end{equation}
and
\begin{equation}
\label{stat6}
\diff \Lag_\xi^h(u(0)) =0. 
\end{equation}
\end{propo}
\begin{proof} By a straightforward calculation, we have, for all $t,x\in \mathbb{R}$, 
\[ \partial_t u(t,x) =  i \xi_1 u(t,x) - \xi_2 \partial_x u(t,x) .\]
Consequently, testing this relation against $v\in BL^2_h$, we get for all $t,x\in \mathbb{R}$, 
\[  \langle i\partial_t u(t),v \rangle_{L^2(\mathbb{R})} = - \diff \left( \frac{\xi_1}2 \| \cdot \|_{L^2(\mathbb{R})}^2 + \frac{\xi_2}2 \langle i \partial_x   \cdot \, , \cdot \rangle_{L^2(\mathbb{R})} \right) (u(t))(v) . \]
So \eqref{stat5} is clearly equivalent to
\begin{equation}
 \forall t\in \mathbb{R}, \  \diff \Lag_\xi^h(u(t)) =0.
\end{equation}
In particular \eqref{stat5} $\Rightarrow$  \eqref{stat6} is obvious.

Conversely, to prove  \eqref{stat6} $\Rightarrow$ \eqref{stat5}, we just need to prove that if $u_0\in BL^2_h$ is a critical point of $\Lag_\xi^h$ and $\gamma,\xO\in \mathbb{R}$ then $e^{i\gamma}u_0(.-\xO)$ is also a critical point of $\Lag_\xi^h$. 
Define $T_{\gamma,\xO} : BL^2_h \to BL^2_h$ by
\[ \forall v\in BL^2_h, \quad \  T_{\gamma,\xO} v =e^{i\gamma} v(.-\xO). \]
Since $\Lag_\xi^h$ is invariant by gauge transform and advection, we have
\[ \forall v\in BL^2_h, \quad \  \Lag_\xi^h(T_{\gamma,\xO} v) = \Lag_\xi^h(v). \]
 Calculating the derivative with respect to $v$ in $u_0$, we get
\[ \forall v\in BL^2_h, \ \quad \diff \Lag_\xi^h(T_{\gamma,\xO} u_0)(T_{\gamma,\xO} v) = \diff \Lag_\xi^h(u_0)(v)=0. \]
Since $T_{\gamma,\xO}$ is an invertible operator on $BL^2_h$ (because $T_{\gamma,\xO}^{-1} = T_{-\gamma,-\xO}$), $T_{\gamma,\xO} u_0$ is also a critical point $\Lag_\xi^h$.
\end{proof}
In the following Theorem, we construct critical points of the Lagrange functions $\Lag_\xi^h$ as perturbations of the continuous traveling waves $\psi_\xi$ embedded in $BL^2_h$. 
\begin{theo}
\label{thm_ExTW}
There exist $h_0,C,\rho,\alpha>0$ such that for all $h<h_0$ and for all $\xi\in \Omega$, there exists $\eta_\xi^h \in BL^2_h$ satisfying
\begin{enumerate}[a)]
\item $\diff \Lag_\xi^h(\eta_\xi^h)=0,$
\item $\| \eta_\xi^h - \psi_\xi \|_{H^1(\mathbb{R})} \leq C h^2$,
\item $\forall x\in \mathbb{R}, \ \overline{\eta_\xi^h}(-x) = \eta_\xi^h(x)$,
\item if $u \in BL^2_h$ is such that $\| u - \eta_\xi^h \|_{H^1(\mathbb{R})}<\rho$, $ \overline{u}(-x) = u(x)$ for all $x\in \mathbb{R}$ and $\diff \Lag_\xi^h(u)=0$ then $u= \eta_\xi^h$,
\item if  $v\in BL^2_h \cap \Span( \eta_\xi^h,i\eta_\xi^h,\partial_x \eta_\xi^h )^{\perp_{L^2}}$, then we have
\[  \diff^2 \Lag_\xi^h(\eta_\xi^h)(v,v) \geq \alpha \|v\|_{H^1(\mathbb{R})}^2.\]
Furthermore, $\xi \mapsto \eta_\xi^h$ is $C^1$ and for all $h<h_0$, for all $\xi \in \Omega$, we have
\[ \forall \zeta \in \mathbb{R}^2, \quad \ \| \diff_{\xi} \eta_\xi^h(\xi)(\zeta) - \diff_{\xi} \psi_\xi(\xi)(\zeta)  \|_{H^1(\mathbb{R})} \leq C |\zeta| h^2 .\]
\end{enumerate}
\end{theo}

The remainder of this section is devoted to the proof of this Theorem.  It is divided in three steps. 
 The idea of the proof is to apply, for each value of $\xi$, the inverse function Theorem to solve $\diff \Lag_\xi^h(u)=0$. We give an adapted version of this result, see Theorem \ref{Thm_Inv_loc}, proven in Appendix. 
Moreover, we have to pay attention to symmetries and establish estimates uniform with respect to $\xi \in \Omega$ and $h$ small enough.

\underline{Step 1: Identify the function to invert}

First, we need a point around which apply the inverse function Theorem. To do this, we consider the orthogonal projection of the continuous traveling wave $\psi_\xi$ on $BL^2_h$ (for the $L^2(\mathbb{R})$ norm) denoted by $\psi_\xi^h$. Using Fourier Plancherel transform we observe that $\psi^h_\xi$ and $\psi_\xi$ are linked by their Fourier transform through the relation
\begin{equation}
\label{ortho_proj}
 \widehat{\psi^h_\xi} = \mathbb{1}_{(-\frac{\pi}h,\frac{\pi}h)} \widehat{\psi_\xi}.
\end{equation}
Sometimes it is useful to extend this notation for $h=0$ with $\psi^0_\xi =\psi_\xi $.

Now, we have to take care about the symmetries of the problem. Indeed, since the set of the critical points of $\Lag_\xi^h$ is stable under advection and gauge transform, we expect that the differential of $\diff \Lag_\xi^h$ is not invertible in this critical point. However, there is a classical trick to avoid the problem generated by these symmetries. To explain this trick we need to introduce an operator on $BL^2_h$
\[  S_h : \left\{ \begin{array}{cccc} BL_h^2 & \to & BL_h^2 \\
												u & \mapsto & (x\mapsto\overline{u(-x)}).
\end{array}    \right. \]
This symmetry is natural for our problem because $ \Lag_\xi^h$ is invariant under its action.
\begin{lem}
\label{lem_IS}
For all $h>0$, for all $\xi\in \mathbb{R}^2$, for all $u\in BL^2_h$, we have
\[ \Lag_\xi^h(S_h(u)) = \Lag_\xi^h(u). \]
\end{lem}
\begin{proof}
It can be proven by a straightforward calculation.
\end{proof}
This operator induces a decomposition of $BL^2_h$ very well adapted to our problem
\[  BL_h^2 = {\rm Ker}(\id - S_h) \oplus  {\rm Ker}(\id + S_h).   \]
This decomposition is also a topological decomposition because these subspaces are closed for the $\| \cdot  \|_{H^1(\mathbb{R})}$ norm. In all the paper, these spaces are always implicitly equipped with this norm.

The continuous traveling waves is invariant under this symmetry. Indeed, we can verify (see \eqref{def_CTW}) that
\[ \forall x\in \mathbb{R}, \quad \ \overline{\psi_\xi(-x)}= \psi_\xi(x). \]
Consequently, we expect $\eta_\xi^h$ to be invariant under the action of $S_h$. 
The space ${\rm Ker}(\id - S_h)$ is not invariant under advection or gauge transform, so we avoid the previous difficulty. Moreover, we have the following result
\begin{lem}
\label{lem_invsym}
For all $h>0$, for all $\xi\in \mathbb{R}^2$, for all $u\in {\rm Ker}(\id - S_h)$, for all $v\in {\rm Ker}(\id + S_h)$, we have
\[ \diff \Lag_\xi^h(u)(v) =0.\]
\end{lem}
\begin{proof}
Applying Lemma \ref{lem_IS}, we get
\[  \Lag_\xi^h(u-v) = \Lag_\xi^h(u+v). \]
Then, if we compute the derivative with respect to $v\in  {\rm Ker}(\id + S_h)$, we get
\[ \diff \Lag_\xi^h(u)(v) = - \diff \Lag_\xi^h(u)(v).\]
\end{proof}
With this lemma, we see that   a critical point of $\diff \Lag^h_{\xi \ |  {\rm Ker}(\id - S_h)}$ is a critical point of $\Lag^h_{\xi}$.
Hence we will apply  the inverse function Theorem \ref{Thm_Inv_loc} in the point $\psi_\xi^h$ which is in ${\rm Ker}(\id - S_h)$ (it is a straightforward calculation), and to the function $\diff \Lag^h_{\xi \ |  {\rm Ker}(\id - S_h)}$. 


 \underline{Step 2: Invertibility of the derivative}

 Now, we want to prove that $\diff^2 \Lag^h_{\xi \ |  {\rm Ker}(\id - S_h)}(\psi_\xi^h)$ is invertible and to estimate the norm of its invert uniformly with respect to $\xi \in \Omega$ and $h$ small enough. The strategy of the proof is to establish that $\diff^2 \Lag_\xi^h( \psi_\xi^h)$ is negative in the direction of $\psi_\xi^h$ and positive in the direction $L^2$-orthogonal to $\psi_\xi^h$ in ${\rm Ker}(\id - S_h)$. Then it will be possible to conclude using a classical lemma of functional analysis (see Lemma \ref{lem_exRiesz} ).

We are going to establish most of our estimates from the continuous limit. So we need to introduce the continuous Lagrange function associated to NLS, defined on $H^1(\mathbb{R})$ by
\[ \Lag_\xi(u) = \frac1{2}\|\partial_x u\|_{L^2(\mathbb{R})}^2 -\frac14 \| u\|_{L^4(\mathbb{R})}^4 + \frac{\xi_1}2 \|u\|_{L^2(\mathbb{R})}^2 + \frac{\xi_2}2 \langle i\partial_x u,u \rangle_{L^2(\mathbb{R})}. \]
Of course, as expected, we can verify that $\psi_\xi$ is a critical point of $\Lag_\xi$. We will have to  compare precisely $\psi_\xi^h$ and $\psi_\xi$. So we need a precise control of the regularity of $\psi_\xi$.
\begin{lem} 
\label{reg_CTW}
There exist $C>0$ and $\varepsilon>0$ such that for all $\xi\in \Omega$ and all $\omega \in \mathbb{R}$
\[ |\widehat{\psi_\xi}(\omega)|\leq C e^{- \varepsilon |\omega|}. \]
\end{lem}
\begin{proof}
It is a classical result of elliptic regularity. Here we can see it directly through formula \eqref{def_CTW}. We also could prove it directly with the same ideas as in  Theorem \ref{thm_anareg} below.
\end{proof}

First, we prove, through the following lemma, that $\diff^2 \Lag_\xi^h( \psi_\xi^h)$ is negative in the direction of $\psi_\xi^h$.
\begin{lem}
\label{lem_neg}
There exist $\alpha>0$ and $h_0>0$ such that for all $h<h_0$ and all $\xi\in \Omega$ we have
\[ \diff^2 \Lag_\xi^h( \psi_\xi^h)( \psi_\xi^h, \psi_\xi^h) \leq - \alpha \| \psi_\xi^h\|_{H^1(\mathbb{R})}^2.  \]
\end{lem}
\begin{proof}
If $u\in H^1(\mathbb{R})$ we have
\[ \diff^2  \Lag_\xi(u)(u,u) = \diff  \Lag_\xi(u)(u) - 2 \| u\|_{L^4(\mathbb{R})}^4.    \]
Consequently, since $\psi_\xi$ is a critical point of $ \Lag_\xi$, we have
\[ \diff^2 \Lag_\xi( \psi_\xi)( \psi_\xi, \psi_\xi) = - 2  \| \psi_\xi \|_{L^4(\mathbb{R})}^4 .\]
However, $\xi \mapsto \| \psi_\xi \|_{L^4(\mathbb{R})}^4$ and $\xi \mapsto \| \psi_\xi \|_{H^1(\mathbb{R})}^2$ are continuous positive maps on $\overline{\Omega}$. So, there exists $\alpha>0$ such that, for all $\xi \in \Omega$,
\[ \diff^2 \Lag_\xi( \psi_\xi)( \psi_\xi, \psi_\xi) = - 2  \| \psi_\xi \|_{L^4(\mathbb{R})}^4 \leq  - \alpha  \| \psi_\xi \|_{H^1(\mathbb{R})}^2 .\]
Since $ \| \psi_\xi^h \|_{H^1(\mathbb{R})}^2 \leq  \| \psi_\xi \|_{H^1(\mathbb{R})}^2$ (see \eqref{ortho_proj}), to conclude this proof it is enough to prove that $\Lag_\xi^h( \psi_\xi^h)( \psi_\xi^h, \psi_\xi^h)$ goes to $ \diff^2 \Lag_\xi( \psi_\xi)( \psi_\xi, \psi_\xi)$ when $h$ goes to $0$, uniformly with respect to $\xi \in \Omega$.
We can write 
\begin{align}
\nonumber
 \diff^2 \Lag_\xi^h( \psi_\xi^h)( \psi_\xi^h, \psi_\xi^h) = & \diff^2 \Lag_\xi( \psi_\xi)( \psi_\xi, \psi_\xi) + \int_{\mathbb{R}}  \left| \frac{\psi_\xi^h(x+h) - \psi_\xi^h(x)}h \right|^2 - |\partial_x \psi_\xi^h|^2 \dx \\\label{liszt}	
 																			&+\diff^2 \Lag_\xi( \psi_\xi^h)( \psi_\xi^h, \psi_\xi^h) -  \diff^2 \Lag_\xi( \psi_\xi)( \psi_\xi, \psi_\xi) .
\end{align}

First, with Fourier Plancherel isometry, we control by the classical estimate of consistency, the term generated by the discretization of the second derivative
\begin{align*}
\left| \int_{\mathbb{R}}   |\partial_x \psi_\xi^h|^2 - \left| \frac{\psi_\xi^h(x+h) - \psi_\xi^h(x)}h \right|^2 \dx \right| &= \frac1{2\pi} \int_{-\frac{\pi}h}^{\frac{\pi}h} \left[ \omega^2 - \frac{4}h^2 \sin^2\left( \frac{\omega h}2 \right)  \right] |\widehat{\psi_\xi}(\omega)|^2 \domega \\
							&\leq  \frac1{2\pi} \int_{-\frac{\pi}h}^{\frac{\pi}h} \frac{ 1 -  \sinc^2\left( \frac{\omega h}2 \right)}{\omega^2} \omega^4 | \widehat{\psi_\xi}(\omega)|^2 \domega \\
							&\leq \sup_{\omega \in \mathbb{R}} \frac{ 1 -  \sinc^2\left( \frac{\omega h}2 \right)}{\omega^2} \|  \partial_x^2 \psi_\xi \|_{L^2(\mathbb{R})}^2 \\
							&= \left(  \frac{h}2 \right)^2 \sup_{\omega \in \mathbb{R}} \frac{ 1 -  \sinc^2\left( \omega \right)}{\omega^2} \|  \partial_x^2 \psi_\xi \|_{L^2(\mathbb{R})}^2.
\end{align*}
Furthermore, we deduce from Lemma \ref{reg_CTW} that $\|  \partial_x^2 \psi_\xi \|_{L^2(\mathbb{R})}^2$ can be estimated uniformly with respect to $\xi \in \Omega$.

The convergence of the second term in \eqref{liszt} is easier. 
Indeed, we deduce from Lemma \ref{reg_CTW} that $\psi_\xi^h$ goes to $\psi_\xi$ when $h$ goes to $0$, uniformly with respect to $\xi \in \Omega$.  We conclude because it is clear that the map $u\mapsto \diff^2 \Lag_\xi( u)(u,u) $ is Lipschitz on bounded subsets of $H^1(\mathbb{R})$, uniformly with respect to $\xi \in \Omega$.
\end{proof}

Now, we give the most important lemma of this proof, establishing the coercivity property of  the discrete Lagrange functions uniformly with respect to the parameters.
\begin{lem}
\label{lem_cor_psi}
There exist $\alpha>0$ and $h_0>0$ such that for all $\xi \in \Omega$ and all $h<h_0$ we have
\begin{equation}
\label{haydn} \forall v\in BL^2_h \cap \Span(i \psi_\xi^h, \partial_x \psi_\xi^h, \psi_\xi^h)^{\perp_{L^2}}, \quad \  \diff^2 \Lag_\xi^h( \psi_\xi^h)( v,v) \geq \alpha \|v\|_{H^1(\mathbb{R})}^2. \end{equation}
\end{lem}
\begin{proof}
We are going to establish this estimate by a perturbation of the continuous case. Indeed, for the continuous Lagrangian this result has been proved by Weinstein in \cite{MR820338}. There exists $\alpha>0$ such that for all $\xi\in \Omega$  
\[ \forall u\in H^1(\mathbb{R}) \cap \Span(i \psi_\xi, \partial_x \psi_\xi, \psi_\xi)^{\perp_{L^2}}, \quad \  \diff^2 \Lag_\xi( \psi_\xi)( v,v) \geq \alpha \|v\|_{H^1(\mathbb{R})}^2. \]
 Literally, it is not exactly the result of Weinstein. We explain, in Lemma \ref{est_wein} of the Appendix how to get this estimate from the original result.
Moreover, this result can be slightly extended to obtain the existence of two constants $c_1,c_2>0$ such that for all $\xi\in \Omega$,
\begin{align}
\nonumber
\mbox{if} \quad \| u - \psi_\xi\|_{H^1(\mathbb{R})}<c_1&\quad\mbox{and} \quad  \max \left( | \langle \psi_\xi, v \rangle_{L^2(\mathbb{R})} | , |\langle i\psi_\xi, v \rangle_{L^2(\mathbb{R})}| , |\langle \partial_x \psi_\xi, v \rangle_{L^2(\mathbb{R})}| \right)<c_2 \|v\|_{H^1} \\
\label{est_wein_perturb}
&\mbox{then} \quad 
\diff^2\Lag_\xi(u)(v,v)\geq \frac{\alpha}8 \|v\|_{H^1(\mathbb{R})}^2. 
\end{align}

This result is a consequence of Lemma \ref{perturb_coer} given in Appendix.  With its formalism we take $E = H^1(\mathbb{R})$, $b = \diff^2\Lag_\xi$ and $X = \Span(i \psi_\xi, \partial_x \psi_\xi, \psi_\xi)$. This last family is free because $\psi_\xi$ is not a plane wave. Consequently, the associated Gram matrix is invertible. Finally, we just need to verify that the constants $c_1$ and $c_2$ given by the lemma can be controlled uniformly with respect to $\xi \in \overline{\Omega}$. But it is a direct consequence of the estimate proven in Lemma \ref{perturb_coer} since the Gram matrix is a continuous function of $\xi \in \overline{\Omega}$.

Now, we focus on estimate \eqref{haydn} of Lemma \ref{lem_cor_psi}. Let $h_0>0$ be  small enough to get that for all $h<h_0$ and all $\xi \in \Omega$, we have $\| \psi_\xi^h - \psi_\xi \|_{H^1(\mathbb{R})}<c_1$.
 Let us fix $h<h_0$, $\xi \in \Omega$ and consider a direction $v\in BL^2_h \cap \Span( \psi_\xi^h,i\psi_\xi^h,\partial_x \psi_\xi^h)^{\perp_{L^2}}$. We decompose $v$ as 
\[ v = v_{\ell} + v_{b} \quad \textrm{ with }\quad  \widehat{v_{\ell}} = \mathbb{1}_{(-\omega_0,\omega_0)} \widehat{v} \quad \textrm{ and }\quad  \omega_0 = \frac{2\theta}{h_0} \]
where $\theta \in (0,\frac{\pi}2)$ is a constant (independent of $h,\xi$ and $h_0$) that we will determine later. 
Consider the following decomposition 
\begin{equation}
\label{the_dec}
 \diff^2 \Lag_\xi^h (\psi_\xi^h)(v,v)  = \diff^2 \Lag_\xi^h (\psi_\xi^h)(v_{\ell},v_{\ell})   + \diff^2 \Lag_\xi^h (\psi_\xi^h)(v_{b},v_{b})   + 2 \diff^2 \Lag_\xi^h (\psi_\xi^h)(v_{b},v_{\ell}).
\end{equation}
We estimate separately each one of these terms as follows:

\begin{itemize}
\item For the first one, we deduce from Lemma \ref{reg_CTW} and the constraint on $v$ that there exists $\varepsilon,C>0$ (independent of $\xi$) such that
\[ \max \left( | \langle \psi_\xi, v_{\ell} \rangle_{L^2(\mathbb{R})} | , |\langle i\psi_\xi, v_{\ell} \rangle_{L^2(\mathbb{R})}| , |\langle \partial_x \psi_\xi, v_{\ell} \rangle_{L^2(\mathbb{R})}| \right) \leq C e^{-\varepsilon \omega_0 } \|v\|_{H^1(\mathbb{R})}.\]
Consequently, if $h_0$ is small enough to get $C e^{-\varepsilon \omega_0 }<c_2$, we can apply \eqref{est_wein_perturb} to get
\[ \diff^2\Lag_\xi( \psi_{\xi}^h )(v_{\ell},v_{\ell})\geq \frac{\alpha}8 \| v_{\ell} \|_{H^1(\mathbb{R})}^2. \]
Hence we have
\begin{align*}
 \diff^2\Lag_\xi^h( \psi_{\xi}^h )(v_{\ell},v_{\ell})&\geq \frac{\alpha}8 \| v_{\ell} \|_{H^1(\mathbb{R})}^2 +  \diff^2\Lag_\xi^h( \psi_{\xi}^h )(v_{\ell},v_{\ell})-\diff^2\Lag_\xi( \psi_{\xi}^h )(v_{\ell},v_{\ell})\\
 & = \frac{\alpha}8 \| v_{\ell} \|_{H^1(\mathbb{R})}^2 + \frac1{2\pi} \int_{\mathbb{R}}   \left[ \frac4{h^2}\sin^2\left(  \frac{ \omega h}2\right)-\omega^2 \right]  |\widehat{v_{\ell}}(\omega)|^2 \domega \\
 &= \frac{\alpha}8 \| v_{\ell} \|_{H^1(\mathbb{R})}^2 - \frac1{2\pi} \int_{|\omega|<\omega_0}   \left[ 1 - \sinc^2(\frac{\omega h}2) \right]  |\omega \widehat{v_{\ell}}(\omega)|^2 \domega \\
 & \geq \frac{\alpha}8 \| v_{\ell} \|_{H^1(\mathbb{R})}^2 - \left[ 1-\sinc^2(\theta) \right] \|v_{\ell}\|_{H^1(\mathbb{R})}^2
\end{align*}

Choosing $\theta \in (0,\frac{\pi}2)$ to have $1-\sinc^2(\theta)  < \frac{ \alpha }{16}$, we get
\[  \diff^2\Lag_\xi^h( \psi_{\xi}^h )(v_{\ell},v_{\ell}) \geq \frac{\alpha}{16} \|v_{\ell}\|_{H^1(\mathbb{R})}^2   .\]
\item For the second term, we use Fourier Plancherel isometry to get
\begin{align*}
\diff^2\Lag_\xi^h( \psi_{\xi}^h )(v_{b},v_{b}) & \geq \frac1{2\pi} \int_{\mathbb{R}} \sinc^2\left(  \frac{ \omega h}2\right) \omega^2 |\widehat{v_b}(\omega)|^2 \domega - 3 \|\psi_\xi^h\|_{L^\infty(\mathbb{R})}^2 \| v_b\|^2_{L^2(\mathbb{R})} - \frac{|\xi_2|}2 \|\partial_x v_b\|_{L^2(\mathbb{R})} \|v_b\|_{L^2(\mathbb{R})}\\  
																&\geq \sinc^2(\theta) \| \partial_x v_b\|_{L^2(\mathbb{R})}^2 - 3 \|\psi_\xi^h\|_{L^\infty(\mathbb{R})}^2 \| v_b\|^2_{L^2(\mathbb{R})} - \frac{|\xi_2|}2 \|\partial_x v_b\|_{L^2(\mathbb{R})} \|v_b\|_{L^2(\mathbb{R})}.
\end{align*}
However, applying Fourier Plancherel isometry we get
\[   \|v_b\|_{L^2(\mathbb{R})}^2 = \frac1{2\pi} \int_{|\omega|>\omega_0} |\widehat{v}(\omega)|^2 \domega \leq \frac1{\omega_0^2} \frac1{2\pi} \int_{|\omega|>\omega_0} |\omega \widehat{v}(\omega)|^2 \domega = \frac1{\omega_0^2} \| \partial_x v_b \|_{L^2(\mathbb{R})}^2. \]
Consequently, we have
\begin{align*}
\diff^2\Lag_\xi^h( \psi_{\xi}^h )(v_{b},v_{b})  &\geq \left( \sinc^2(\theta) -  \frac{3 \|\psi_\xi^h\|_{L^\infty(\mathbb{R})}^2}{\omega_0^2} -  \frac{|\xi_2|}{2\omega_0}\right)  \|\partial_x v_b\|_{L^2(\mathbb{R})}^2 \\
&\geq \left( \sinc^2(\theta) -  \frac{3 \|\psi_\xi^h\|_{L^\infty(\mathbb{R})}^2}{\omega_0^2} -  \frac{|\xi_2|}{2\omega_0}\right) \frac{\omega_0^2}{1+\omega_0^2} \|  v_b\|_{H^1(\mathbb{R})}^2
\end{align*}
Since these quantities can be controlled uniformly with respect to $\xi\in \Omega$, if $h_0$ is small enough, we have for all $\xi \in \Omega$
\[  \diff^2\Lag_\xi^h( \psi_{\xi}^h )(v_{b},v_{b})  \geq \frac12\sinc^2(\theta)  \|v_b\|_{H^1(\mathbb{R})}^2.  \]
\item For the third term, since the frequency supports of $v_{\ell}$ and $v_b$ are disjoint, we get
\begin{align*}
\diff^2 \Lag_\xi^h (\psi_\xi^h)(v_{b},v_{\ell}) &= \diff^2 \frac{\| \cdot \|_{L^4}^4}4(\psi_\xi^h)(v_b,v_l)\\
																&\geq -3 \|\psi_\xi^h\|_{L^{\infty}(\mathbb{R})}^2 \|v_b\|_{L^2(\mathbb{R})}\|v_{\ell}\|_{L^2(\mathbb{R})} \\
																&\geq -3 \|\psi_\xi^h\|_{L^{\infty}(\mathbb{R})}^2 \|v_{\ell}\|_{H^1(\mathbb{R})} \frac{\|v_b\|_{H^1(\mathbb{R})}}{\sqrt{1+\omega_0^2}}\\ 
																&\geq -\frac{ 3 \|\psi_\xi^h\|_{L^{\infty}(\mathbb{R})}^2}{2 \sqrt{1+\omega_0^2}} \left( \|v_{\ell}\|_{H^1(\mathbb{R})}^2 + \|v_b\|_{H^1(\mathbb{R})}^2 \right).
\end{align*}
 Controlling this quantity uniformly with respect to $\xi\in \Omega$, we deduce that if $h_0$ is small enough then
 \[ \diff^2 \Lag_\xi^h (\psi_\xi^h)(v_{b},v_{\ell}) \geq - \frac{\beta}2  \left( \|v_{\ell}\|_{H^1(\mathbb{R})}^2 + \|v_b\|_{H^1(\mathbb{R})}^2 \right), \]
 with $\beta = \min ( \frac12 \sinc^2(\theta), \frac{\alpha}{16})$.
\end{itemize}
Applying these three estimates, we deduce that there exists an $h_0>0$ such that if $h<h_0$ and $\xi \in \Omega$ then for all $v\in BL^2_h \cap \Span( \psi_\xi^h,i\psi_\xi^h,\partial_x \psi_\xi^h)^{\perp_{L^2}}$, we have
\[ \diff^2 \Lag_\xi^h (\psi_\xi^h)(v,v) \geq \frac{\beta}2  \left( \|v_{\ell}\|_{H^1(\mathbb{R})}^2 + \|v_b\|_{H^1(\mathbb{R})}^2 \right) =  \frac{\beta}2 \| v\|_{H^1(\mathbb{R})}^2  .\]
\end{proof}

Before focusing on the invertibility of $\diff^2 \Lag^h_{\xi \ |  {\rm Ker}(\id - S_h)}(\psi_\xi^h)$, we give a small but useful lemma (particularly to control uniformly the norm of the inverse).
\begin{lem}
\label{lem_d2_unif_bound}
For all $r>0$, there exists $C>0$ such that for all $h>0$ and all $\xi \in \Omega$, we have for all $u,v,w\in BL^2_h$ with $\|w\|_{H^1(\mathbb{R})}<r$
\[ |\diff^2 \Lag_\xi^h(w)(u,v)|\leq C \|u\|_{H^1(\mathbb{R})} \|v\|_{H^1(\mathbb{R})}. \]
\end{lem}
\begin{proof}
Since $|\sin (\omega)|\leq |\omega|$, we observe that, for all $u,v\in BL^2_h$
\begin{multline*}
| \diff^2 \Lag^h_{\xi} (w )| \leq \| \partial_x u \|_{L^2(\mathbb{R})} \| \partial_x v \|_{L^2(\mathbb{R})} + 3\|w\|_{L^{\infty}(\mathbb{R})}^2 \| u \|_{L^2(\mathbb{R})} \| v \|_{L^2(\mathbb{R})}\\  + \xi_1 \| u \|_{L^2(\mathbb{R})} \| v \|_{L^2(\mathbb{R})} + |\xi_2|  \| \partial_x u \|_{L^2(\mathbb{R})} \| v \|_{L^2(\mathbb{R})}.
\end{multline*}
The result is thus a simple consequence of the classical Sobolev inequality, 
\[ \|w\|_{L^{\infty}(\mathbb{R})}^2 \leq \| w \|_{L^2(\mathbb{R})}  \|\partial_x w\|_{L^2(\mathbb{R})} .\]
\end{proof}

In the following concluding Lemma, we prove the invertibility of $\diff^2 \Lag^h_{\xi \ |  {\rm Ker}(\id - S_h)}(\psi_\xi^h)$ and control the norm of its inverse uniformly with respect to $\xi \in \Omega$ and $h$ small enough. 
\begin{lem}
\label{lem_invert_univ}
There exist $h_0>0$ and $C>0$ such that for all $\xi\in \Omega$ and all $h<h_0$, $\diff^2 \Lag^h_{\xi \ |  {\rm Ker}(\id - S_h)}(\psi_\xi^h)$ is invertible and the norm of its inverse is smaller than $C$.
\end{lem}
\begin{proof} We use Lemma \ref{lem_exRiesz} of the Appendix, by taking $E =  {\rm Ker}(\id - S_h)$ (equipped with $\|\cdot \|_{H^1(\mathbb{R})}$ norm), $T = \diff^2 \Lag^h_{\xi \ |  {\rm Ker}(\id - S_h)}(\psi_\xi^h)$, $E_p = \Span (\psi_\xi^h)^{\perp_{L^2}} \cap {\rm Ker}(\id - S_h)$ and $E_m = \Span (\psi_\xi^h)$.

To get the coercivity estimate  on $E_m$ we apply Lemma \ref{lem_neg}, while coercivity on $E_p$ is obtained from Lemma \ref{lem_cor_psi} after noticing that \[   {\rm Ker}(\id - S_h) \subset BL^2_h \cap \Span( \psi_\xi^h,i\psi_\xi^h,\partial_x \psi_\xi^h)^{\perp_{L^2}}, \]
which is obvious since $i\psi_\xi^h,\partial_x \psi_\xi^h \in {\rm Ker}(\id + S_h) \subset {\rm Ker}(\id - S_h)^{\perp_{L^2}}$.

Applying Lemma \ref{lem_exRiesz}, we obtain the invertibility of $\diff^2 \Lag^h_{\xi \ |  {\rm Ker}(\id - S_h)}(\psi_\xi^h)$ and an explicit control of the norm of its inverse in terms of $\alpha_p$, $\alpha_m$ and $\| T \|$.  However, with Lemma \ref{lem_neg} and Lemma \ref{lem_cor_psi}, we have a uniform control of $\alpha_p$ and $\alpha_m$  with respect to $\xi\in \Omega$ and $h$ small enough, the uniform control of  $\|T\|$ being given by Lemma \ref{lem_d2_unif_bound}.
\end{proof}

\underline{Step 3: The resolution and its consequences}

We now want to apply the inverse function theorem \ref{Thm_Inv_loc} to $\diff \Lag^h_{\xi \ |  {\rm Ker}(\id - S_h)}$ in $\psi_\xi^h$. In the following Lemma, we focus on the last assumption required, i.e. $\diff^2 \Lag_\xi^h$ is a Lipschitz function. 
\begin{lem}
\label{est_lips}
For all $R>0$ there exists $k>0$ such that for all $\xi\in \Omega$, $h>0$, $u_1,u_2,v,w\in BL^2_h$, with $\|u_1\|_{H^1(\mathbb{R})}<R$ and $\|u_2\|_{H^1(\mathbb(\mathbb{R}))}< R$, we have
\[ \| \diff^2 \Lag_\xi^h(u_1)(v,w) -  \diff^2 \Lag_\xi^h(u_2)(v,w) \| \leq k \| u_1 - u_2 \|_{H^1(\mathbb{R})} \| v \|_{H^1(\mathbb{R})} \| w \|_{H^1(\mathbb{R})} .\]
\end{lem}
\begin{proof}
We use mean value inequality. Indeed $\diff^3 \Lag_\xi^h = -\frac14 \diff^3 \| \cdot \|_{L^4(\mathbb{R})}^4$ is clearly a bounded function on bounded subsets of $H^1(\mathbb{R})$. 
\end{proof}

Applying Lemma \ref{est_lips} and Lemma \ref{lem_invert_univ}, we deduce that assumptions of the inverse function Theorem \ref{Thm_Inv_loc} are fulfilled. In the following Proposition, we give its conclusion.
\begin{propo}
\label{prop_apply_inv_thm}
There exist $h_0,r,\lambda,C>0$ such that if $h<h_0$ and $\xi\in \Omega$ then
\begin{itemize}
\item $\diff \Lag^h_{\xi \ |  {\rm Ker}(\id - S_h)}$ is a $C^1$ diffeomorphism from $\{ u \in  {\rm Ker}(\id - S_h) \ | \ \| u - \psi_\xi^h \|_{H^1(\mathbb{R})} < r  \}$ onto its image,
\item if $u \in  {\rm Ker}(\id - S_h) $ and $\| u - \psi_\xi^h \|_{H^1(\mathbb{R})} < r $ then $ \| \diff^2 \Lag^h_{\xi \ |  {\rm Ker}(\id - S_h)}(u)^{-1}\|_{\scriptscriptstyle \mathscr{L}( {\rm Ker}(\id - S_h)';{\rm Ker}(\id - S_h) )} \leq C$,
\item if $\rho<r$ and $\Phi \in  {\rm Ker}(\id - S_h)'$ with $\| \Phi - \diff \Lag^h_{\xi \ |  {\rm Ker}(\id - S_h)}(\psi_\xi^h)  \|_{{\rm Ker}(\id - S_h)'} < \lambda \rho$ then there exists $ u \in  {\rm Ker}(\id - S_h)$ such that $\| u - \psi_\xi^h \|_{H^1(\mathbb{R})} < \rho$ and
\[  \diff \Lag^h_{\xi \ |  {\rm Ker}(\id - S_h)}(u) = \Phi. \]
\end{itemize}
\end{propo}

To apply this result to  $\Phi=0$, we will show that the norm of $ \diff \Lag^h_{\xi \ |  {\rm Ker}(\id - S_h)}(\psi_\xi^h) $ is small when $h \to 0$, uniformly in $\xi \in \Omega$.  It is exactly, what we establish in the following Lemma, which also explains the error term "$h^2$" in Theorem \ref{Thm_DTW}.
\begin{lem}
\label{lem_cons_error}
For all $h_0>0$ there exists $M>0$ such that if $h<h_0$ and $\xi \in \Omega$ then
\[  \forall v\in BL^2_h, \quad \ |\diff \Lag_\xi^h(\psi_\xi^h)(v)| \leq M h^2 \|v\|_{H^1(\mathbb{R})}.  \]
\end{lem}
\begin{proof} The arguments are very  similar to the proof of Lemma \ref{lem_neg}. The key point is the estimate of the consistency error associated to the discretization of the second derivative by finite differences.

Since $\psi_\xi$ is a critical point of $\Lag_\xi$, we deduce from the definition of $\psi_\xi^h$ (see \eqref{ortho_proj}) that
\begin{align}
\nonumber
 \diff \Lag_\xi^h(\psi_\xi^h)(v) &= \diff \Lag_\xi^h(\psi_\xi^h)(v)  -\diff \Lag_\xi(\psi_\xi)(v) \\
 										  & \label{hindemith} = \langle (\partial_x^2 - \Delta_h) \psi_\xi   ,v   \rangle_{L^2(\mathbb{R})} + \diff \frac{\|\cdot \|_{L^4(\mathbb{R})}^4}4 (\psi_\xi)(v)   - \diff \frac{\|\cdot\|_{L^4(\mathbb{R})}^4}4 (\psi_\xi^h)(v) .
\end{align}

To estimate the first term, we use Fourier Plancherel isometry to get
\begin{multline*}
 |\langle (\partial_x^2 - \Delta_h) \psi_\xi   ,v   \rangle_{L^2(\mathbb{R})}| = \left| \frac1{2\pi} \int_{\mathbb{R}} \left[ \frac4{h^2}\sin^2\left(  \frac{\pi \omega h}2\right)-\omega^2 \right] \widehat{\psi_\xi}(\omega).\widehat{v}(\omega)\domega \right|\\
 													 \leq \sup_{\omega\in \mathbb{R}} \left| \frac{\sinc^2\left(  \frac{ \omega h}2\right)-1}{\omega^2} \right| \|\partial_x^4 \psi_\xi \|_{L^2(\mathbb{R})} \|v\|_{L^2(\mathbb{R})} 
 													= \left( \frac{h}2 \right)^2 \sup_{\omega\in \mathbb{R}} \left| \frac{\sinc^2(\omega)-1}{\omega^2} \right| \|\partial_x^4 \psi_\xi \|_{L^2(\mathbb{R})} \|v\|_{L^2(\mathbb{R})} .
\end{multline*}
As we can see from Lemma \ref{reg_CTW}, $\|\partial_x^4 \psi_\xi \|_{L^2(\mathbb{R})}$ is clearly bounded uniformly with respect to $\xi\in \Omega$.

To control the second term in \eqref{hindemith},  we use mean value inequality and Lemma \ref{reg_CTW} to get some constants $M,C>0$ independent of $h$ and $\xi \in \Omega$ such that
\[ \left| \diff \frac{\|.\|_{L^4(\mathbb{R})}^4}4 (\psi_\xi)(v)   - \diff \frac{\|.\|_{L^4(\mathbb{R})}^4}4 (\psi_\xi^h)(v) \right| \leq M \|\psi_\xi - \psi_\xi^h \|_{L^2(\mathbb{R})} \|v\|_{L^2(\mathbb{R})} \leq C e^{-\frac{\pi \varepsilon}h}  \|v\|_{L^2(\mathbb{R})},\]
which shows the result, provided $h < h_0$ small enough.  
\end{proof}

Applying Lemma \ref{lem_cons_error}, if $h_0$ is smaller than $\sqrt{\frac{\lambda r}{2M}}$ we can choose $\Phi = 0$ in Proposition \ref{prop_apply_inv_thm} and we  denote by $\eta_\xi^h$ the corresponding critical point of $\Lag^h_{\xi \ |  {\rm Ker}(\id - S_h)}$.  As shown in the first step, $\eta_\xi^h$ is thus a critical point of $\Lag^h_\xi$, and with Proposition \ref{prop_apply_inv_thm}, we have proven the points $a)$ to $d)$ of Theorem \ref{thm_ExTW}. It remains to show the coercivity estimate $e)$ and the regularity with respect to $\xi$.

To obtain the coercivity estimate, we just have to perturb the estimate of Lemma \ref{lem_cor_psi} with Lemma \ref{perturb_coer} presented in Appendix. This is given by the following result 
 
\begin{lem} 
\label{lem_coer_near}
There exist $\alpha>0$, $h_0>0$ and $\rho>0$ such that for all $\xi \in \Omega$, $h<h_0$ and $u \in BL^2_h$ such that $\| u - \psi_\xi^h \|_{H^1(\mathbb{R})}< \rho$, we have
\begin{equation}
\label{bach} \forall v\in BL^2_h \cap \Span(i u, \partial_x u, u)^{\perp_{L^2}}, \  \quad  \diff^2 \Lag_\xi^h( u)( v,v) \geq \alpha \|v\|_{H^1(\mathbb{R})}^2. 
\end{equation}
\end{lem}
\begin{proof}
The proof is very similar to the first part of the proof of Lemma \ref{lem_cor_psi}, but we need to track precisely the dependence of the constant with respect to $h$.

First, applying Lemma \ref{lem_cor_psi}, we know that there exists $h_0>0$ and $\alpha>0$ such that for all $h<h_0$ and all $\xi\in \Omega$ we have
\[ \forall v\in BL^2_h \cap \Span(i \psi_\xi^h, \partial_x \psi_\xi^h, \psi_\xi^h)^{\perp_{L^2}}, \ \quad  \diff^2 \Lag_\xi^h( \psi_\xi^h)( v,v) \geq \alpha \|v\|_{H^1(\mathbb{R})}^2. \]

We want to apply Lemma \ref{perturb_coer} in $\psi_\xi^h$ in order to perturb this estimate and prove that there exist $h_0>0$, $c_1,c_2>0$ such that for all $\xi \in \Omega$ and all $h<h_0$, if 
\[ \| u - \psi_\xi^h\|_{H^1(\mathbb{R})}\leq c_1 \quad \textrm{ and }  \quad  \max \left( | \langle \psi_\xi^h, v \rangle_{L^2(\mathbb{R})} | , |\langle i\psi_\xi^h, v \rangle_{L^2(\mathbb{R})}| , |\langle \partial_x \psi_\xi^h, v \rangle_{L^2(\mathbb{R})}| \right)\leq c_2 \|v\|_{H^1(\mathbb{R})},  \] then 
\[  \diff^2 \Lag_\xi^h( u)( v,v) \geq \alpha \|v\|_{H^1(\mathbb{R})}^2 \geq \frac{\alpha}2 \| v\|_{H^1(\mathbb{R})}^2. \]
To do this, we  apply Lemma \ref{perturb_coer} in $\psi_\xi^h$ with $E = BL^2_h$, $X = \Span(i \psi_\xi^h, \partial_x \psi_\xi^h, \psi_\xi^h)$ and $b = \diff^2\Lag_\xi^h$. The Gram matrix is 
\[ G_{\xi}^h = \begin{pmatrix} \| \psi_\xi^h \|_{L^2(\mathbb{R})}^2 & \langle i \psi_\xi^h , \partial_x \psi_\xi^h \rangle_{L^2(\mathbb{R})} & 0 \\
							 \langle i \psi_\xi^h , \partial_x \psi_\xi^h \rangle_{L^2(\mathbb{R})}  & 	 \| \partial_x \psi_\xi^h \|_{L^2(\mathbb{R})}^2 & 0 \\
				0 & 0 &	\| \psi_\xi^h \|_{L^2(\mathbb{R})}^2
\end{pmatrix}.\]
To prove that the constants $c_1,c_2>0$ --explicitly given by Lemma \ref{perturb_coer}-- are independent of $\xi \in \Omega$ and $h$ small enough, we have to control uniformly the inverse of $G_{\xi}^h$, the norm of $\psi_\xi^h$ in $H^1(\mathbb{R})$, the norm of $\diff^2\Lag_\xi^h(\psi_\xi^h)$ and prove that $\diff^2\Lag_\xi^h$ is uniformly Lipschitz.

The control of $\psi_\xi^h$ in $H^1(\mathbb{R})$ is obvious, and Lemma \ref{est_lips} shows that $\diff^2\Lag_\xi^h$ is uniformly Lipschitz. 
In Lemma \ref{lem_d2_unif_bound}, we have proven that the norm $\diff^2\Lag_\xi^h(\psi_\xi^h)$ is uniformly bounded with respect to $h$ and $\xi \in \Omega$. So we just need to focus on the Gram matrix.

   As explained in the proof of Lemma \ref{lem_cor_psi} $G_\xi^{0} = G_\xi$ is invertible. Furthermore, $(h,\xi) \mapsto G_{\xi}^h$ is a continuous function on $\mathbb{R}_+ \times \overline{\Omega}$, so there exists $h_0>0$ and $M>0$ such that for all $h<h_0$ and all $\xi \in \Omega$,  $G_{\xi}^h$ is invertible and $\| (G_{\xi}^h)^{-1}\|_{\infty} \leq M$.

To prove \eqref{bach}, let us set $\rho = \min(c_1,c_2)$ and consider $h<h_0$ and $\xi \in \Omega$. Let $u,v\in BL^2_h$ be such that $\| u - \psi_\xi^h \|_{H^1(\mathbb{R})}< \rho$ and $v\in \Span(i u, \partial_x u, u)^{\perp_{L^2}}$. Then, we have
 \[ \max \left( | \langle \psi_\xi^h, v \rangle_{L^2(\mathbb{R})} | , |\langle i\psi_\xi^h, v \rangle_{L^2(\mathbb{R})}| , |\langle \partial_x \psi_\xi^h, v \rangle_{L^2(\mathbb{R})}| \right)\leq \| u - \psi_\xi^h \|_{H^1(\mathbb{R})} \|v\|_{H^1(\mathbb{R})} \leq c_2 \|v\|_{H^1(\mathbb{R})}.  \]
 Consequently, we can apply the result of Lemma \ref{perturb_coer} to get
 \[  \diff^2 \Lag_\xi^h( u)(v,v) \geq \alpha \|v\|_{H^1(\mathbb{R})}^2 \geq \frac{\alpha}2 \| v\|_{H^1(\mathbb{R})}^2. \]
 which shows the result. 
\end{proof}

The following Lemma concludes the proof of Theorem \ref{thm_ExTW}. It shows that $\xi \mapsto \eta_\xi^h$ is $C^1$ and that its derivative with respect to $\xi$ is a good approximation of the derivative of $\psi_\xi$ with respect to $\xi$.
\begin{lem}
\label{lem_reg_C1}
Let $h_0,r,\lambda,C>0$ be the constants given in Proposition \ref{prop_apply_inv_thm} and $M>0$ be the constant associated with $h_0>0$ given in Lemma \ref{lem_cons_error}. Let $h_1 := \min(h_0, \sqrt{\frac{\lambda r}{2M}} )$ and for any $h<h_1$ and $\xi \in \Omega$, let  $\eta_\xi^h$ denotes the critical point of $\Lag_\xi^h$ at a distance smaller than $r$ from $\psi_\xi^h$.
There exists $k>0$ such that for all $h<h_1$, for all $\xi \in \Omega$, $\xi \mapsto \eta_\xi^h$ is $C^1$ and 
\[  \| \diff_\xi \psi_\xi(\zeta)-\diff_\xi \eta_\xi^h(\zeta) \|_{H^1(\mathbb{R})} \leq k |\zeta| h^2 .\]
\end{lem}
\begin{proof}
Let $h<h_1$ and $\xi \in \Omega$. The function $(u,\zeta)\mapsto \diff \Lag^h_{\zeta \ |  {\rm Ker}(\id - S_h)}(u)$ is clearly a $C^1$ function. Applying Proposition \ref{prop_apply_inv_thm}, its derivative with respect to $u$ in $(\eta_\xi^h,\xi)$ is invertible. By construction, $(\eta_\xi^h,\xi)$ is a zero point of this function. So we can apply the implicit function theorem.

There exists $\rho>0$ such that $B(\xi,\rho) \subset \Omega$ and $\Gamma\in C^1(B(\xi,\rho) ; {\rm Ker}(\id - S_h)) $ such that
\[ \forall \zeta \in B(\xi,\rho),\quad  \ \diff \Lag^h_{\zeta \ |  {\rm Ker}(\id - S_h)}(\Gamma(\zeta))=0. \]
To prove that $\Gamma(\zeta)= \eta_\zeta^h$, it is enough to prove that  $\| \Gamma(\zeta) - \psi_{\zeta}^h \|_{H^1(\mathbb{R})}< r$. But by construction of $h_1$, we deduce of Proposition \ref{prop_apply_inv_thm} that
\[  \| \Gamma(\xi) - \psi_{\xi}^h \|_{H^1(\mathbb{R})}= \| \eta_\xi^h - \psi_{\xi}^h \|_{H^1(\mathbb{R})}\leq \frac{r}2.\]
Furthermore, $\zeta \mapsto \Gamma(\zeta)$ and $\zeta \mapsto \psi_{\zeta}^h$ are continuous functions. So there exists $\widetilde{\rho}<\rho$ such that,
\[  \forall \zeta \in B(\xi, \widetilde{\rho}), \  \quad  \| \Gamma(\xi) - \Gamma(\zeta) \|_{H^1(\mathbb{R})} +  \| \psi_{\zeta}^h - \psi_{\xi}^h \|_{H^1(\mathbb{R})} \leq \frac{r}4. \]
Applying the triangle inequality for $\zeta\in B(\xi, \widetilde{\rho})$, we thus obtain
\[  \| \Gamma(\zeta) - \psi_{\zeta}^h \|_{H^1(\mathbb{R})}\leq \frac34 r < r.\]
Since we have proven in Proposition \ref{prop_apply_inv_thm} that $\diff \Lag^h_{\zeta \ |  {\rm Ker}(\id - S_h)}$ is invective on $\{ u \in  {\rm Ker}(\id - S_h) \ | \ \| u - \psi_\zeta^h \|_{H^1(\mathbb{R})} < r  \}$, we get $\Gamma(\zeta) = \eta_\zeta^h$ for all $\zeta\in B(\xi, \widetilde{\rho})$. Consequently, $\zeta \mapsto \eta_\zeta^h$ is $C^1$.

Now, we have to prove that $\diff_\xi \eta_\xi^h$ is an approximation of $\diff_\xi \psi_\xi$. First, we introduce some constants $c,\varepsilon>0$ such that for all $\xi \in \Omega$ and all $\zeta \in \mathbb{R}^2$, we have
\begin{equation}
\label{reg_ana_der}
 \forall \omega \in \mathbb{R}, \quad \ |\widehat{ \diff_\xi \psi_\xi(\zeta) }(\omega)| \leq c|\zeta|e^{-\varepsilon |\omega|}.
\end{equation}
There are several ways to establish this property. The most direct is probably to deduce it from the explicit formula of $\psi_\xi$ (see \eqref{def_CTW}). But it can also be proven with elliptic regularity as in Theorem \ref{thm_anareg} below.

Then, we deduce from the definition of $\psi_\xi^h$ that for all $h>0$, $\xi \mapsto \psi_\xi^h$ is $C^1$ and there exists $k>0$ such that
\begin{equation}
\label{approx_deriv}
 \forall h>0,\quad \forall \xi \in \Omega,\quad  \forall \zeta \in \mathbb{R}^2, \quad \  \| \diff_\xi \psi_\xi(\zeta) - \diff_\xi \psi_\xi^h(\zeta) \|_{H^1(\mathbb{R})}  \leq  k|\zeta|e^{-\frac{\varepsilon \pi}{2h} }.
\end{equation} So we just need to prove that $\diff_\xi \eta_\xi^h$ is an approximation of $\diff_\xi \psi_\xi^h$ of order $2$ in $h$. To compare these quantities, we are going to prove that they are almost solutions of the same linear equation.

Since $\eta_\xi^h$ is a critical point of $\Lag_\xi^h$, it satisfies for all $v\in {\rm Ker}(\id - S_h)$, $\diff \Lag_{\xi \ |  {\rm Ker}(\id - S_h)}^h (\eta_\xi^h)(v)=0$. So we can calculate the derivative with respect to $\xi$ to obtain that 
\[ \forall \zeta \in \mathbb{R}^2, \quad \forall v\in {\rm Ker}(\id - S_h),\quad  \ \diff^2 \Lag_{\xi \ |  {\rm Ker}(\id - S_h)}^h (\eta_\xi^h)(v,\diff_\xi \eta_\xi^h(\zeta)) + b_{\zeta}^h[\eta_\xi^h](v) = 0,   \]
where $b_{\zeta}^h[u] \in ({\rm Ker}(\id - S_h))'$ is defined for $u \in  {\rm Ker}(\id - S_h)$ by
\[  b_{\zeta}^h[u](v) := \zeta_1 \langle u , v \rangle_{L^2(\mathbb{R})} + \zeta_2 \langle i \partial_x u , v  \rangle_{L^2(\mathbb{R})}.  \]
Similarly, we define $E_{\xi,\zeta}^h \in {\rm Ker}(\id - S_h)'$ by
\[ \forall \zeta \in \mathbb{R}^2, \quad  \forall v\in {\rm Ker}(\id - S_h), \quad  \ \diff^2 \Lag_{\xi \ |  {\rm Ker}(\id - S_h)}^h (\psi_\xi^h)(v,\diff_\xi \psi_\xi^h(\zeta)) + b_{\zeta}^h[\psi_\xi^h](v)  = E_{\xi,\zeta}^h(v). \]
Then, we get (in  ${\rm Ker}(\id - S_h)'$), for all $\zeta \in \mathbb{R}^2$ ,
\begin{align*}
  \diff^2 \Lag_{\xi \ |  {\rm Ker}(\id - S_h)}^h (\eta_\xi^h)(\diff_\xi \psi_\xi^h(\zeta)-\diff_\xi \eta_\xi^h(\zeta)) =&  \left[ \diff^2 \Lag_{\xi \ |  {\rm Ker}(\id - S_h)}^h (\eta_\xi^h) - \diff^2 \Lag_{\xi \ |  {\rm Ker}(\id - S_h)}^h (\psi_\xi^h) \right](\diff_\xi \psi_\xi^h(\zeta)) \\   
  &+b_{\zeta}^h[\eta_\xi^h - \psi_\xi^h] +  E_{\xi,\zeta}^h(v).
\end{align*}
However, we have proven in Proposition \ref{prop_apply_inv_thm} that $ \diff^2 \Lag_{\xi \ |  {\rm Ker}(\id - S_h)}^h (\eta_\xi^h)$ is invertible and that the norm of its invert is smaller than $C$. So we just need to control the three right terms of the last equality.
\begin{itemize}
\item Applying \eqref{reg_ana_der} and \eqref{approx_deriv}, for all $h>0$ and all $\xi \in \Omega$, we have $\|\diff_\xi \psi_\xi^h(\zeta)\|_{H^1(\mathbb{R})}\leq 2|\zeta|k$. So applying Lemma \ref{est_lips}, there exists $\kappa>0$, such that for all $h<h_1$, all $\xi \in \Omega$, all $\zeta \in \mathbb{R}^2$ and all $v\in {\rm Ker}(\id - S_h)$,
\begin{multline*}
 \left| \left[ \diff^2 \Lag_{\xi \ |  {\rm Ker}(\id - S_h)}^h (\eta_\xi^h) - \diff^2 \Lag_{\xi \ |  {\rm Ker}(\id - S_h)}^h (\psi_\xi^h) \right](\diff_\xi \psi_\xi^h(\zeta))(v) \right| \\
\leq \kappa \|  \eta_\xi^h - \psi_\xi^h \|_{H^1(\mathbb{R})} |\zeta| \|v\|_{H^1(\mathbb{R})} \leq \frac{M\kappa}{\lambda} h^2 |\zeta| \|v\|_{H^1(\mathbb{R})}.
\end{multline*}

\item The estimate of the second term is obvious. Indeed, for all $h<h_1$, all $\xi \in \Omega$, all $\zeta \in \mathbb{R}^2$ and all $v\in {\rm Ker}(\id - S_h)$ we have
$$
|b_{\zeta}^h[\eta_\xi^h - \psi_\xi^h](v)| \leq |\zeta| ( \| \eta_\xi^h - \psi_\xi^h \|_{L^2(\mathbb{R})} \| v\|_{L^2(\mathbb{R})} + \| \partial_x (\eta_\xi^h - \psi_\xi^h) \|_{L^2(\mathbb{R})} \|  \partial_x  v\|_{L^2(\mathbb{R})}  )\leq 2 \frac{M}{\lambda} h^2 |\zeta| \|v\|_{H^1(\mathbb{R})}.
$$

\item The bound on the term $E_{\xi,\zeta}^h$ is more difficult to obtain. First, we have to identify it. Since $\psi_\xi$ is a critical point of $\Lag_\xi$, it satisfies $\diff \Lag_\xi(\psi_\xi)(v) =0$ for all $v\in H^1(\mathbb{R})$. By calculating its derivative with respect to $\xi$, we get for $\zeta\in \mathbb{R}^2$,
\[ \diff^2 \Lag_\xi(\psi_\xi)(v, \diff_\xi \psi_\xi(\zeta)) +  \zeta_1 \langle \psi_\xi , v \rangle_{L^2(\mathbb{R})} + \zeta_2 \langle i \partial_x \psi_\xi , v  \rangle_{L^2(\mathbb{R})} =0 .\]
In particular, we can choose $v\in {\rm Ker}(\id - S_h)$. Consequently, we get
\begin{align*}
 \diff^2 \Lag_\xi^h(\psi_\xi^h)(v, \diff_\xi \psi_\xi^h(\zeta)) +  b_{\zeta}^h[\psi_\xi^h] &+ \langle ( \Delta_h - \partial_x^2) \diff_\xi \psi_\xi^h(\zeta) , v \rangle_{L^2(\mathbb{R})}\\
 & + \diff^2 \frac{\|\cdot\|_{L^4(\mathbb{R})}^4}4 (\psi_\xi^h)(v, \diff_\xi \psi_\xi^h(\zeta)) - \diff^2 \frac{\|\cdot\|_{L^4(\mathbb{R})}^4}4 (\psi_\xi)(v, \diff_\xi \psi_\xi(\zeta)) = 0. 
\end{align*}
So we have
\[ E_{\xi,\zeta}^h(v) =  \diff^2 \frac{\|\cdot\|_{L^4(\mathbb{R})}^4}4 (\psi_\xi)(v, \diff_\xi \psi_\xi(\zeta)) - \diff^2 \frac{\|\cdot\|_{L^4(\mathbb{R})}^4}4 (\psi_\xi^h)(v, \diff_\xi \psi_\xi^h(\zeta)) + \langle ( \partial_x^2 - \Delta_h ) \diff_\xi \psi_\xi^h(\zeta) , v \rangle_{L^2(\mathbb{R})}. \]

To estimate  $\langle ( \partial_x^2 - \Delta_h ) \diff_\xi \psi_\xi^h(\zeta) , v \rangle_{L^2(\mathbb{R})}$ we use the same method as in Lemma \ref{lem_cons_error} and we can find an universal constant $C_{\rm univ}>0$ such that
\begin{equation}
\label{level_1}
  \left| \langle ( \partial_x^2 - \Delta_h ) \diff_\xi \psi_\xi^h(\zeta) , v \rangle_{L^2(\mathbb{R})} \right| \leq C_{\rm univ} h^2 \|\diff_\xi \psi_\xi^h(\zeta)\|_{H^2(\mathbb{R})} \| v\|_{L^2(\mathbb{R})}.
\end{equation}
On the other hand, we have
\begin{multline*}
\left| \diff^2 \|\cdot\|_{L^4(\mathbb{R})}^4 (\psi_\xi)(v, \diff_\xi \psi_\xi(\zeta))  - \diff^2 \|\cdot\|_{L^4(\mathbb{R})} (\psi_\xi^h)(v, \diff_\xi \psi_\xi^h(\zeta))  \right| \\
\leq  12 \| \psi_\xi +  \psi_\xi^h \|_{L^4(\mathbb{R})}  \| \psi_\xi -  \psi_\xi^h \|_{L^4(\mathbb{R})} \| v\|_{L^4(\mathbb{R})} \|  \diff_\xi \psi_\xi(\zeta) \|_{L^4(\mathbb{R})} \\
+ 12  \| \diff_\xi \psi_\xi(\zeta)- \diff_\xi \psi_\xi^h(\zeta) \|_{L^4(\mathbb{R})} \| \psi_\xi^h \|_{L^4(\mathbb{R})}^2 \| v\|_{L^4(\mathbb{R})}.
\end{multline*}
Applying Gagliardo-Nirenberg inequality, \eqref{reg_ana_der} and Lemma \ref{reg_CTW}, it is clear that $ \| \psi_\xi +  \psi_\xi^h \|_{L^4(\mathbb{R})} $, $\| \psi_\xi^h \|_{L^4(\mathbb{R})}^2$, $|\zeta|^{-1}\|  \diff_\xi \psi_\xi(\zeta) \|_{L^4(\mathbb{R})}$ and $|\zeta|^{-1} \|\diff_\xi \psi_\xi^h(\zeta)\|_{H^2(\mathbb{R})}$ are bounded uniformly with respect to $\xi \in \Omega$ and $h<h_1$. \\
Consequently, by using \eqref{approx_deriv}, there exist $\ell>0$, $\kappa>0$ such that for all $h<h_1$, all $\xi \in \Omega$ and all $\zeta \in \mathbb{R}^2$, we have
\[ |E_{\xi,\zeta}^h(v)| \leq \kappa \left( h^2 + e^{-\frac{\ell}h} \right) \leq \kappa h^2 \left( 1 + \left( \frac2{e\ell} \right)^2 \right), \]
which concludes the proof of the Lemma. 
\end{itemize}
\end{proof}

\subsection{Gevrey uniform regularity, Lyapunov stability and some adjustments}

The discrete traveling waves constructed in Theorem \ref{thm_ExTW} enjoy most of the properties of the  continuous traveling waves $\psi_\xi$. In this subsection, we analyse some of these properties useful to prove Theorem \ref{Thm_DTW}.

First, we  study their regularity. Of course, since they belong to $BL^2_h$ they are entire functions but we can give a control of them in Gevrey norms uniformly with respect to $h$ and $\xi$.
\begin{theo}
\label{thm_anareg}
There exists $h_0>0$ such that for all $M>0$, there exist $C,\varepsilon>0$ such that for all $ h<h_0$ and all $\xi \in \Omega$, if $u\in BL_2^h$ satisfies $\|u\|_{H^1(\mathbb{R})}\leq M$ then
\begin{equation}
\label{poulenc}\diff \Lag_\xi^h(u) = 0 \ \Rightarrow \  \forall \omega\in \mathbb{R}, \ |\widehat{u}(\omega)| < C e^{-\varepsilon |\omega|}.\end{equation}
\end{theo}
\begin{proof}
To get this result of elliptic regularity, we prove, in the following lemma, a result of coercivity. 
\begin{lem}
\label{coer_ellip}
Let $f:\mathbb{R}\to \mathbb{R}$ be a function continuous in $0$ such that $f(0)=1$. Assume that there exists $m>0$ such that $f\geq m$ on $\mathbb{R}$. Then there exist $\alpha>0$ and $h_0>0$ such that for all $\xi\in \Omega$ and $h<h_0$  we have
\[ \forall \omega\in \mathbb{R}, \quad  \  \omega^2 f(h\omega) + \xi_2 \omega +\xi_1 \geq \alpha \left( 1 + \omega^2  \right).  \] 
\end{lem}
\begin{proof}
First, observe that we have
\[ \omega^2 + \xi_2 \omega +\xi_1 = \left( \omega + \frac{\xi_2}2  \right)^2 + \xi_1 - \left( \frac{\xi_2}2 \right)^2.\]
Consequently, there exists $\beta>0$ such that for all $\xi\in \Omega$, we have
\[ \omega^2 + \xi_2 \omega +\xi_1 \geq \beta \left( 1 + \omega^2 \right).\]

Second observe that there exists $\omega_0>0$ such that, for all $|\omega|>\omega_0$ we have
\[ m \omega^2 + \xi_2 \omega +\xi_1 \geq \frac{m}2 \left( 1 + \omega^2  \right). \]
Consequently, for such $\omega$ and for any $h>0$, we have
\[\omega^2 f(h\omega) + \xi_2 \omega +\xi_1 \geq \frac{m}2 \left( 1 + \omega^2  \right).  \]

Now, since $f$ is continuous in $0$, there exists $\delta>0$ such that if $|\omega|<\delta$ then $|f(\omega)-1|<\frac{\beta}2$. Consequently, if $|\omega|<\omega_0$ and $h<\frac{\delta}{\omega_0}=:h_0$ then we have
\[ \omega^2 f(h\omega) + \xi_2 \omega +\xi_1 =  \omega^2  + \xi_2 \omega +\xi_1 + \omega^2 (f(h\omega) -1)\geq \beta \left( 1 + \omega^2 \right) - \frac{\beta}2\omega^2 \geq \frac{\beta}2  \left( 1 + \omega^2 \right). \]
\end{proof}
We now prove the  elliptic regularity result \eqref{poulenc}. Let us write Equation $\diff \Lag_\xi^h(u)=0$ in terms of the Fourier transform $\widehat{u}$. It is written
\begin{equation}
\label{crit_Four}
\forall \omega \in \left( -\frac{\pi}h, \frac{\pi}h \right), \ \quad   \left( \frac{4}{h^2}\sin^2\left( \frac{\omega h}2\right) - \xi_2 \omega + \xi_1 \right)  \widehat{u}(\omega) = \widehat{u}* \widehat{\bar{u}} *\widehat{u}(\omega).
\end{equation}
Applying Lemma \ref{coer_ellip} to $f(\omega) = \sinc^2\left( \frac{\omega }2\right) + \mathbb{1}_{\left( -\pi, \pi \right)^c}(\omega)$, for which $m=\frac4{\pi^2}$, there exist $h_0>0$ and $\alpha>0$ such that if $\xi \in \Omega$ and $h<h_0$, 
\[ \forall \omega \in \left(-\frac{\pi}h, \frac{\pi}h \right), \quad \  \frac{4}{h^2}\sin^2\left( \frac{\omega h}2\right) - \xi_2 \omega + \xi_1 \geq  \alpha \left( 1 + \omega^2 \right).\]
Hence, we have using \eqref{crit_Four} 
\begin{equation}
\label{crit_Four_ine}
\forall \omega \in \left(-\frac{\pi}h, \frac{\pi}h \right), \ \quad  \alpha \left( 1+\omega^2 \right)  |\widehat{u}(\omega)| \leq |\widehat{u}(\omega)|* |\widehat{\bar{u}}(\omega)| *|\widehat{u}(\omega)|.
\end{equation}
Now, we prove by induction (on $n$) that there exists $C>0$, that only depend of $\alpha$ and $M$ such that
\begin{equation}
\label{ind_hyp}
\forall \ 1\leq p \leq \infty, \ \quad  \| \omega^n  \widehat{u}  \|_{L^p(\mathbb{R})} \leq C^n n! .
\end{equation}
First, we consider the cases $n=1$ and $n=0$. Since we have assumed that $\|u\|_{H^1(\mathbb{R})}\leq M$, we have
\[ \| \widehat{u} \|_{L^1(\mathbb{R})} \leq  \left\| \frac1{\sqrt{1+\omega^2}}  \right\|_{L^2(\mathbb{R})} \| \sqrt{1+\omega^2} \widehat{u}(\omega)   \|_{L^2(\mathbb{R})}= \sqrt2 \pi \|u\|_{H^1(\mathbb{R})}\leq \sqrt2 \pi M . \]
Then, we get from \eqref{crit_Four_ine}
\[ \|\omega \widehat{u} \|_{L^1(\mathbb{R})} \leq \| (1+\omega^2) \widehat{u} \|_{L^1(\mathbb{R})} \leq  \frac1{\alpha} \| \widehat{u} \|_{L^1(\mathbb{R})}^3 \leq \frac{\sqrt8 \pi^3 M^3}{\alpha}.  \]
Furthermore, we also get from \eqref{crit_Four_ine},
\[  \|(1+|\omega|)  \widehat{u} \|_{L^{\infty}(\mathbb{R})} \leq \frac1{\alpha} \left\| \frac{1+|\omega|}{1+\omega^2} \right\|_{L^{\infty}(\mathbb{R})} \| \widehat{u} \|_{L^1(\mathbb{R})}^3.\]
We deduce \eqref{ind_hyp} for $n = 0$ and $1$ and for the other values of $p$ using H\"older inequality.

Now, we assume that \eqref{ind_hyp} is proved for all $0\leq n\leq N+1$. We deduce from \eqref{crit_Four} that for all $\omega\in  \left(-\frac{\pi}h, \frac{\pi}h \right)$, we have
\begin{multline*}
 |\omega|^{N+2} |\widehat{u}(\omega)| \leq |\omega|^{N}(1+\omega^2) |\widehat{u}(\omega)| \\
 														\leq \frac1{\alpha} \left| \omega^N (\widehat{u}* \widehat{\bar{u}} *\widehat{u}(\omega))  \right| 
 														= \frac1{\alpha} \left| \sum_{n_1+n_2+n_3 = N} \frac{N!}{n_1 ! n_2! n_3 !}(\omega^{n_1}\widehat{u})* (\omega^{n_2}\widehat{\bar{u}}) *(\omega^{n_3}\widehat{u})(\omega)  \right|.
\end{multline*}
We deduce from Young convolution inequality that if $\frac{3}q = 2 + \frac1p$ then
\[  \|\omega^{N+2} \widehat{u}\|_{L^p(\mathbb{R})} \leq \frac1{\alpha} \sum_{n_1+n_2+n_3 = N} \frac{N!}{n_1 ! n_2! n_3 !}  \prod_{j=1}^n \|\omega^{n_j} \widehat{u}\|_{L^q(\mathbb{R})} .  \]
Using the induction hypothesis, we obtain
\[  \|\omega^{N+2} \widehat{u}\|_{L^p(\mathbb{R})} \leq \frac1{\alpha} C^N N! \ \# \{ (n_1,n_2,n_3)  \ | \ n_1+n_2+n_3 = N \} = \frac1{2\alpha C^2} C^{N+2} (N+2)! \]
So, if $C$ is chosen large enough to ensure $2\alpha C^2 \geq 1$, we obtain the result by induction.

Choosing 
$p=\infty$ in \eqref{ind_hyp},    we get
\[ \forall n\in \mathbb{N}, \quad \forall \omega\in \mathbb{R}^{*}, \quad \ |\widehat{u}(\omega)|\leq \left(\frac{C}{|\omega|} \right)^n n!.\]
But using Stirling formula, we get an universal constant $c>0$ such that
$ n! \leq c e^{-\frac{2n}3} n^n$. 
Consequently, if $|\omega|\geq C$ and $n=\lfloor \frac{\omega}{C} \rfloor$, we have
$  |\widehat{u}(\omega)|\leq c e^{- \frac{|\omega|}{2C}}$, and this shows the result. 
\end{proof}

In the following lemma, we prove that Lagrange functions are Lyapunov functions for the traveling waves of the homogeneous Hamiltonian. These uniform estimates are  discrete versions of the continuous case, see for example Proposition 8.8 of \cite{MR3443560}. They are the key estimates for applying the energy-momentum method.
\begin{lem} 
\label{lem_orb_stab}
Let $h_0,C,\rho,\alpha>0$ be the constants given by Theorem \ref{thm_ExTW}. There exist $r,\beta,h_1>0$ such that for all $h<h_1$, all $\xi \in \Omega$, all $u \in BL^2_h \cap \Span(i\eta_\xi^h,\partial_x \eta_\xi^h)$, if $\| u - \eta_\xi^h \|_{H^1(\mathbb{R})}\leq r$ and $\| u \|_{L^2(\mathbb{R})}^2= \| \eta_\xi^h \|_{L^2(\mathbb{R})}^2$ then
\begin{equation}
	\label{est_Lyap}
	\beta \| u -\eta_\xi^h \|_{H^1(\mathbb{R})}^2 \leq \Lag_\xi^h(u) -\Lag_\xi^h(\eta_\xi^h).
\end{equation} 
\end{lem}
\begin{proof}
Let $h_1<h_0$ and $\varepsilon>0$ be such that
\[ \forall h<h_1,\quad \forall \xi \in \Omega, \quad \ \|\eta_\xi^h \|_{L^2(\mathbb{R})}^2 \geq  \frac{\|\psi_\xi \|_{L^2(\mathbb{R})}^2}2 \geq \frac{\varepsilon}2.\]
Let $r\in(0,1)$ be a positive constant that will be determined later.

Since $\eta_\xi^h$ is bounded in $H^1(\mathbb{R})$, uniformly with respect to $\xi\in \Omega$ and $h<h_0$, there exists a constant $M>0$ such that for all $\xi\in \Omega$, $h<h_0$, $w_1,w_2,w_3 \in BL^2_h$, we have $\|\eta_\xi^h\|_{H^1(\mathbb{R})} \leq M$,
\[  |\diff^2 \Lag_\xi^h(\eta_\xi^h)(w_1,w_2)|\leq M \|w_1 \|_{H^1(\mathbb{R})} \|w_2 \|_{H^1(\mathbb{R})}      \]
and
\[  \sup_{\|\eta_\xi^h - w\|_{H^1(\mathbb{R})}\leq 1} |\diff^3 \Lag_\xi^h(w)(w_1,w_2,w_3)|\leq M \|w_1 \|_{H^1(\mathbb{R})} \|w_2 \|_{H^1(\mathbb{R})} \|w_3 \|_{H^1(\mathbb{R})} .  \]
Indeed, the first estimate has been establish in Lemma \ref{lem_d2_unif_bound} and the second is obvious since $\diff^3 \Lag_\xi^h = \diff^3 \frac{\|\cdot\|_{L^4(\mathbb{R})}^4}4$.

Consider $h<h_1$, $\xi \in \Omega$ and $u\in BL^2_h \cap \Span(i\eta_\xi^h,\partial_x \eta_\xi^h)^{\perp_{L^2}}$ such that $\| u - \eta_\xi^h \|_{H^1(\mathbb{R})}\leq r$ and $\| u \|_{L^2(\mathbb{R})}^2= \| \eta_\xi^h \|_{L^2(\mathbb{R})}^2$. Then we define 
\[ v = \eta_\xi^h + \left[ (u-\eta_\xi^h) - \frac{\eta_\xi^h}{\| \eta_\xi^h \|_{L^2(\mathbb{R})} } \langle u - \eta_\xi^h ,\frac{\eta_\xi^h}{\| \eta_\xi^h \|_{L^2(\mathbb{R})}}\rangle_{L^2(\mathbb{R})} \right].\]
By construction, $v-\eta_\xi^h$ belongs to $\Span(i\eta_\xi^h,\partial_x \eta_\xi^h,\eta_\xi^h)^{\perp_{L^2}}$. Furthermore, $v - \eta_\xi^h$ is a second order perturbation of $u-\eta_\xi^h$ because, since $\| u \|_{L^2(\mathbb{R})}^2= \| \eta_\xi^h \|_{L^2(\mathbb{R})}^2$, we have
\[  \langle \eta_\xi^h , u-\eta_\xi^h   \rangle_{L^2(\mathbb{R})} = -\frac12 \| u -\eta_\xi^h \|_{L^2(\mathbb{R})}^2. \]
So, we get
\[  \| u - v\|_{H^1(\mathbb{R})} = \frac { \| \eta_\xi^h \|_{H^1(\mathbb{R})} }{  2 \| \eta_\xi^h \|_{L^2(\mathbb{R})}^2 }\| u -\eta_\xi^h\|_{L^2(\mathbb{R})}^2 \leq \frac{2M}{\epsilon^2} \| u -\eta_\xi^h\|_{H^1(\mathbb{R})}^2.\]

Now, we can establish our estimate through a Taylor expansion of $\Lag_\xi^h(u)$ around $\eta_\xi^h$. The first order term vanishes since $\eta_\xi^h$ is a critical point of  $\Lag_\xi^h$. The second order term is  controlled by applying the coercivity estimate of $\diff^2 \Lag_\xi^h$ (see \eqref{est_Lyap}),
\begin{align*}
&\Lag_\xi^h(u) - \Lag_\xi^h(\eta_\xi^h) \\
&\geq \diff^2 \Lag_\xi^h(\eta_\xi^h)(u-\eta_\xi^h,u-\eta_\xi^h) - M \|u-\eta_\xi^h\|_{H^1(\mathbb{R})}^3 \\
&=\diff^2 \Lag_\xi^h(\eta_\xi^h)(v-\eta_\xi^h,v-\eta_\xi^h) - \diff^2 \Lag_\xi^h(\eta_\xi^h)(u-v,u-v) + 2  \diff^2 \Lag_\xi^h(\eta_\xi^h)(u-\eta_\xi^h,u-v) - M \|u-\eta_\xi^h\|_{H^1(\mathbb{R})}^3 \\
&\geq \alpha \| v-\eta_\xi^h \|_{H^1(\mathbb{R})}^2  - \|u-\eta_\xi^h\|_{H^1(\mathbb{R})}^3 \left( M \left( \frac{2M}{\epsilon^2} \right)^2 \| u -\eta_\xi^h\|_{H^1(\mathbb{R})}  + 2 M \frac{2M}{\epsilon^2} + M   \right)\\
&= \alpha \| u-\eta_\xi^h \|_{H^1(\mathbb{R})}^2 + \alpha \| v-u \|_{H^1(\mathbb{R})}^2 -   2\alpha \langle v-u , u-\eta_\xi^h \rangle_{H^1(\mathbb{R})} \\
&   \hspace{5cm} - \|u-\eta_\xi^h\|_{H^1(\mathbb{R})}^3 \left( \left( \frac{2M^2}{\epsilon^2} \right)^2 \| u -\eta_\xi^h\|_{H^1(\mathbb{R})}  +  \frac{4M^2}{\epsilon^2} + M   \right)\\
&\geq \| u-\eta_\xi^h \|_{H^1(\mathbb{R})}^2 \left[ \alpha - \| u-\eta_\xi^h \|_{H^1(\mathbb{R})}  \left(  2\alpha \frac{2M}{\epsilon^2} +  \left( \frac{2M^2}{\epsilon^2} \right)^2 \| u -\eta_\xi^h\|_{H^1(\mathbb{R})}  +  \frac{4M^2}{\epsilon^2} + M   \right) \right] \\
&\geq \| u-\eta_\xi^h \|_{H^1(\mathbb{R})}^2 \left[ \alpha - r  \left(  \alpha \frac{4 M}{\epsilon^2} +  \left( \frac{2M^2}{\epsilon^2} \right)^2  +  \frac{4M^2}{\epsilon^2} + M   \right) \right].
\end{align*}
Consequently, to prove the Theorem, we just need to choose
\[ r < \frac{\alpha}2 \left(  \alpha \frac{4 M}{\epsilon^2} +  \left( \frac{2M^2}{\epsilon^2} \right)^2  +  \frac{4M^2}{\epsilon^2} + M   \right)^{-1}.\] 
\end{proof}

The previous lemma provides a stability control for the solutions of the homogeneous Hamiltonian system. To apply it, two strong assumptions are required: $u\in \Span(i\eta_\xi^h,\partial_x \eta_\xi^h)$ and $\| u \|_{L^2(\mathbb{R})}^2 = \| \eta_\xi^h  \|_{L^2(\mathbb{R})}^2$. If $u$ is close enough to $\eta_\xi^h$ there are two classical tricks to get these assumptions. To fulfill the first condition, the idea is to apply a small gauge transform and a small advection to $u$. We focus on this problem in the two following Lemmas. To satisfy the second assumption, the idea is to modify $\xi_1$. It is the object of the last Theorem of this section.

When $\eta_\xi^h$ is well defined through Theorem  \ref{thm_ExTW}, for any $v\in BL^2_h$, we define the matrix $A_{\xi,h}[v]$ by
\begin{equation}
\label{def_A}
A_{\xi,h}[v] :=  \begin{pmatrix}  \langle i\eta_\xi^h , iv  \rangle_{L^2(\mathbb{R})} & -\langle i\eta_\xi^h , \partial_x v   \rangle_{L^2(\mathbb{R})} 
\\ \langle \partial_x \eta_\xi^h , iv   \rangle_{L^2(\mathbb{R})} & -\langle \partial_x \eta_\xi^h , \partial_{x} v   \rangle_{L^2(\mathbb{R})} \end{pmatrix}. 
\end{equation}

We will explain later why this matrix is very useful, but first we give a technical Lemma.
\begin{lem}
\label{lem_matrix_inv}
Let $h_0,C,\rho,\alpha>0$ be the constants given by Theorem \ref{thm_ExTW}. 
There exists $h_1<h_0$, $M>0$ and $\delta>0$ such that for all $h<h_1$, all $\xi \in \Omega$ and all $v\in BL^2_h$ with $\| v- \eta_\xi^h\|_{H^1(\mathbb{R})}<\delta$, $A_{\xi,h}[v]$ is invertible and $\|(A_{\xi,h}[v])^{-1}\|_{\infty}\leq M$.
\end{lem}
\begin{proof}
Let $h<h_0$, $\xi \in \Omega$ and $v \in BL^2_h$.
Since $v\mapsto  A_{\xi,h}[v]$ is a linear map we have
\begin{equation}
\label{dec_lin}
A_{\xi,h}[v] = A_{\xi,h}[\eta_\xi^h]+ A_{\xi,h}[v-\eta_\xi^h].
\end{equation}
However, since $\| \eta_\xi^h - \psi_\xi \|_{H^1(\mathbb{R})} \leq Ch^2$, $A_{\xi,h}[\eta_\xi^h]$ converges to $G_\xi$, uniformly with respect to $\xi\in \Omega$, as $h$ goes to $0$,  where
\[ G_\xi = \begin{pmatrix} \| \psi_\xi\|_{L^2(\mathbb{R})}^2 & -\langle i\psi_\xi ,\partial_x \psi_\xi   \rangle_{L^2(\mathbb{R})} 
\\ \langle i\psi_\xi, \partial_x \psi_\xi    \rangle_{L^2(\mathbb{R})} & - \| \partial_x \psi_\xi\|_{L^2(\mathbb{R})}^2 \end{pmatrix}.   \]
Applying Cauchy-Schwarz inequality, we have
\[ \det G_\xi = \langle i\psi_\xi, \partial_x \psi_\xi    \rangle_{L^2(\mathbb{R})}^2 -   \| \psi_\xi\|_{L^2(\mathbb{R})}^2 \| \partial_x \psi_\xi\|_{L^2(\mathbb{R})}^2 \leq 0.   \]
But the case of equality is excluded since $\psi_\xi $ is not a plane wave (i.e. $\Span(i\psi_\xi, \partial_x \psi_\xi)$ is a free family). So $G_\xi $ is an invertible matrix. As $\xi \mapsto G_\xi$ is a continuous map on $\overline{\Omega}$, there exists $M>0$ such that for all $\xi \in \Omega$
\[ \|G_\xi ^{-1} \|_{\infty} \leq \frac{M}2. \]
As $A_{\xi,h}[\eta_\xi^h]$ converges to $G_\xi$ when $h \to 0$, there exists $h_1<h_0$ such that for all $h<h_1$ and $\xi \in \Omega$, $A_{\xi,h}[\eta_\xi^h]$ is invertible and
\[ \| (A_{\xi,h}[\eta_\xi^h])^{-1}\|_{\infty} \leq M. \]
Applying the linear decomposition \eqref{dec_lin}, we have
\[ A_{\xi,h}[v] =   A_{\xi,h}[\eta_\xi^h](I_2 + (A_{\xi,h}[\eta_\xi^h])^{-1} A_{\xi,h}[v-\eta_\xi^h]).  \]
However, since $\eta_\xi^h$ is bounded in $H^1(\mathbb{R})$ uniformly with respect to $\xi$ and $h$, there exists $\delta >0$ such that for all $\xi \in \Omega$ and all $h<h_1$, we have
\[ \| (A_{\xi,h}[\eta_\xi^h])^{-1} A_{\xi,h}[v-\eta_\xi^h] \|_{\infty} < \frac1{2\delta} \|v-\eta_\xi^h\|_{H^1(\mathbb{R})}.  \]
Consequently, if $\|v-\eta_\xi^h\|_{H^1(\mathbb{R})}\leq \delta$ then $ A_{\xi,h}[v]$ is invertible and the norm of its invert is bounded by $2M$.
\end{proof}

\begin{lem}
\label{lem_why_orb}
There exists $\lambda,\delta>0$ and $h_1<h_0$, such that for all $\xi \in \Omega$, $h<h_1$, $v\in BL^2_h$, if $\| v- \eta_\xi^h \|_{H^1(\mathbb{R})}< \delta$ then there exists $\gamma,\xO \in \mathbb{R}$ such that 
\[ \max(|\gamma|,|\xO|)\leq \lambda \| v- \eta_\xi^h \|_{H^1(\mathbb{R})} \quad \textrm{ and } \quad e^{i\gamma}v(\,\cdot\, -\xO) -\eta_\xi^h \in \Span(i\eta_\xi^h,\partial_x \eta_\xi^h)^{\perp_{L^2}}. \]
\end{lem}
\begin{proof}
For this proof, we introduce a notation. If $\gamma,\xO\in \mathbb{R}$ and $v:\mathbb{R}\to \mathbb{R}$ then
\[ T_{\gamma,\xO} v := e^{i\gamma} v(\,\cdot\, -\xO) \]
Let $v\in BL^2_h$. We are going to apply the inverse function Theorem \ref{Thm_Inv_loc} to the following function
\[ g_{\xi,h}^v : \left\{  \begin{array}{cccc}     \mathbb{R}^2 & \to & \mathbb{R}^2 \\
						\begin{pmatrix}
						\gamma \\
						\xO
						\end{pmatrix}
						 & \mapsto & \begin{pmatrix}
						  \langle i\eta_\xi^h , T_{\gamma,\xO}v-\eta_\xi^h   \rangle_{L^2(\mathbb{R})} \\ \langle \partial_x \eta_\xi^h , T_{\gamma,\xO}v-\eta_\xi^h   							\rangle_{L^2(\mathbb{R})}
						 \end{pmatrix}
						   \end{array} \right. .  \]
$g_{\xi,h}^v$ is clearly a $C^1$ function whose Jacobian matrix is given by
\[ {\rm J} g_{\xi,h}^v (\gamma,\xO) = A_{\xi,h}[T_{\gamma,\xO} v]. \]
Applying Lemma \ref{lem_matrix_inv}, we can find $h_1<h_0$, $\delta>0$ and $M>0$ such that if $h<h_1$ and $\|v-\eta_\xi^h\|_{H^1(\mathbb{R})}< \delta$ then $ {\rm J} g_{\xi,h}^v (0,0)$ is invertible and its norm is smaller than $M$.
We want to prove that ${\rm J} g_{\xi,h}^v $ is Lipschitz uniformly with respect to $\xi,h,v$. In fact, since it is a $C^1$ function, we just need to control its derivative. Using integration by parts, there exists a constant $\kappa>0$ such that for all $\xO,\gamma\in \mathbb{R}$ we have
\[ \|\diff {\rm J} g_{\xi,h}^v(\gamma,\xO)\|_{\Lag(\mathbb{R}^2;M_2(\mathbb{R}^2))} \leq \kappa \|\eta_\xi^h\|_{H^2(\mathbb{R})} \| T_{\gamma,\xO} v\|_{H^1(\mathbb{R})}  = \kappa \|\eta_\xi^h\|_{H^2(\mathbb{R})} \| v\|_{H^1(\mathbb{R})}.  \]
But, applying the result of elliptic regularity (Theorem \ref{thm_anareg}), $ \|\eta_\xi^h\|_{H^2(\mathbb{R})} $ is bounded in $H^2(\mathbb{R})$ uniformly with respect to $\xi\in \Omega$ and $h<h_0$. So, there exists $k>0$ such that for all $\xi \in \Omega$, $h<h_0$ and $v\in BL^2_h$ with $\|v-\eta_\xi^h \|_{H^1(\mathbb{R})}< \delta$, we have
\[ \|\diff {\rm J} g_{\xi,h}^v(\gamma,\xO)\|_{\Lag(\mathbb{R}^2;M_2(\mathbb{R}^2))} \leq  k.  \]

Now, we apply the inverse function theorem \ref{Thm_Inv_loc} to $g_{\xi,h}^v$ and we obtain some constants $\lambda>0$ and $r>0$, such that for all $h<h_1$, $\xi \in \Omega$ and $v\in BL^2_h$ with  $\|v-\eta_\xi^h \|_{H^1(\mathbb{R})}\leq R$,
\[   \forall \nu \in \mathbb{R}^2, \quad \ |\nu| \leq r \ \Rightarrow \ \exists \gamma,\xO \in \mathbb{R}, \  g_{\xi,h}^v(\gamma,\xO) = g_{\xi,h}^v(0,0)+\nu \quad \textrm{ and }\quad  \max(|\gamma|,|\xO|) \leq \lambda |\nu|.\]

To prove the lemma, we would like to choose $\nu = -g_{\xi,h}^v(0,0)$ small enough. But since $\eta_\xi^h$ is uniformly bounded in $H^1(\mathbb{R})$, there exists a constant $K>0$ such that for all $h<h_0$, $v\in BL^2_h$, $\xi \in \Omega$,
\[  |g_{\xi,h}^v(0,0)| \leq K \| \eta_\xi^h - v\|_{H^1(\mathbb{R})}.\]
So, if $\| \eta_\xi^h - v\|_{H^1(\mathbb{R})}\leq \frac{r}{K}$, we can choose $\nu = -g_{\xi,h}^v(0,0)$ and the lemma is proven.
\end{proof}

In the following Theorem, we focus on a change of variable. Usually, NLS traveling waves are not indexed by $\xi$ but by their $L^2$ norm and their momentum. It would be possible to do the same here. Here, we prove that it is possible to index them by their $L^2$ norm and their speed of advection (i.e. $\xi_2$).
\begin{theo}
\label{thm_change_variable}
Let $h_0,C,\rho,\alpha>0$ be the constants given by Theorem \ref{thm_ExTW} and let $\widetilde{\Omega}$ be a relatively compact open subset of $\Omega$. Then there exist $h_1<h_0,\delta>0,k>0$ such that for all $h<h_1$, for all $\xi \in \widetilde{\Omega}$ and for all $u\in BL^2_h$, if $\|u - \eta_{\xi}^h \|_{H^1(\mathbb{R})}< \delta$ then there exists $\zeta \in \Omega$ such that
\begin{equation}
\left\{ \begin{array}{lll} \xi_2 &=& \zeta_2 \\
				\|  \eta_{\zeta}^h \|_{L^2(\mathbb{R})}^2 &=& \| u \|_{L^2(\mathbb{R})}^2
\end{array}\right.
\qquad\mbox{and} \qquad
 |\zeta-\xi | \leq k | \|  \eta_\xi^h \|_{L^2(\mathbb{R})}^2 -  \| u \|_{L^2(\mathbb{R})}^2 |.
\end{equation}
\end{theo}
\begin{proof}
From the definition of $\psi_\xi$ (see \eqref{def_CTW}), we observe that for all $\xi \in \Omega$,
\[ \| \psi_\xi \|_{L^2(\mathbb{R})}^2 = m_{\xi} \| \psi_{1,0} \|_{L^2(\mathbb{R})}^2 = 4 m_{\xi} = 4 \sqrt{\xi_1^2 - \left(  \frac{\xi_2}2 \right)^2}.\]
Consequently, there exists $\beta>0$ such that for all $\xi \in \Omega$,
\[ \partial_{\xi_1} \| \psi_\xi \|_{L^2(\mathbb{R})}^2 = \frac2{m_\xi} \geq 2\beta. \]
Let $h<h_0$. Applying Theorem \ref{thm_ExTW}, we know that $\xi\mapsto \eta_\xi^h$ is a $C^1$ approximation of $\xi \mapsto \psi_\xi$ up to an second order error term.
Consequently, we have
\begin{align*}
| \partial_{\xi_1} \| \psi_\xi \|_{L^2(\mathbb{R})}^2 - \partial_{\xi_1} \| \eta_\xi^h \|_{L^2(\mathbb{R})}^2  |
&= 2 | \langle \partial_{\xi_1} \psi_\xi - \partial_{\xi_1} \eta_\xi^h  , \psi_\xi \rangle_{L^2(\mathbb{R})} +  \langle \partial_{\xi_1} \eta_\xi^h , \psi_\xi - \eta_\xi^h \rangle_{L^2(\mathbb{R})}  | \\
&\leq 2 C h^2 \left(  \| \psi_\xi  \|_{L^2(\mathbb{R})} + \|  \partial_{\xi_1} \eta_\xi^h \|_{L^2(\mathbb{R})}  \right)\\
&\leq 2 C h^2 \left(  \| \psi_\xi  \|_{L^2(\mathbb{R})} + Ch^2 + \|  \partial_{\xi_1} \psi_\xi \|_{L^2(\mathbb{R})}  \right) \\
&\leq 2 C h^2 \sup_{\xi \in \Omega} \left(  \| \psi_\xi  \|_{L^2(\mathbb{R})} + C h_0^2 + \|  \partial_{\xi_1} \psi_\xi \|_{L^2(\mathbb{R})}  \right) \\
&=: M h^2.
\end{align*}
Let $h_1 = \min(h_0 , \beta \sqrt{M})$. If $h<h_0$ and $\xi \in \Omega$, we have
\[ \partial_{\xi_1} \| \eta_\xi^h \|_{L^2(\mathbb{R})}^2 \geq \beta. \]

Since $\widetilde{\Omega}$ is relatively compact open subset of $\Omega$, there exists $r>0$ such that
\[ \widetilde{\Omega} + \overline{B_{\mathbb{R}^2}(0,r)} \subset \Omega.\]
Let $\xi \in \widetilde{\Omega}$, $h<h_1$ and let $g$ be the following function
\[ g: \left\{  \begin{array}{cccc} [\xi_1 - r,\xi_1 + r] & \to & \mathbb{R} \\
						\zeta_1 & \mapsto &   \| \eta_{\zeta_1,\xi_2}^h \|_{L^2(\mathbb{R})}^2.
\end{array}\right. \]
Since $g$ is a continuous map, we have
\begin{equation}
\label{surj_mass}
  [ \| \eta_{\xi_1 - r,\xi_2}^h \|_{L^2(\mathbb{R})}^2 , \| \eta_{\xi_1 + r,\xi_2}^h \|_{L^2(\mathbb{R})}^2]  \subset g( [\xi_1 - r,\xi_1 + r] ).
\end{equation}
But applying the mean value equality, we have
\[  \| \eta_{\xi_1 - r,\xi_2}^h \|_{L^2(\mathbb{R})}^2 < \| \eta_\xi^h \|_{L^2(\mathbb{R})}^2 - \beta r < \| \eta_\xi^h \|_{L^2(\mathbb{R})}^2 + \beta r< \| \eta_{\xi_1 + r,\xi_2}^h \|_{L^2(\mathbb{R})}^2.\]

Let $u\in BL^2_h$ be such that $\| u - \eta_\xi^h \|_{H^1(\mathbb{R})} < \delta$, where $\delta \in (0,1)$ is a positive constant that will be fixed later. Applying triangle inequality, we get
\begin{align*}
\left|  \| u\|_{L^2(\mathbb{R})}^2 - \| \eta_\xi^h\|_{L^2(\mathbb{R})}^2 \right| & \leq \delta (  \| u\|_{L^2(\mathbb{R})} + \| \eta_\xi^h\|_{L^2(\mathbb{R})} )\\
														&\leq \delta (\delta + 2 \| \eta_\xi^h\|_{L^2(\mathbb{R})}  ) \\
														&\leq \delta (1+ 2 \sup_{\xi \in \Omega, \ h<h_0}  \| \eta_\xi^h\|_{L^2(\mathbb{R})}  ) \\
														&= : \delta \kappa.
\end{align*}
So, choosing $\delta = \frac{\beta r}{\kappa}$, we deduce from \eqref{surj_mass} that there exists $\zeta_1 \in [\xi_1 - r,\xi_1 + r] $ such that
\[ \| u\|_{L^2(\mathbb{R})}^2 = g(\zeta_1) = \| \eta_\zeta^h \|_{L^2(\mathbb{R})}^2  ,\]
where $\zeta_2 := \xi_2.$
Applying the mean value equality, we obtain
\[ |\xi - \zeta| \leq \beta^{-1} \left|  \| u\|_{L^2(\mathbb{R})}^2 - \| \eta_\zeta^h \|_{L^2(\mathbb{R})}^2 \right|. \]
which proves the result. 
\end{proof}

\section{Control of the instabilities and modulation }
In the last section we have constructed approximate traveling waves $\eta_\xi^h$. In order to prove Theorem \ref{Thm_DTW}, we now study the dynamics of DNLS around these approximate traveling waves.

We are going to use many results established in the previous section about $\eta_\xi^h$ and its properties. In a first paragraph, we summarize the results that will be useful and fix most of the constants.

\underline{ Step 1: variational properties around the equilibria}

Let $\widetilde{\Omega}$ be a relatively compact open subset of $\left\{ \xi \in \mathbb{R}^2 \ | \ \xi_1>\left( \frac{\xi_2}2 \right)^2 \right\}$ and $\Omega$  a relatively compact open subset of $ \widetilde{\Omega}$.
In the previous section, we have proven there exist some constants $h_0,\varepsilon,C,\rho>0$ and, for all $\xi \in \widetilde{\Omega}$ and all $h<h_0$, a function $\eta_\xi^h \in BL^2_h$ satisfying the following properties.
\begin{itemize} 
\item From Theorem \ref{thm_ExTW}, $\eta_\xi^h$ is a critical point of $\Lag_\xi^h$ and it is an approximation of $\psi_\xi$
\[ \| \eta_\xi^h - \psi_\xi\|_{H^1(\mathbb{R})} \leq C h^2.\]
\item From Theorem \ref{thm_anareg}, $\eta_\xi^h$ is regular function
\begin{equation}
\label{summ_reg}
 \forall \omega \in \mathbb{R}, \ |\widehat{\eta_\xi^h}(\omega)| \leq C e^{- \varepsilon |\omega|}. 
\end{equation}
Consequently, we also have $\|\widehat{\eta_\xi^h}\|_{H^3(\mathbb{R})}\leq C$.
\item From Lemma \ref{lem_orb_stab}, if $u\in BL^2_h \cap \Span(i\eta_\xi^h,\partial_x \eta_\xi^h)^{\perp_{L^2}}$, $\| u\|_{L^2(\mathbb{R})}^2 = \| \eta_\xi^h\|_{L^2(\mathbb{R})}^2$ and $\|u - \eta_\xi^h\|_{H^1(\mathbb{R})} \leq \rho$ then 
\begin{equation}
\label{summ_est_lyap}
\frac1C \| u -\eta_\xi^h \|_{H^1(\mathbb{R})}^2 \leq  \Lag_\xi^h(u) -\Lag_\xi^h(\eta_\xi^h).
\end{equation}
\item From Theorem \ref{thm_change_variable}, if $u \in BL^2_h(\mathbb{R})$, $\xi \in \Omega$ and $\|u - \eta_\xi^h\|_{H^1(\mathbb{R})} \leq \rho$ then there exists $\zeta\in \widetilde{\Omega}$ such that
\begin{equation}
\label{summ_fix_mass}
\left\{ \begin{array}{lll} \xi_2 &=& \zeta_2 \\
				\|  \eta_{\zeta}^h \|_{L^2(\mathbb{R})}^2 &=& \| u \|_{L^2(\mathbb{R})}^2
\end{array}\right.
\end{equation}
and (using regularity of $\xi \mapsto \eta_{\xi}^h$ uniformly with respect to $h$, see Theorem \ref{thm_ExTW})
\begin{equation}
\label{summ_est_fix_mass}
 |\zeta-\xi | +  \|u - \eta_{\zeta}^h \|_{H^1(\mathbb{R})}  \leq C \|u - \eta_{\xi}^h \|_{H^1(\mathbb{R})}.
\end{equation}
\item From Lemma  \ref{lem_why_orb}, for all $u \in BL^2_h(\mathbb{R})$, if $\|u - \eta_\xi^h\|_{H^1(\mathbb{R})} \leq \rho$ then there exists $\gamma,\xO\in \mathbb{R}$ such that
\begin{equation}
\label{summ_why_orb}
\max(|\gamma|,|\xO|)\leq C \| u- \eta_\xi^h \|_{H^1(\mathbb{R})} \quad \textrm{ and } \quad e^{i\gamma}u(\, \cdot\, -\xO) -\eta_\xi^h \in \Span(i\eta_\xi^h,\partial_x \eta_\xi^h)^{\perp_{L^2}}.
\end{equation}
\item From Lemma \ref{lem_matrix_inv}, for all $u \in BL^2_h(\mathbb{R})$, if $\|u - \eta_\xi^h\|_{H^1(\mathbb{R})} \leq \rho$ and  $A_{h,\xi}[u]$ is the matrix defined in \eqref{def_A} then
\begin{equation}
\label{summ_lem_matrix_inv}
A_{h,\xi}[u] \textrm{ is is invertible and } \| (A_{h,\xi}[u])^{-1} \|_1 \leq C.
\end{equation} 
\item From Lemma \ref{lem_d2_unif_bound} and Lemma \ref{est_lips}, for all $u \in BL^2_h(\mathbb{R})$, if $\|u - \eta_\xi^h\|_{H^1(\mathbb{R})} \leq \rho$ then
\begin{equation}
\label{summ_d2_unif_bound}
 \forall v,w\in BL^2_h, \ \quad  \left| \diff^2 \Lag_\xi^h (u)(v,w) \right| \leq C \|v\|_{H^1(\mathbb{R})} \|w\|_{H^1(\mathbb{R})} .
 \end{equation}
\end{itemize}

We finish this paragraph by a remark. In Theorem \ref{Thm_DTW}, we compare a solution $\ug$ of DNLS with some discretizations of $\eta_\xi^h$ using discrete Sobolev norms. However, as we explain in Lemma \ref{lem_compare_norms}, it is equivalent to compare directly the Shannon interpolation $u$ of the discrete solution with $\eta_\xi^h$ using continuous Sobolev norms.

\underline{ Step 2: Lyapunov estimation and modulation}

Let $r>0$ be a positive constant independent of $\xi$ and $h$ that will be determined at the end of this paragraph.
Recall that for  $v:\mathbb{R} \to \mathbb{R}$ we have 
\[  \forall x\in \mathbb{R},\ \quad  T_{\gamma,\xO} v(x) := e^{i\gamma}v(x - \xO).\]
and note that  $T_{\gamma,\xO}^{-1} = T_{-\gamma,-\xO}$.
Let $u_0 \in BL^2_h$ be such that $\delta(0) =\|  u_0 - T_{\gamma_0,\xO_0} \eta_\xi^h\|_{H^1(\mathbb{R})}< r$ where $\xi \in \Omega$,  $\xO_0,\gamma_0\in \mathbb{R}$. Let $u$ be the solution of DNLS in $BL^2_h$ (see Lemma \ref{lem_DNLS_BL2})  such that $u(0) = u_0$.

Assume that $r<\rho$. Applying \eqref{summ_fix_mass} and \eqref{summ_est_fix_mass}, there exists $\zeta \in \widetilde{\Omega}$ such that 
\[ \left\{ \begin{array}{lll} \xi_2 &=& \zeta_2 \\
				\|  \eta_{\zeta}^h \|_{L^2(\mathbb{R})}^2 &=& \| u_0 \|_{L^2(\mathbb{R})}^2
\end{array}\right. \qquad \mbox{and} \qquad |\zeta-\xi | +  \|u_0 - T_{\gamma_0,\xO_0} \eta_\xi^h \|_{H^1(\mathbb{R})} \leq C \delta(0). \]
Consequently, we have
\[  \| \eta_{\xi}^h - \eta_{\zeta}^h \|_{H^1(\mathbb{R})} \leq (1+C) \delta(0).\]
Now, assume that $Cr < \rho$, then applying \eqref{summ_why_orb}, there exist $\delta_\gamma,\delta_{\xO}\in \mathbb{R}$ such that
\begin{equation}
\label{prokofiev} 
\left\{
\begin{array}{rcl}
\theta_0 &=& \gamma_0+\delta_\gamma \\
p_0 &=& \xO_0+\delta_{\xO}
\end{array}
\right.  
\quad\mbox{with}\quad 
\max(|\delta_\gamma|,|\delta_{\xO}|) \leq C^2 \delta(0)\quad  \textrm{ and } \quad  T_{\theta_0,p_0}^{-1} u_0 \in \Span(i \eta_{\zeta}^h, \partial_x \eta_\zeta^h)^{\perp_{L^2}}.  
\end{equation}
We would like to get some functions $\theta,p \in C^1(\mathbb{R}_+)$ such that as long as $u(t)$ is close to the orbit of $\eta_\zeta^h$ (up to gauge transform and advection), we have $ T_{\theta(t),p(t)}^{-1} u(t) \in \Span(i \eta_{\zeta}^h, \partial_x \eta_\zeta^h)^{\perp_{L^2}}$.
We are going to construct them by solving a differential equation. Taking a time derivative,  if such functions exist they have to  satisfy 
\begin{equation}
\label{edo_modulation}
 A_{\zeta,h}[T_{\theta(t),p(t)}^{-1} u(t)] \begin{pmatrix} \dot{\theta}(t) \\
										 \dot{p}(t)
\end{pmatrix}  = \begin{pmatrix} \langle T_{\theta(t),p(t)}^{-1} \partial_t u(t) , i \eta_{\zeta}^h\rangle_{L^2(\mathbb{R})} \\
						\langle T_{\theta(t),p(t)}^{-1} \partial_t u(t) , \partial_x \eta_{\zeta}^h\rangle_{L^2(\mathbb{R})}
\end{pmatrix} .
\end{equation}
We would like to solve the Cauchy problem associated with this ordinary differential equation with $\theta(0)=\theta_0$ and $p(0) = p_0$. 
Note that all the terms depend smoothly on $t,p(t),\theta(t)$, hence to get the existence of a local solution, 
we need to invert $ A_{\zeta,h}[T_{\theta(t),p(t)}^{-1} u(t)] $. Using the regularity of $\eta_{\zeta}^h$ (see \eqref{summ_reg}), we have
\[ \|  u_0 - T_{\theta_0,p_0}\eta_{\zeta}^h\|_{H^1(\mathbb{R})} \leq C^3 \delta(0). \]
Assuming that $C^3 r < \rho$, we get from \eqref{summ_lem_matrix_inv} that $A_{\zeta,h}[T_{\theta_0,p_0}^{-1} u_0]$ is invertible and 
\[ \| (A_{\zeta,h}[ T_{\theta_0,p_0}^{-1} u_0])^{-1} \|_1 \leq C. \]
So (applying, for example, Cauchy-Lipschitz theorem or the implicit functions theorem), there exist $T_{\max}\in(0,\infty]$ and a solution $\theta,p\in C^1([0,T_{\max}))$ of \eqref{edo_modulation} on $[0,T_{\max})$ such that
\begin{itemize}
\item $\theta(0)=\theta_0$ and  $p(0) = p_0$,
\item for all $t\in [0,T_{\max})$, $ A_{\zeta,h}[T_{\theta(t),p(t)}^{-1} u(t)] $ is invertible,
\item $\displaystyle \lim_{t \to T_{\max}} |\theta(t)|+ |p(t)| + \| ( A_{\zeta,h}[T_{\theta(t),p(t)}^{-1} u(t)] )^{-1}\|_1   = \infty$ 
\end{itemize}

We would like to prove that while $\|  u(t) - T_{\gamma(t),\xO(t)} \eta_\xi^h\|_{H^1(\mathbb{R})}< r$, with $\gamma = \theta - \delta_{\gamma}  $ and $\xO = p - \delta_{\xO} $ where $\delta_\gamma$ and $\delta_{\xO}$ are given in \eqref{prokofiev},  the last condition is not satisfied and so $\gamma(t)$ and $\xO(t)$ are well defined. 
This is done by the following Lemma,  whose proof is given in Section \ref{sec_proof_lem_modulation_constructor} of the Appendix.

\begin{lem}
\label{lem_modulation_constructor}
There exist $\gamma,\xO\in C^1(\mathbb{R}_+)$ such that $\gamma(0) = \gamma_0$, $\xO(0)=\xO_0$ and if $T>0$ satisfies
\[ \forall t\in (0,T), \ \quad  \|  u(t) - T_{\gamma(t),\xO(t)} \eta_\xi^h\|_{H^1(\mathbb{R})}< r,\]
then $T<T_{\max}$ and $\gamma = \theta - \delta_{\gamma}  $, $\xO = p - \delta_{\xO} $ on $(0,T)$, where $\delta_\gamma$ and $\delta_{\xO}$ are defined in \eqref{prokofiev}. 
\end{lem}

From now on, we consider the functions $\gamma,\xO$ given by Lemma \ref{lem_modulation_constructor} and $T>0$ satisfying the bootstrap condition
 \[ \forall t\in (0,T), \ \quad \delta(t) :=\|  u(t) - T_{\gamma(t),\xO(t)} \eta_\xi^h\|_{H^1(\mathbb{R})}< r.\] 
By construction, we have
\begin{align*}
\|  u(t) - T_{\theta(t),p(t)} \eta_{\zeta}^h\|_{H^1(\mathbb{R})} &\leq  \| u(t) - T_{\gamma(t),\xO(t)} \eta_\xi^h \|_{H^1(\mathbb{R})} + \| \eta_{\xi}^h - T_{\delta_\gamma,\delta_{\xO}} \eta_\xi^h \|_{H^1(\mathbb{R})} + \| \eta_{\zeta}^h - \eta_\xi^h \|_{H^1(\mathbb{R})} \\
												&\leq \delta(t) + C^3 \delta(0) + (1+C) \delta(0) \\
												&< (2+C+C^3)r.
\end{align*}
We assume that $(2+C+C^3)r \leq \rho$. Since $\| u\|_{L^2(\mathbb{R})}^2$ is a constant of the motion, we have  $\|u(t)\|_{L^2(\mathbb{R})}^2 = \| \eta_{\zeta}^h \|_{L^2(\mathbb{R})}^2$. Furthermore, by construction $T_{\theta(t),p(t)}^{-1}u \in \Span(i \eta_{\zeta}^h,\partial_x \eta_{\zeta}^h)^{\perp_{L^2}}$, so we can apply \eqref{summ_est_lyap} to get the Lyapunov control of the stability
\begin{equation}
\label{true_coer}
\frac1{C} \|  u(t) - T_{\theta(t),p(t)} \eta_{\zeta}^h\|_{H^1(\mathbb{R})}^2 \leq \Lag_{\zeta}^h(u(t)) - \Lag_{\zeta}^h(\eta_{\zeta}^h).
\end{equation}

To be rigorous, we can verify our assumptions on $r$ and observe that $r = \frac{\rho}{2+C+C^3}$ is a possible choice.

\underline{Step 3: Estimation of $\delta(t)$}

Usually, when we apply the energy-momentum method, the Lagrange function is a constant of the motion of DNLS. An estimate of the form \eqref{true_coer} allows to control
$\|  u(t) - T_{\theta(t),p(t)} \eta_{\zeta}^h\|_{H^1(\mathbb{R})}^2$ by $\Lag_{\zeta}^h(u_0) - \Lag_{\zeta}^h(\eta_{\zeta}^h)$. This latter quantity can be controlled by 
using a Taylor expansion
\begin{align*}
\Lag_{\zeta}^h(u_0) - \Lag_{\zeta}^h(\eta_{\zeta}^h) &= \Lag_{\zeta}^h(T_{\theta_0,p_0}^{-1}  u_0) - \Lag_{\zeta}^h(\eta_{\zeta}^h) \\
									      &\leq \frac12 \sup_{\| v -  \eta_{\zeta}^h \|_{H^1(\mathbb{R})}\leq \rho} \left| \diff^2 \Lag_{\zeta}^h(v)(T_{\theta_0,p_0}^{-1}  u_0 - \eta_{\zeta}^h) \right| \\
									      &\leq \frac{C}2 \|  u_0 - T_{\theta_0,p_0} \eta_{\zeta}^h\|_{H^1(\mathbb{R})}^2,
\end{align*}
where the last estimate is given by \eqref{summ_d2_unif_bound}.

In our case, because of the aliasing terms,  $\Lag_\zeta^h(u(t))$ is not a constant of the motion. So we have to control its variations. Let $t<T$, since $H_{\rm DNLS}^h(u(t))$ and $\|u(t)\|_{L^2(\mathbb{R})}^2$ are constant of the motion, applying the formula of Lemma \ref{lem_eq_HDNLS}, we obtain the following decomposition
\begin{align}
\nonumber
  \Lag_\zeta^h(u(t)) -  \Lag_\zeta^h(u(0)) &= H_{\rm DNLS}^h(u(t)) -H_{\rm DNLS}^h(u(0)) + \frac{\zeta_1}2 \left( \|u(t)\|_{L^2(\mathbb{R})}^2 - \|u(0)\|_{L^2(\mathbb{R})}^2 \right) \\ \nonumber
  & - \frac12\int_{\mathbb{R}} \cos \left(  \frac{2\pi x}h \right) (|u(t,x)|^4 - |u(0,x)|^4 ) \dx  \\ \nonumber
  							     &+ \frac{\zeta_2}2 \left(  \langle i\partial_x u(t), u(t) \rangle_{L^2(\mathbb{R})}-\langle i\partial_x u(0), u(0) \rangle_{L^2(\mathbb{R})} \right)\\ \label{nougaro}
							     &= E_1(0) -  E_1(t) + \frac12E_2(t),
\end{align}
where
\[ E_1(t) = \frac12\int_{\mathbb{R}} \cos \left(  \frac{2\pi x}h \right) |u(t,x)|^4  \dx  \]
and
\[  E_2(t) =  \xi_2 \left(  \langle i\partial_x u(t), u(t) \rangle_{L^2(\mathbb{R})}-\langle i\partial_x u(0), u(0) \rangle_{L^2(\mathbb{R})} \right).  \]
Note that we write $\xi_2$ instead of $\zeta_2$ as these two numbers are equal by construction (see \eqref{summ_fix_mass}).

First, we explain how to bound $E_1(t)$. It can be decomposed as follow
\[ E_1(t) =  \frac14 \left( E_3(u(t)) + \overline{  E_3(u(t)) } \right),
\quad\mbox{with} \quad E_3(v) = \int_{\mathbb{R}} e^{\frac{2i \pi}h} |u(t,x)|^4 \dx. \]
Since $E_3$ is a $4-$homogeneous continuous function, its Taylor expansion is exact. So, we have
\begin{equation}
\label{Taylor_expansion}
 E_3(u(t)) = \sum_{j=0}^4 \frac1{j!} \diff^j E_3 (T_{\theta(t),p(t)} \eta_{\zeta}^h)  (\underbrace{ u(t) - T_{\theta(t),p(t)}\eta_{\zeta,h},\dots,  u(t) - T_{\theta(t),p(t)}\eta_{\zeta,h}}_{ j \textrm{ times}} )  .
\end{equation}
To control these derivatives, we use the following lemma.
\begin{lem}
\label{lem_estA}
 If $u_1,u_2,u_3,u_4\in  BL^2_h$ and
\[ M_h(u_1,u_2,u_3,u_4) = \int_{\mathbb{R}} e^{\frac{2i\pi x}h} u_1(x) u_2(x) u_3(x) u_4(x) \dx, \]
then we have
\[ |M_h (u_1,u_2,u_3,u_4)|\leq  \frac14 \sum_{\sigma\in S_4}  \|\widehat{u_{\sigma_1}} \mathbb{1}_{\omega \geq \frac{\pi}{3h}}\|_{L^2(\mathbb{R})} \|\widehat{u_{\sigma_2}} \mathbb{1}_{\omega \geq\frac{\pi}{3h}}\|_{L^2(\mathbb{R})} \|\widehat{u_{\sigma_3}}\|_{L^1(\mathbb{R})} \|\widehat{u_{\sigma_4}} \|_{L^1(\mathbb{R})}  .  \]
\end{lem}
\begin{proof}
We identify $M_h$ with a convolution product
\[ M_h (u_1,u_2,u_3,u_4) = \widehat{u_1}* \widehat{u_2}*\widehat{u_3}*\widehat{u_4} ( \frac{2\pi}{h} ). \]
But \it if the sum of four numbers, all smaller than $1$, is equals to $2$, then at least $2$ of them are larger than $\frac13$\rm. Consequently, since $\supp \widehat{u_j} \subset \left[ -\frac{\pi}h,\frac{\pi}h \right]$, it comes
\[  |M_h (u_1,u_2,u_3,u_4)|\leq \frac1{4} \sum_{\sigma\in S_4}  |\mathbb{1}_{\omega \geq  \frac{\pi}{3h}}\widehat{u_{\sigma_1}}|* |\mathbb{1}_{\omega \geq  \frac{\pi}{3h}} \widehat{u_{\sigma_2}}| *|\widehat{u_{\sigma_3}}|*|\widehat{u_{\sigma_4}}|  (  \frac{2\pi}{h} ) .   \]
Then, we conclude the proof using Young convolution inequalities.
\end{proof}
Applying this Lemma to estimate the terms of \eqref{Taylor_expansion} we obtain four types of contributions. 
\begin{itemize}
\item Applying \eqref{summ_reg} and defining $\ell = \frac{\pi \varepsilon}{3}$, we have
\[ \|\mathscr{F}[T_{\theta(t),p(t)} \eta_{\zeta}^h] \mathbb{1}_{\omega \geq \frac{\pi}{3h}}\|_{L^2(\mathbb{R})}^2 \leq C^2 \int_{\omega \geq \frac{\pi}{3h}} e^{-2\varepsilon \omega} \domega = \frac{C^2}{\varepsilon} e^{- 2\varepsilon \frac{\pi}{3h}} = \frac{C^2}{\varepsilon} e^{- \frac{2\ell}{h}  }. \]
\item Up to an universal constant $c>0$, we have
\[ \|\mathscr{F}[T_{\theta(t),p(t)} \eta_{\zeta}^h] \|_{L^1(\mathbb{R})} \leq c\,  C. \]
\item Up to an universal constant $c>0$, we have
\[    \|\mathscr{F}[u(t)-T_{\theta(t),p(t)} \eta_{\zeta}^h] \mathbb{1}_{\omega \geq \frac{\pi}{3h}}\|_{L^2(\mathbb{R})} \leq \frac{3 h}{\pi} \|  \mathscr{F}[u(t)-T_{\theta(t),p(t)} \eta_{\zeta}^h] \ \omega \|_{L^2(\mathbb{R})} \leq c h \|  u(t)-T_{\theta(t),p(t)} \eta_{\zeta}^h \|_{H^1(\mathbb{R})}.\]
\item Up to an universal constant $c>0$, we have
\[ \|\mathscr{F}[u(t)-T_{\theta(t),p(t)} \eta_{\zeta}^h] \|_{L^1(\mathbb{R})} \leq c \|  u(t)-T_{\theta(t),p(t)} \eta_{\zeta}^h \|_{H^1(\mathbb{R})}.\]
Sometimes, it is also useful to control it by $c \rho$.
\end{itemize}
With these estimates, we get a constant $M>0$ (depending only of $\varepsilon,C,\rho,h_0$) such that
\begin{equation}
\label{estimation_of_E3}
 |E_3(u(t))| \leq 2M  e^{- \frac{\ell}{h}  } + 2M  h^2  \|  u(t)-T_{\theta(t),p(t)} \eta_{\zeta}^h \|_{H^1(\mathbb{R})}^2  .
\end{equation}
So we deduce that
\begin{equation}
\label{estimation_of_E1}
 |E_1(t)| \leq M  e^{- \frac{\ell}{h}  } + h^2 M  \|  u(t)-T_{\theta(t),p(t)} \eta_{\zeta}^h \|_{H^1(\mathbb{R})}^2  .
\end{equation}

We show now how to control the term $E_2$ in \eqref{nougaro}. It is precisely the error generated by the default of invariance by advection. First, we give a more adapted expression of $E_2$:
\begin{multline*}
E_2(t) = \xi_2 \int_{0}^t  \partial_s   \langle i\partial_x u(s), u(s) \rangle_{L^2(\mathbb{R})}   \ds 
	= 2\xi_2 \int_{0}^t   \langle i\partial_x u(s), \partial_s   u(s) \rangle_{L^2(\mathbb{R})}   \ds \\
	= -4 \xi_2 \int_{0}^t   \langle \partial_x u(s), \cos \left( \frac{2\pi x}h \right) |u(s)|^2 u(s) \rangle_{L^2(\mathbb{R})}   \ds\\
	= - \xi_2 \frac{2\pi}h \int_0^t  \int_{\mathbb{R}} \sin \left( \frac{2\pi x}h \right)  |u(s,x)|^4 \dx \ds 
	= - \xi_2 \frac{\pi}h \int_0^t E_3(u(s)) - \overline{  E_3(u(s)) }  \ds.
\end{multline*}
Applying Estimate of $E_3(u(s))$ \eqref{estimation_of_E3}, we obtain
\[  |E_2(t)| \leq 4M \pi | \xi_2 | h \int_0^t    \frac{e^{- \frac{\ell}{h} } }{h^2} +   \|  u(s)-T_{\theta(s),p(s)} \eta_{\zeta}^h \|_{H^1(\mathbb{R})}^2 \ds.\]

Finally, we apply estimate \eqref{true_coer} and we get
\begin{align*}
& \frac1{C} \|  u(t)-T_{\theta(t),p(t)} \eta_{\zeta}^h \|_{H^1(\mathbb{R})}^2\\
 &\leq \Lag_\xi^h( u(t) ) - \Lag_\xi^h(  \eta_{\zeta}^h )  \\
 &=  \Lag_\xi^h( u(0) ) - \Lag_\xi^h(  \eta_{\zeta}^h ) + \Lag_\xi^h( u(t) ) - \Lag_\xi^h( u(0) )   \\
&=  \Lag_\xi^h( u(0) ) - \Lag_\xi^h(  \eta_{\zeta}^h ) + E_1(0) -E_1(t) + E_2(t) \\
&\leq \frac{C}2  \|  u(0)-T_{\theta(0),p(0)} \eta_{\zeta}^h \|_{H^1(\mathbb{R})}^2 + M  e^{- \frac{\ell}{h}  } + h^2 M  \|  u(t)-T_{\theta(t),p(t)} \eta_{\zeta}^h \|_{H^1(\mathbb{R})}^2 \\
&+ M  e^{- \frac{\ell}{h}  } + h^2 M  \|  u(0)-T_{\theta(0),p(0)} \eta_{\zeta}^h \|_{H^1(\mathbb{R})}^2 + 4M \pi | \xi_2 | h \int_0^t    \frac{e^{- \frac{\ell}{h} } }{h^2} +   \|  u(s)-T_{\theta(s),p(s)} \eta_{\zeta}^h \|_{H^1(\mathbb{R})}^2 \ds.
\end{align*}

So there exist some constants $h_1<h_0$, $c>0$ and $\lambda>0$  (depending only of $\varepsilon,C,\rho,h_0$) such that, for all $h<h_1$, we have
\[   \|  u(t)-T_{\theta(t),p(t)} \eta_{\zeta}^h \|_{H^1(\mathbb{R})}^2 \leq c e^{-\frac{\ell}{2h}} + c \|  u(0)-T_{\theta(0),p(0)} \eta_{\zeta}^h \|_{H^1(\mathbb{R})}^2 + 2\lambda h |\xi_2| \int_0^t   e^{- \frac{\ell}{2h} } +   \|  u(s)-T_{\theta(s),p(s)} \eta_{\zeta}^h \|_{H^1(\mathbb{R})}^2 \ds   .\]
Applying Gr\"onwall's lemma, we obtain the estimate 
\[ \|  u(t)-T_{\theta(t),p(t)} \eta_{\zeta}^h \|_{H^1(\mathbb{R})}^2 + e^{- \frac{\ell}{2h} } \leq e^{ 2\lambda |\xi_2| h t  }  \left[ (1+c) e^{-\frac{\ell}{2h}} + c \|  u(0)-T_{\theta(0),p(0)} \eta_{\zeta}^h \|_{H^1(\mathbb{R})}^2 \right] .\]
Now applying Minkowski inequality, we get
\[  \|  u(t)-T_{\theta(t),p(t)} \eta_{\zeta}^h \|_{H^1(\mathbb{R})}    \leq  \sqrt{1+c} \  e^{ \lambda |\xi_2| h t  }  \left[ e^{-\frac{\ell}{4h}} + \|  u(0)-T_{\theta(0),p(0)} \eta_{\zeta}^h \|_{H^1(\mathbb{R})} \right].\]

We want to deduce a bound on $\delta$ from this inequality. Applying the inequalities established in the previous paragraph, we have
\begin{align*}
 \|  u(0) - T_{\theta(0),p(0)} \eta_{\zeta}^h\|_{H^1(\mathbb{R})} &\leq  \| u(0) - T_{\gamma(0),\xO(0)} \eta_\xi^h \|_{H^1(\mathbb{R})} + \| \eta_{\xi}^h - T_{\delta_\gamma,\delta_{\xO}} \eta_\xi^h \|_{H^1(\mathbb{R})} + \| \eta_{\zeta}^h - \eta_\xi^h \|_{H^1(\mathbb{R})} \\
												&\leq \delta(0) + C^3 \delta(0) + (1+C) \delta(0)  .
\end{align*}
On the other hand, applying the same inequalities, we have
\begin{align*}
 \delta(t) = \| u(t) - T_{\gamma(t),\xO(t)} \eta_\xi^h \|_{H^1(\mathbb{R})}  &\leq   \|  u(t) - T_{\theta(t),p(t)} \eta_{\zeta}^h\|_{H^1(\mathbb{R})} + \| \eta_{\xi}^h - T_{\delta_\gamma,\delta_{\xO}} \eta_\xi^h \|_{H^1(\mathbb{R})} + \| \eta_{\zeta}^h - \eta_\xi^h \|_{H^1(\mathbb{R})} \\
												&\leq  \|  u(t) - T_{\theta(t),p(t)} \eta_{\zeta}^h\|_{H^1(\mathbb{R})} + C^3 \delta(0) + (1+C) \delta(0) .
\end{align*}
Consequently, we have proven our estimate:
\begin{equation}
 \delta(t) \leq \sqrt{1+c} \  e^{ \lambda |\xi_2| h t  }  \left[ e^{-\frac{\ell}{4h}} + \delta(0) \left(  2+C+C^3 + \frac{1+C+C^3}{\sqrt{1+c}}  \right) \right]. 
 \end{equation}
 
\begin{rem}
\label{rem_if_reg} We could get an other kind of estimate of $\delta(t)$ based on the high order Sobolev norms of $u(t)$. Indeed, if $n\in \mathbb{N}^*$, using Lemma \ref{lem_estA}, we have
\[ |E_3(u(t))| \lesssim \| \widehat{u}(\omega) \mathbb{1}_{|\omega|\geq \frac{\pi}{3h}}\|_{L^2(\mathbb{R})}^2 \lesssim  h^{2n} \| u(t) \|_{\dot{H}^n(\mathbb{R})}^2  .\]   
Applying this inequality for $E_2$ and realizing the same proof without applying Gr\"onwall's lemma, we get
\[ \delta(t) \lesssim \delta(0) + e^{-\frac{\ell}h}+ \sqrt{t|\xi_2|} h^{n - \frac12} \sup_{0<s<t} \| u(s) \|_{\dot{H}^n(\mathbb{R})}.\]
\end{rem} 
 
\underline{Step 4: Control of $\dot{\gamma}$ and $\dot{\xO}$}

The idea to obtain the estimate \eqref{modulation_control} is that $\xi$ is the solution of a perturbed linear equation whose $(\dot{\gamma},\dot{\xO})$ is a solution (i.e. \eqref{edo_modulation}).
We work with a fixed $t<T$. To simplify the notation, we assume that $\theta(t)=p(t)=0$. We introduce a notation : for $v\in BL^2_h$, we define
\begin{equation}
\label{chopin} b_{\zeta,h}[v] := \begin{pmatrix}   \langle  \Delta_h v + |v|^2 v , \eta_{\zeta}^h \rangle_{L^2(\mathbb{R})}        \\
						     -\langle \Delta_h v + |v|^2 v  , i\partial_x \eta_{\zeta}^h \rangle_{L^2(\mathbb{R})}  
			 \end{pmatrix} .\end{equation}
With this formalism, equation \eqref{edo_modulation} becomes (see Lemma \ref{lem_DNLS_BL2})
\[ A_{\zeta,h}[u(t)]    \begin{pmatrix}  \dot{\theta}(t)      \\
						     \dot{p}(t)
			 \end{pmatrix} =  b_{\zeta,h}[u(t)] + 2 E_4  ,\quad \mbox{where}\quad  E_4 = \begin{pmatrix}   \langle  \cos \left(  \frac{2\pi x}h \right) |v|^2 v , \eta_{\zeta}^h \rangle_{L^2(\mathbb{R})}        \\
						     -\langle \cos \left(  \frac{2\pi x}h \right) |v|^2 v  , i\partial_x \eta_{\zeta}^h \rangle_{L^2(\mathbb{R})}  
			 \end{pmatrix}.  \]

By construction $\eta_{\zeta}^h$ generates a traveling wave of the perturbation of DNLS whose speed is $\zeta$. It means we can apply Proposition \ref{prop_what_TW_DNLSA} with $\mathfrak{u}(t,x):=e^{i\zeta_1}\eta_\zeta^h(x-\zeta_2 t)$. However, we have $e^{-i\zeta_1}\mathfrak{u}(t,\cdot+\xi_2 t) = \eta_\zeta^h \in \Span (i \eta_\zeta^h,\partial_x \eta_\zeta^h)^{\perp_{L^2}}$. So calculating $\partial_t \mathfrak{u}$ with Equation \eqref{stat5} of Proposition \ref{prop_what_TW_DNLSA}, we get
\[  A_{\zeta,h}[\eta_{\zeta}^h]   \zeta =  b_{\zeta,h}[\eta_{\zeta}^h] .\]
Consequently, we have
\begin{equation}
\label{dec_speed_error}
  A_{\zeta,h}[u(t)]    \begin{pmatrix}  \dot{\theta}(t)-\zeta_1      \\
						     \dot{p}(t) - \zeta_2
			 \end{pmatrix} =  \left( b_{\zeta,h}[u(t)] - b_{\zeta,h}[\eta_{\zeta}^h] \right) - A_{\zeta,h}[u(t)-\eta_{\zeta}^h]\zeta + 2 E_4 .
\end{equation}

It is with this equation that we will obtain an estimate on $\dot{\theta}(t)-\zeta_1$ and $ \dot{p}(t) - \zeta_2$. Indeed, as we have seen in the second step, since $t<T$, $A_{\zeta,h}[u(t)]$ is invertible and $\|A_{\zeta,h}[u(t)]^{-1}\|_1 \leq C$. So we just need to control the three terms in the right-hand side of the previous equation. 
\begin{itemize}
\item We first prove that $b_{\zeta,h}$ is a Lipschitz function on bounded subsets of $BL^2_h$, for the norm $\| \cdot \|_{H^1(\mathbb{R})}$, uniformly with respect to $\zeta$ and $h$. Considering the first coordinate (see \eqref{chopin}), we have
\[ (b_{\zeta,h}[v])_1 =  \langle  \Delta_h v + |v|^2 v , \eta_{\zeta}^h \rangle_{L^2(\mathbb{R})}  =   \langle   v , \Delta_h \eta_{\zeta}^h \rangle_{L^2(\mathbb{R})} + \langle    |v|^2 v , \eta_{\zeta}^h \rangle_{L^2(\mathbb{R})}.\]
But $\| \Delta_h \eta_{\zeta}^h \|_{L^2(\mathbb{R})} \leq \| \partial_x^2 \eta_{\zeta}^h \|_{L^2(\mathbb{R})} \leq C$ (see \eqref{summ_reg}) and $v\mapsto |v|^2 v$ is a Lipschitz function on bounded subsets of $H^1(\mathbb{R})$. So, since $\| \eta_{\zeta}^h\|_{H^1(\mathbb{R})} \leq C$ and $\| u(t)- \eta_{\zeta}^h\|_{H^1(\mathbb{R})} \leq \rho$, there exists a constant $k>0$ (depending only of $C$ and $\rho$) such that 
\[ |(b_{\zeta,h}[u(t)] -  b_{\zeta,h}[\eta_{\zeta}^h])_1| \leq k \| u(t) - \eta_{\zeta}^h\|_{H^1(\mathbb{R})}. \]
Since $\| \eta_{\zeta}^h\|_{H^3(\mathbb{R})}\leq C$, the second coordinate of $(b_{\zeta,h}[u(t)] -  b_{\zeta,h}[\eta_{\zeta}^h])_1$ clearly enjoys the same estimate.
\item Since $\| \eta_{\zeta}^h\|_{H^1(\mathbb{R})}\leq C$, it is obvious, from the definition of $A_{\zeta,h}$ (see \eqref{def_A}) that there exists an universal constant $c>0$ such that
\[  \|  A_{\zeta,h}[u(t)-\eta_{\zeta}^h]  \|_1 \leq c C \| u(t) - \eta_{\zeta}^h\|_{H^1(\mathbb{R})} .\]
\item We can estimate $E_4$ as we have estimated $E_1(t)$ in the previous paragraph. Consequently, we get some constants $M,\ell$ independent of $h$ and $\zeta$ such that
\[  |E_4| \leq M e^{-\frac{\ell}h} + M \| u(t) - \eta_{\zeta}^h\|_{H^1(\mathbb{R})}. \] 
\end{itemize}
Applying these three estimates and the control of the norm of the invert of $A_{\zeta,h}[u(t)]$, we get from \eqref{dec_speed_error}
\[  | \dot{\theta}(t)-\zeta_1   | + | \dot{p}(t)-\zeta_2 | \leq CM e^{-\frac{\ell}h} + C(M+k+cC) \| u(t) - \eta_{\zeta}^h\|_{H^1(\mathbb{R})}.\]
However, we have proven that $\| u(t) - \eta_{\zeta}^h\|_{H^1(\mathbb{R})}\leq \delta(t) + (1+C+C^3)\delta(0)$ and $|\xi - \zeta|\leq C \delta(0)$. So, since $\dot{\theta} = \dot{\gamma}$ and $\dot{p} = \dot{\xO}$, we have proven that
\[  | \dot{\gamma}(t)-\xi_1   | + | \dot{\xO}(t)-\xi_2 | \leq K (e^{-\frac{\ell}h} + \delta(t) + \delta(0)) ,\]
where $K$ depends only of $C,M,c$ and $k$.

\section{Appendix}
\subsection{Proof of Theorem \ref{thm_stab_reg}}
\label{sec_app_proof}
Let $s>0$, $\varepsilon\in(0,2)$ and $n\in \mathbb{N}^*$ be such that $n\geq n_0 \geq 2$ where $n_0\in \mathbb{N}^*$ will be determined later to be large enough. Let $\rho>0$ and $v\in H^n(\mathbb{R})$ be such that
\[ \| v \|_{\dot{H}^n(\mathbb{R})} \leq \rho \quad \textrm{ and } \quad \|  \psi_\xi - v  \|_{H^1(\mathbb{R})} \leq \frac{r}{2(1+\kappa)},\]
with $\xi\in \Omega$. Let $h_1<h_0$ a constant that we will determine later.

Now consider $h<h_1$ and $\ug$ a solution of DNLS such that 
\[   \exists \xO_0,\gamma_0\in \mathbb{R}, \quad \forall g\in h\mathbb{Z}, \quad \ \ug_g(0) = e^{i\gamma_0}v(g - \xO_0).\]
We denote by $u$ the Shannon interpolation of $\ug$. Without loss of generality, since DNLS is invariant by gauge transform, we can assume $\gamma_0 = 0$.

\begin{lem} \label{lem_jaliaze} The following inequality holds:
 \[   \| u_0 - \eta_\xi^h(\cdot- \xO_0 ) \|_{H^1(\mathbb{R})} \leq \| v -  \eta_\xi^h \|_{H^1(\mathbb{R})} + h^{n-1} \rho. \]
\end{lem}
This lemma is a classical estimate of aliasing, it will be proven at the end of this subsection.

Since $u_0,\eta_\xi^h\in BL^2_h$, we can apply Lemma \ref{lem_compare_norms} to obtain
\begin{equation}
\label{est_delta0_end}
 \delta(0):=\| \ug(0) -  \left( \eta_\xi^h(\cdot- \xO_0 ) \right)_{| h\mathbb{Z}} \|_{H^1(h\mathbb{Z})} \leq  \| u_0 - \eta_\xi^h(\cdot- \xO_0 ) \|_{H^1(\mathbb{R})} \leq \| v -  \eta_\xi^h \|_{H^1(\mathbb{R})} + h^{n-1} \rho.  
\end{equation}
Applying the triangle inequality, we deduce of Theorem \ref{Thm_DTW} that
\[ \delta(0) \leq  \| v -  \psi_\xi \|_{H^1(\mathbb{R})} + \| \psi_\xi -  \eta_\xi^h \|_{H^1(\mathbb{R})} + h^{n-1} \rho \leq  \frac{r}{2(1+\kappa)} + \kappa h^2 + h^{n-1} \rho. \]
Consequently, if $h_1$ is small enough then $\delta(0)\leq \frac{r}{1+\kappa}$. So we can apply Theorem \ref{Thm_DTW} and Theorem \ref{thm_stab_Hn}. In particular, we get functions $\gamma,\xO \in C^1(\mathbb{R}_+)$ such that, if for all $t\in (0,T)$
\begin{equation}
\label{boot_cond}
 \delta(t) := \| \ug(t) - (e^{i\gamma(t)}\eta_\xi^h(\cdot-\xO(t)))_{| h\mathbb{Z}} \|_{H^1(h\mathbb{Z})} \leq r, 
\end{equation}
then we have for all $t\in (0,T)$
\begin{equation}
\label{growth_delta}
\delta(t) \leq \kappa \ \left(  \delta(0) + e^{-\frac{\ell}h}+\sqrt{t |\xi_2|} h^{n - \frac32} \sup_{0<s<t} \| \ug(s) \|_{\dot{H}^{n-1}(h\mathbb{Z})} \right), 
\end{equation}
and 
\begin{equation}
\label{mod_control_end}
 | \dot{ \gamma }(t) -\xi_1 | +  | \dot{ \xO }(t) -\xi_2 | \leq \kappa \ (\delta(0)+\delta(t) + e^{-\frac{\ell}h }). 
\end{equation}

Applying Theorem \ref{thm_growth}, we deduce that if \eqref{growth_delta} is satisfied then 
\begin{equation}
\label{growth_delta_improved}
\delta(t) \leq \kappa \ \left(  \delta(0) + e^{-\frac{\ell}h}+   C \sqrt{|\xi_2|} t^{\frac{n}2} h^{n - \frac12} M_{\ug(0)}^{\frac{4n-1}3} + C\sqrt{|\xi_2|}\sqrt{t} h^{n - \frac12} \left( \| \ug(0) \|_{\dot{H}^{n}(h\mathbb{Z})} +   M_{\ug(0)}^{\frac{2n+1}3}   \right) \right),
\end{equation}
where
\[  M_{\ug(0)} = \|\ug(0) \|_{\dot{H}^1(h\mathbb{Z})} +\|\ug(0) \|_{L^2(h\mathbb{Z})}^3  .\]

So, to use \eqref{growth_delta_improved}, we have to estimate $M_{\ug(0)}$ and $\| \ug(0) \|_{\dot{H}^{n}(h\mathbb{Z})}$ uniformly with respect to $h$ and $\xi$. We get these bounds in the following lemma that will be proven at the end of this subsection. 
\begin{lem}
\label{lem_aliasing_again}
There exists a constant $K>0$, depending only of $\Omega,\rho$ and $n$ such that for all $h<h_0$,
\[ \kappa C  M_{\ug(0)}^{\frac{4n-1}3} \leq K \quad \textrm{ and }\quad  \kappa C  \left( \| \ug(0) \|_{\dot{H}^{n}(h\mathbb{Z})} +   M_{\ug(0)}^{\frac{2n+1}3} \right) \leq K. \]
\end{lem}
With the estimate, \eqref{growth_delta_improved} becomes 
\begin{equation}
\label{growth_delta_plusplus}
\delta(t) \leq \kappa   \delta(0) + \kappa e^{-\frac{\ell}h}+   K \sqrt{|\xi_2|}  t^{\frac{n}2} h^{n - \frac12} + K \sqrt{|\xi_2|}  \sqrt{t} h^{n - \frac12} .
\end{equation}

Now, we overcome the bootstrap condition \eqref{boot_cond}. Let $T_0\in (0,\infty]$ be a function of $|\xi_2|$ that will be fixed later. Consider $t\in(0,T_0 h^{-2 + \varepsilon})$ such that for all $\tau \leq t$, $\delta(\tau)\leq r$.
We deduce from \eqref{growth_delta_plusplus} that
\[  \delta(t) \leq  \kappa   \delta(0) + \kappa e^{-\frac{\ell}h}+   K T_0^{\frac{n}2} \sqrt{|\xi_2|}  h^{ \frac{n \varepsilon -1}2} + K \sqrt{T_0 |\xi_2|} h^{n - \frac32 + \frac{\epsilon}2}. \]
Assuming $n_0\geq \max(2,\frac{1+2s}{\varepsilon},\frac{s+3-\varepsilon}2)$, $h_1\leq1$ and $T_0 = \min(|\xi_2|^{-1},|\xi_2|^{-\frac1n})$, we deduce
\begin{equation}
\label{eq_win}
 \delta(t) \leq  \kappa   \delta(0) + \kappa e^{-\frac{\ell}h}+  2 K  h^{s}  \leq  \kappa   \delta(0) + \left( \kappa \left( \frac{s}{\ell e}\right)^s + 2K \right)    h^{s}  .
 \end{equation}
 So assuming $h_1 <  \left[ \frac{r}{1+\kappa} \left( \kappa \left( \frac{s}{\ell e}\right)^s + 2K \right)^{-1} \right]^{-s}$, we get 
 $ \delta(t) < r .$
 Consequently, proceeding as usual by contradiction, we deduce that it was useless to assume that for all $\tau \leq t$, $\delta(\tau)\leq r$.

Finally, to conclude rigorously this proof, we have to explain how to get \eqref{eq_delta_hs} and \eqref{eq_mod_hs}. On the one hand, to get \eqref{eq_delta_hs}, we just have to estimate $\delta(0)$ by \eqref{est_delta0_end} in \eqref{eq_win} (and to assume that $n_0-1\geq s$). On the other hand, we have to estimate the terms of \eqref{mod_control_end}. We control $\delta(0)$ as previously, $\delta(t)$ by \eqref{eq_delta_hs} and $e^{-\frac{\ell}h}$ by $\left( \frac{hs}{\ell e}\right)^s$.

\medskip 

\noindent{\bf Proof of Lemma \ref{lem_jaliaze}.}
Let $v_h$ be the $L^2$ orthogonal projection of $v$ on $BL^2_h$, $i.e.$
\[ \widehat{v_h} = \mathbb{1}_{\left(-\frac{\pi}h, \frac{\pi}h \right)} \widehat{v}. \]
We introduce $w_h = u_0 - v_h(\cdot- \xO_0)$. Since the $H^1$ norm is invariant by advection, we have
\[  \| u_0 - \eta_\xi^h(\cdot- \xO_0 ) \|_{H^1(\mathbb{R})} \leq \| v_h - \eta_\xi^h \|_{H^1(\mathbb{R})} + \|w_h\|_{H^1(\mathbb{R})}.\]
Since $\eta_\xi^h \in BL^2_h$, $v-v_h$ is orthogonal to $\eta_\xi^h$ in $H^1(\mathbb{R})$. Consequently, we have $\| v_h - \eta_\xi^h \|_{H^1(\mathbb{R})}\leq \| v -  \eta_\xi^h \|_{H^1(\mathbb{R})} $. So we just have to prove that $ \|w_h\|_{H^1(\mathbb{R})}\leq \rho h^{n-1}.$

 Applying Proposition \ref{prop_per}, we have
\[  \forall \omega \in \left(-\frac{\pi}h, \frac{\pi}h \right),\quad  \ \widehat{w_h}(\omega) = \sum_{k\in \mathbb{Z}^*} e^{-i\left(\omega + \frac{2k\pi}h\right)\xO_0} \widehat{v}(\omega + \frac{2k\pi}h). \]
Consequently, we have
$$
\| w_h \|_{H^1(\mathbb{R})} \leq \frac1{\sqrt{2\pi}}  \sum_{k\in \mathbb{Z}^*} \| \widehat{v}(\omega + \frac{2k\pi}h) \sqrt{1+\omega^2} \|_{L^2  \left(-\frac{\pi}h, \frac{\pi}h \right) } 
							\leq \frac1{\sqrt{2\pi}} \sum_{k\in \mathbb{Z}^*} \left\| \widehat{\partial_x v}(\omega + \frac{2k\pi}h) \frac{\sqrt{1+\omega^2}}{\omega + \frac{2k\pi}h}  \right\|_{L^2  \left(-\frac{\pi}h, \frac{\pi}h \right) }  .
$$
Assuming $h_1 \leq 2\pi$, we have $\left| \frac{\sqrt{1+\omega^2}}{\omega + \frac{2k\pi}h} \right| \leq \frac2{2|k|-1}$ for $\omega \in  \left(-\frac{\pi}h, \frac{\pi}h \right)$. Consequently, applying Cauchy-Schwarz inequality, we get
\[ \| w_h \|_{H^1(\mathbb{R})} \leq \| \partial_x (v-v_h) \|_{L^2(\mathbb{R})} \sqrt{\sum_{k\in \mathbb{Z}^*} \frac{4}{(2|k|-1)^2}} = \frac{\pi}{\sqrt2}\| \partial_x (v-v_h) \|_{L^2(\mathbb{R})} .    \]
Since the Fourier support of $v-v_h$ is localized outsize $\left[ -\frac{\pi}h, \frac{\pi}h \right]$ and $n\geq 2$, we have
\[  \| w_h \|_{H^1(\mathbb{R})} \leq  \frac{\pi}{\sqrt2}\| \partial_x (v-v_h) \|_{L^2(\mathbb{R})} \leq \left( \frac{h}{\pi} \right)^{n-1}\frac{\pi}{\sqrt2}\| \partial_x^n (v-v_h) \|_{L^2(\mathbb{R})} \leq h^{n-1} \frac{\pi^{2-n}}{\sqrt2} \rho \leq h^{n-1} \rho.  \]
 
 \medskip

\noindent{\bf Proof of Lemma \ref{lem_aliasing_again}.} There are two quantities to control, $\| \ug(0)\|_{\dot{H}^n(h\mathbb{Z})}$ and $M_{\ug(0)}$.
To control $\| \ug(0)\|_{\dot{H}^n(h\mathbb{Z})}$, it is enough to prove that the restriction to $h\mathbb{Z}$ is a continuous map from $\dot{H}^n(\mathbb{R})$ to $\dot{H}^n(h\mathbb{Z})$, uniformly with respect to $h$. Indeed, denote $w = v(\cdot - \xO_0)$. Then applying Proposition \ref{prop_per}, we have, for all $\omega\in \left( - \frac{\pi}h,\frac{\pi}h \right)$,
\[ \widehat{u_0}(\omega) = \sum_{k\in \mathbb{Z}^*} \widehat{w}(\omega + \frac{2k \pi}h).\]
Since for $k\neq 0$ and $\omega \in \left( -\frac{\pi}h,\frac{\pi}h\right)$, we have
\[  \left|\frac{\omega}{\omega+ \frac{2k\pi}h } \right| \leq \frac1{2|k|-1},\]
applying Cauchy Schwarz inequality (and \eqref{eq_dis_cont_sobolev} ), we get
\begin{align*}
\| \ug(0) \|_{\dot{H}^n(h\mathbb{Z})} & \leq \| \omega^n \widehat{u_0} \|_{L^2\left( -\frac{\pi}h,\frac{\pi}h\right)} \\
&\leq \| \omega^n \widehat{w} (\omega)\|_{L^2\left( -\frac{\pi}h,\frac{\pi}h\right)} + \sum_{k\in \mathbb{Z}^*}  \left\| \left( \frac{\omega}{\omega + \frac{2k\pi}h }\right)^n \widehat{\partial_x^n w} (\omega+ \frac{2k\pi}h )\right\|_{L^2\left( -\frac{\pi}h,\frac{\pi}h\right)} \\
&\leq \| \partial_x^n w \|_{L^2(\mathbb{R})} + \sum_{k\in \mathbb{Z}^*}  \|  \widehat{\partial_x^n w} (\omega+ \frac{2k\pi}h )\|_{L^2\left( -\frac{\pi}h,\frac{\pi}h\right)} \frac1{(2|k|-1)^n} \\
&\leq \| \partial_x^n w \|_{L^2(\mathbb{R})} + \| \partial_x^n w \|_{L^2(\mathbb{R})} \sqrt{ \sum_{k \in  \mathbb{Z}^* }\frac1{(2|k|-1)^{2n}} } \\
&= \left(1+ \sqrt{2\left(1-\frac1{4^n}\right)\zeta(2n)} \right) \| \partial_x^n v \|_{L^2(\mathbb{R})} \leq \left(1+ \sqrt{2\left(1-\frac1{4^n}\right)\zeta(2n)} \right) \rho,
\end{align*}
where $\zeta$ is the Riemann zeta function.

Finally, to control $M_{\ug(0)}$, we just have to control $\| \ug(0) \|_{H^{1}(h\mathbb{Z})}$. But since we have proven that $\delta(0) \leq \frac{r}{1+\kappa}$, we just need to control $\| \eta_\xi^h\|_{H^1(\mathbb{R})}$ uniformly with respect to $\xi\in \Omega$ and $h<h_0$. Such an estimate can be obtained by using the bound $\|\eta_\xi^h - \psi_\xi^h\|_{H^1(\mathbb{R})} \leq \kappa h^2$ of Theorem \ref{Thm_DTW}.

\subsection{Proof of Lemma \ref{lem_modulation_constructor}}
\label{sec_proof_lem_modulation_constructor}

We would like to define the functions $\gamma$ and $\xO$ from $\theta$ and $p$. So we introduce a new time: $T_{{\rm crit}}$. It is the largest time, smaller than $T_{\max}$, such that for all $t\in (0,T_{{\rm crit}})$, we have
\begin{equation}
\label{assump1}
  \| ( A_{\zeta,h}[T_{\theta(t),p(t)}^{-1} u(t)] )^{-1}\| \leq 2C 
\end{equation}
and
\begin{equation}
\label{assump2}
 |\theta(t)-\theta_0| + |p(t)-p_0| \leq 1 + c_2 t,
\end{equation}
where $c_2>0$ is a real constant that will be determine later.

Now we define $\gamma$ and $\xO$ as $C^1$ functions on $\mathbb{R}_+$ such that 
\begin{equation}
\label{def_gamma_x0}
\forall t \in (0,T_{{\rm crit}}), \ \gamma(t) = \theta(t) - \delta_{\gamma}  \textrm{ and } \xO(t) = p(t) - \delta_{\xO}  . 
\end{equation}
Let $T>0$ be such that for all $t<T$, $\delta(t) =\|  u(t) - T_{\gamma(t),\xO(t)} \eta_\xi^h\|_{H^1(\mathbb{R})}< r$. To prove Lemma \ref{lem_modulation_constructor}, it is enough to prove that $T\leq T_{{\rm crit}}$. 
We proceed by contradiction. Assume that $T_{{\rm crit}} < T$. So if $t<T_{{\rm crit}}$, we have 
\[ \|  u(t) - T_{\theta(t),p(t)} \eta_{\zeta}^h\|_{H^1(\mathbb{R})}\leq (2+C+C^3)r \leq \rho.\]
Applying \eqref{summ_lem_matrix_inv}, we know that
\begin{equation}
\label{tempo}
A_{\zeta,h}[T_{\theta(t),p(t)}^{-1} u(t)] \quad \textrm{ is invertible and } \quad  \| ( A_{\zeta,h}[T_{\theta(t),p(t)}^{-1} u(t)] )^{-1}\|_1 \leq C .
\end{equation}
Furthermore, we can estimate $ \langle T_{\theta(t),p(t)}^{-1} \partial_t u(t) , i \eta_{\zeta}^h\rangle_{L^2(\mathbb{R})}$ and $\langle T_{\theta(t),p(t)}^{-1} \partial_t u(t) , \partial_x \eta_{\zeta}^h\rangle_{L^2(\mathbb{R})}$. Indeed, since $u$ is a solution of DLNS in $BL^2_h$ (see Lemma \ref{stop_discrete}), we have
\[ \langle T_{\theta(t),p(t)}^{-1} \partial_t u(t) , i \eta_{\zeta}^h\rangle_{L^2(\mathbb{R})} = - \left\langle  \Delta_h u(t) + \left( 1+2\cos\left( \frac{2\pi x}h \right) \right) |u(t)|^2 u(t) ,  T_{\theta(t),p(t)}^{-1} \eta_{\zeta}^h\right\rangle_{L^2(\mathbb{R})}.\]
Since this operator is symmetric for the $L^2$ norm, we have
\[ \langle T_{\theta(t),p(t)}^{-1} \partial_t u(t) , i \eta_{\zeta}^h\rangle_{L^2(\mathbb{R})} = - \langle   u(t) ,  T_{\theta(t),p(t)}^{-1} \Delta_h \eta_{\zeta}^h\rangle_{L^2(\mathbb{R})} - \left\langle   \left( 1+2\cos\left( \frac{2\pi x}h \right) \right) |u(t)|^2 u(t) ,  T_{\theta(t),p(t)}^{-1} \eta_{\zeta}^h\right\rangle_{L^2(\mathbb{R})}.\]
We are going to estimate these terms. Since $t<T$, by definition, we have $\|  u(t) - T_{\gamma(t),\xO(t)} \eta_\xi^h\|_{H^1(\mathbb{R})}<r$ and so
\[ \|  u(t)  \|_{H^1(\mathbb{R})} \leq r + C .\] Consequently, we have 
\[  \| |u(t)|^2 u(t)\|_{L^2(\mathbb{R})} \leq \| u(t) \|_{L^{\infty}}^2 \| u(t)\|_{L^2(\mathbb{R})} \leq (r + C)^3.\]
Furthermore, we have seen in \eqref{eq_dis_cont_sobolev} that $\| \Delta_h \eta_{\zeta}^h \|_{L^2(\mathbb{R})} \leq  \| \partial_x^2 \eta_{\zeta}^h \|_{L^2(\mathbb{R})}  \leq  C$.
Consequently, we have
\[ |\langle T_{\theta(t),p(t)}^{-1} \partial_t u(t) , i \eta_{\zeta}^h\rangle_{L^2(\mathbb{R})}| \leq C(r + C)^3 + C(r+C).  \]
Similarly, we could prove that 
\[ |\langle T_{\theta(t),p(t)}^{-1} \partial_t u(t) , \partial_x \eta_{\zeta}^h\rangle_{L^2(\mathbb{R})}| \leq C(r + C)^3 + C(r+C).  \]
So, we have proven that
\[ \max(|\dot{\theta}(t)| , |\dot{p}(t)| ) \leq C^2 (r+C) (1+ (r+C)^2).\]
Defining $c_2 = 2 C^2 (r+C) (1+ (r+C)^2)$, we have
\[   |\theta(t)-\theta_0| + |p(t)-p_0| \leq  c_2 t.\]

We can apply this inequality and \eqref{tempo} for $t=T_{{\rm crit}} $, so we have
\[ \| ( A_{\zeta,h}[T_{\theta(T_{\rm crit}),p(T_{\rm crit})}^{-1} u(t)] )^{-1}\|_1 \leq C  \textrm{ and } |\theta(T_{\rm crit})-\theta_0| + |p(T_{\rm crit})-p_0| \leq c_2 T_{\rm crit}. \]
But it is impossible because by definition of $T_{{\rm crit}}$ we should have
\[ \| ( A_{\zeta,h}[T_{\theta(T_{\rm crit}),p(T_{\rm crit})}^{-1} u(t)] )^{-1}\|_1 = 2C  \textrm{ or } |\theta(T_{\rm crit})-\theta_0| + |p(T_{\rm crit})-p_0| = 1 + c_2 T_{\rm crit}.\]
So, here is the contradiction and we have proven that $T\leq T_{{\rm crit}}$.

\subsection{Inverse function Theorem}
In this subsection, we give a version of the inverse function theorem.
\begin{theo}
\label{Thm_Inv_loc}
Let $X,Y$ be some Banach spaces, $\Omega$ be an open convex subset of $X$ such that $0\in \Omega$.\\
 If $g:\Omega \to Y$ is a $C^1$ function such that
\begin{itemize}
\item $\diff g(0)$ is invertible,
\item $\diff g$ is a $k$-Lipschitz function,
\end{itemize} 
then, defining $\beta = \|  \diff g(0)^{-1} \|^{-1}$ and $r = \frac{\beta}k$, we have
\begin{itemize}
\item $g$ is a $C^1$ diffeomorphism from $B_{X}(0,r) \cap \Omega$ to $g(B_{X}(0,r) \cap \Omega)$,
\item for all $x \in B_{X}(0,r) \cap \Omega$, $ \| \diff g(x) ^{-1} \| \leq \frac{r}{\beta(r- \|x\|)} $,
\item for all $0<\rho \leq r$, if  $B_X(0,\rho)\subset \Omega$ then $B_Y(g(0),  \frac{\beta}2 \rho ) \subset g(B_X(0,\rho))$.
\end{itemize}
\end{theo}
\begin{proof}
First, we prove that $g$ is injective on $B_{X}(0,r) \cap \Omega$. Let $y\in B_{X}(0,r) \cap \Omega$. We introduce the application
\begin{equation*}
\Phi_y :\left\{ \begin{array}{lll} B_{X}(0,r) \cap \Omega & \to & X \\
									x & \mapsto & x - \diff g(0)^{-1} (g(x) - g(y)).
\end{array} \right.
\end{equation*}
It is enough to prove that $y$ is the only fix point of $\Phi_y$. But if $x\in B_{X}(0,r) \cap \Omega $ then
\begin{equation}
\label{contrac}
 \| \diff  \Phi_y(x) \| = \| I_{X} - \diff g(0)^{-1} \diff g(x)  \| \leq  \|  \diff g(0)^{-1}  \| \|  \diff g(0) - \diff g(x)  \| \leq  \frac{\|x\| k}{\beta} < \frac{r k}{\beta} =1.  
\end{equation}
Consequently, we deduce that if $x\neq y$ then $\| \Phi_y(x) - y \| < \| x - y \|$ and so $y$ is the only fix point of $\Phi_y$.

Then, we prove that $\diff g(x)$ is invertible for any $x\in B_{X}(0,r) \cap \Omega $. Indeed, we have
\[  \diff g(x) = \diff g(0) + \diff g(x) - \diff g(0)  = \diff g(0) \left[ I_{X} + \diff g(0)^{-1} (\diff g(x) -\diff g(0) )  \right]  \]
with
\[ \| \diff g(0)^{-1}  (\diff g(x) -\diff g(0) )  \| \leq \frac{k}{\beta} \|x \| < 1. \]
So we also deduce the second point of the theorem through the classical estimate of the Von Neumann series.

Now, applying the classical inverse function theorem, we have proven that $g$ is a $C^1$ diffeomorphism from $B_{X}(0,r) \cap \Omega$ to $g(B_{X}(0,r) \cap \Omega)$.
Finally, we just need to prove the last assertion of the theorem. Let $\rho>0$ be such that $0<\rho \leq r$,  $B_X(0,\rho)\subset \Omega$. We introduce $\delta \in (0,\rho)$ to prove that $\overline{B_Y(g(0),  \frac{\beta}2 \delta )} \subset g(\overline{B_X(0,\delta)})$.
It is enough to prove the last point because
\[ B_Y(g(0),  \frac{\beta}2 \rho ) =\bigcup_{0<\delta < \rho} \overline{B_Y(g(0),  \frac{\beta}2 \delta )}  \textrm{ and }  g(B_X(0,\rho)) = \bigcup_{0<\delta < \rho}  g(\overline{B_X(0,\delta)}). \]

Let $y\in \overline{B_Y(g(0),  \frac{\beta}2 \delta )} $, we want to solve $g(x)=y$. So, we introduce the application $\Psi = \Phi_{y \ | \overline{B_X(0,  \delta )} }$. We want to apply the Banach fix point theorem. We have proven in $\eqref{contrac}$ that $\Psi$ is $\frac{\delta k}{\beta}< 1$ Lipschitz, so we just need to prove that it preserves $ \overline{B_X(0,  \delta )} $. Indeed, we have
\begin{align*}
\| \Psi(x) \| &\leq \| \Psi(0) \| + \| \Psi(x) -  \Psi(0) \| \\
							&\leq \frac{\beta}2 \delta \| \diff g(0)^{-1} \| + \| \diff g(0)^{-1} \| \| g(x) - g(0) - \diff g(0)x  \| \\
							&\leq \frac{\delta}2 + \frac1{\beta} \left\| \int_0^1 \diff g(sx)x \ds - \diff g(0)x \right\| 
							\leq  \frac{\delta}2 + \frac{k \delta}{\beta} \frac{\delta}2  
							\leq  \delta.
\end{align*}
\end{proof}

\subsection{A result of coercivity}

\begin{lem}[A reformulation of a Weinstein result in \cite{MR820338}]
\label{est_wein}
If $\Omega$ is a relatively compact open subset of the set $\left\{ \xi \in \mathbb{R}^2 \ | \ \xi_1> \left(\frac{\xi_2}2 \right)^2\right\}$ then there exists $c>0$ such that for all $\xi \in \Omega$ we have
\begin{equation}
\label{coc}
 \forall v \in H^1(\mathbb{R})\cap \Span( \psi_\xi,i\psi_\xi,\partial_x \psi_\xi)^{\perp_{L^2}}, \  \diff^2\Lag_\xi(\psi_\xi)(v,v)\geq c\|v\|_{H^1(\mathbb{R})}^2.
\end{equation}
\end{lem}
\begin{proof}
Weinstein has proven in \cite{MR820338} that there exists $c>0$ such that for all $v\in H^1(\mathbb{R})$,
\begin{equation}
\label{W_res}
 v \in  \Span( \psi_{(1,0)},i\psi_{(1,0)}^3,\partial_x(\psi_{(1,0)}^3) )^{\perp_{L^2}} \Rightarrow  \diff^2\Lag_{(1,0)}(\psi_{(1,0)})(v,v)\geq c \| v\|_{H^1}^2.
\end{equation}
First, we will deduce from this estimate and Lemma \ref{lem_evn} that \eqref{coc} holds true for $\xi=(1,0)$. Then we will extend this result applying two transformations: \it dilatation and boost. \rm

\underline{Step 1: The case $\xi=(1,0)$.} $\empty$
We apply Lemma \ref{lem_evn} below, with the spaces $E = H^1(\mathbb{R})\cap \Span( \psi_{(1,0)})^{\perp_{L^2}}$, $G =H^1(\mathbb{R})\cap \Span( \psi_{(1,0)},i\psi_{(1,0)}^3,\partial_x(\psi_{(1,0)}^3) )^{\perp_{L^2}}$,  $F =H^1(\mathbb{R})\cap \Span( \psi_{(1,0)},i\psi_{(1,0)},\partial_x \psi_{(1,0)} )^{\perp_{L^2}}$ and eventually $H = \Span( i\psi_{(1,0)},\partial_x \psi_{(1,0)} )$. We equipped all these spaces with the $H^1(\mathbb{R})$ norm for which they are closed. By construction, $F$ and $H$ are obviously complementary spaces. However, we have to prove that $G$ and $H$ are complementary spaces.

First, we prove that $H\cap G= \{0\}$. If $g=\alpha i\psi_{(1,0)} + \beta \partial_x \psi_{(1,0)} \in G$ then $\langle g ,  i\psi_{(1,0)}^3 \rangle_{L^2(\mathbb{R})} = \langle g ,  \partial_x(\psi_{(1,0)}^3) \rangle_{L^2(\mathbb{R})}=0$. However, since $\psi_{(1,0)} $ is a real valued function, we have 
\begin{equation}
\label{rel_ortho}
\langle \partial_x \psi_{(1,0)} ,  i\psi_{(1,0)}^3 \rangle_{L^2(\mathbb{R})} =\langle \partial_x (\psi_{(1,0)}^3) ,  i\psi_{(1,0)} \rangle_{L^2(\mathbb{R})}=0.
\end{equation}
Consequently, we deduce that $\alpha \| \psi_{(1,0)}  \|_{L^4(\mathbb{R})}^4= \beta \langle \partial_x (\psi_{(1,0)}^3) ,  \partial_x \psi_{(1,0)} \rangle_{L^2(\mathbb{R})} =0$. So we just need to verify from \eqref{def_CTW} that $\langle \partial_x (\psi_{(1,0)}^3) ,  \partial_x \psi_{(1,0)} \rangle_{L^2(\mathbb{R})}\neq 0 $ which yields $\alpha=\beta=0$.

Now, we  prove that $H+G = E$. Since, by construction $G + \Span ( i\psi_{(1,0)}^3,\partial_x(\psi_{(1,0)}^3)) = E$, we just need to prove that $ i\psi_{(1,0)}^3,\partial_x(\psi_{(1,0)}^3) \in H+G$. Since $ i\psi_{(1,0)}^3$ and $\partial_x(\psi_{(1,0)}^3)$ are orthogonal, we can decompose $i\psi_{(1,0)}$ and $\partial_x \psi_{(1,0)}$ through the decomposition $E = G + \Span ( i\psi_{(1,0)}^3,\partial_x(\psi_{(1,0)}^3)) $ to get (with \eqref{rel_ortho})
\[ \left\{  \begin{array}{lll}  i \psi_{(1,0)} \| \psi_{(1,0)}  \|_{L^6(\mathbb{R})}^6  - \| \psi_{(1,0)}  \|_{L^4(\mathbb{R})}^4 i\psi_{(1,0)}^3 \in G, \\
									\partial_x \psi_{(1,0)} \| \partial_x(\psi_{(1,0)}^3) \|_{L^2(\mathbb{R})}^2  - \langle \partial_x (\psi_{(1,0)}^3) ,  \partial_x \psi_{(1,0)} \rangle_{L^2(\mathbb{R})} \partial_x(\psi_{(1,0)}^3) \in G.
\end{array} \right. \]
Since the coefficients associated with $ i\psi_{(1,0)}^3$ and $ \partial_x (\psi_{(1,0)}^3)$ are not zero, we deduce that $ i\psi_{(1,0)}^3,\partial_x(\psi_{(1,0)}^3) \in H+G$.

In order to apply Lemma \ref{lem_evn}, with $b=\diff^2\Lag_{(1,0)}(\psi_{(1,0)})$ we have to prove that $\partial_x \psi_{(1,0)}$ and $i \psi_{(1,0)}$ belong to the kernel of $\diff^2\Lag_{(1,0)}(\psi_{(1,0)})$. Indeed, since $\Lag_{(1,0)}(\psi_{(1,0)})$ is invariant by gauge transform and dilatation, the set of its critical points are also invariant by these transform, i.e.
\[ \forall t\in \mathbb{R}, \forall v\in H^1(\mathbb{R}), \ \diff \Lag_{(1,0)}(e^{it} \psi_{(1,0)})(v) = \diff \Lag_{(1,0)}( \psi_{(1,0)}(.- t))(v)=0.  \]
However, since $\psi_{(1,0)}$ is a very regular function (see Lemma \ref{reg_CTW} or directly \eqref{def_CTW}), we can compute the derivative in $t=0$ to get
\[ \forall t\in \mathbb{R}, \forall v\in H^1(\mathbb{R}), \ \diff^2 \Lag_{(1,0)}(\psi_{(1,0)})(i\psi_{(1,0)},v) = \diff^2 \Lag_{(1,0)}(\psi_{(1,0)})(\partial_x \psi_{(1,0)},v)=0.  \]

Now to apply Lemma \ref{lem_evn}, we observe that the required assumption of coercivity of $b$ on $G$ is the result of Weinstein \eqref{W_res}, and we obtain the result. 

\underline{Step 2: Extension by dilatation and boost} 

Denote by $T$ the dilatation action defined by $T_m(u)(x) = mu(mx)$ for all $x\in \mathbb{R}$, $u\in H^1(\mathbb{R})$ and $m>0$, and let $B$ bz the boost action defined by $B_\nu u:= e^{i\nu x} u $ for all $x\in \mathbb{R}$, $u\in H^1(\mathbb{R})$ and $\nu \in \mathbb{R}$. These transformations are useful because we have the following relations 
\[\forall m,\mu>0, \forall \nu\in \mathbb{R}, \ \Lag_{(1,0)}\circ T_m = m^3 \Lag_{(m^{-2} ,0)} \textrm{ and }  \Lag_{(\mu,0)}\circ B_\nu = \Lag_{(\mu+\nu^2,-2\nu)}\]
With these relations a straightforward calculation shows that
\begin{equation}
\label{transp}
 \Lag_\xi = m_\xi^3 \Lag_{(1,0)} \circ T_{m_\xi^{-1}} \circ B_{-\frac{\xi_2}2} \textrm{ with } m_\xi=\sqrt{\xi_1-\left( \frac{\xi_2}2 \right)^2}.
\end{equation}
Furthermore,  using the definition of $\psi_{\xi}$, we have
\[ \psi_\xi = B_{\frac{\xi_2}2}\circ T_{m_\xi}\psi_{(1,0)}.\]

Consequently, we are able to transport the coercivity property from $\xi=(1,0)$ to any $\xi$, provided  that $\xi_1 >\left( \frac{\xi_2}2 \right)^2$. First, we observe that if $v \in H^1(\mathbb{R})\cap \Span(\psi_{\xi},i\psi_{\xi},\psi_{\xi}')^{\perp_{L^2}}$ then
\[ T_{m_\xi^{-1}} \circ B_{-\frac{\xi_2}2} v \in  H^1(\mathbb{R})\cap \Span(\psi_{(1,0)},i\psi_{(1,0)},\psi_{(1,0)}')^{\perp_{L^2}}. \]
Second, we calculate the derivative of the Lagrange function through the transport relation \eqref{transp},
\[  \diff \Lag_{\xi} (\psi_\xi)(v) = m_\xi^3 \diff [\Lag_{(1,0)} \circ T_{m_\xi^{-1}} \circ B_{-\frac{\xi_2}2} ] (\psi_\xi)(v) = m_\xi^3 \diff\Lag_{(1,0)} (\psi_{1,0} ) (  T_{m_\xi^{-1}} \circ B_{-\frac{\xi_2}2} v)=0.\]
Then we deduce a property of coercivity
\[  \diff^2 \Lag_\xi (\psi_\xi)(v,v) = m_\xi^3 \diff^2 \Lag_{(1,0)} (\psi_{1,0})(T_{m_\xi^{-1}} \circ B_{-\frac{\xi_2}2} v,T_{m_\xi^{-1}} \circ B_{-\frac{\xi_2}2} v) \geq c m_\xi^3 \left\|T_{m_\xi^{-1}} \circ B_{-\frac{\xi_2}2} v\right\|_{H^1}^2 . \] 
This inequality implies Estimate \eqref{coc} because applying Peetre inequality \footnote{ If $x,y\in \mathbb{R}$ then $1+(x-y)^2\geq \frac12 (1+x^2)(1+y^2)^{-1}$.}, we get 
\[  \left\|T_{m_\xi^{-1}} \circ B_{-\frac{\xi_2}2} v\right\|_{H^1}^2 = \left\| B_{-\frac{m_\xi^{-1} \xi_2}2} \circ T_{m_\xi^{-1}}  v\right\|_{H^1}^2 \geq \frac12 \frac{ \| T_{m_\xi^{-1}}  v\|_{H^1}^2 }{1+\left( \frac{m_\xi^{-1} \xi_2}2\right)^2} = \frac12 \frac{ m_{\xi}^{-1} \|v\|_{L^2}^2 +m_{\xi}^{-3} \|\partial_x v\|_{L^2}^2 }{1+\left( \frac{m_\xi^{-1} \xi_2}2\right)^2} .  \]
\end{proof}

\subsection{Functional analysis lemmas}
%
%

\begin{lem}
\label{lem_evn}
Let $F$, $G$ be two closed subspaces of a normed space $E$. If $F$ and $H$ admit a same finite dimensional complementary space $H$, denote by $\Pi$ the projection onto $G$ of kernel $H$. Then $\Pi_{| F}$ is a normed space vector isomorphism. \\
Furthermore, if $b$ is a bilinear symmetric form on $E$, $H$ is a subspace of its kernel and if there exists $\alpha>0$ such that
\[ \forall x\in G, \ b(x,x) \geq \alpha \|x\|^2 \]
then there exists $\beta>0$ such that
\[ \forall x\in F, \ b(x,x) \geq \beta \|x\|^2.\]
\end{lem}
\begin{proof}
Let $P$ be the projection onto $F$ of kernel $H$. If $f\in F$ then $P \Pi f = f$. Indeed, if $f=g+h$ with $g\in G$ and $h\in H$ then $g=\Pi f = f-h$. Consequently, we would have $f = P g = P\Pi f $.  Similarly, we can prove that $\Pi P g=g$, for any $g\in G$. So, we have proven that $\Pi_{| F}^{-1} = P_{|G}$.\\
To prove the first part of the lemma, we just have to prove that $\Pi$ and $P$ are continuous to conclude this proof. This is a very classical exercise of normed space vector, whose proof is based on compactness.

The second part of the lemma is a straightforward calculation. Indeed, if $x\in F$ then
\[ b(x,x) = b(\Pi x,\Pi x) \geq \alpha \|  \Pi x \|^2 \geq  \alpha \| \Pi_{| F}^{-1} \|^{-2} \|x\|^2.  \]
\end{proof}

\begin{lem} 
\label{perturb_coer}
Let $E$ be a real vector space whose $(x_j)_{j=1,\dots,n}$ is a free family. Define $X = \Span(x_j)_{j=1,\dots,n}$ the subspace generated by this family. Let $\langle \cdot,\cdot \rangle_1$,$\langle \cdot,\cdot \rangle_2$ be two scalar products on $E$ such that the induced norms satisfy $\|  \cdot\|_1 \leq c \| \cdot\|_2$. Define $G \in M_n(\mathbb{R})$ the Gram matrix associated to $(x_j)_{j=1,\dots,n}$ for the scalar product $\langle \cdot,\cdot\rangle_1$, i.e.
\[ G = \begin{pmatrix} \langle x_1,x_1\rangle_1 & \dots & \langle x_1,x_n\rangle_1 \\
						\vdots & & \vdots \\
						 \langle x_n,x_1 \rangle_1 & \dots & \langle x_n,x_n\rangle_1
\end{pmatrix}. \]
For any $u \in E$, let $b(u)$ be a bilinear symmetric form continuous for the $\|\cdot \|_2$ norm. Assume that $b$ is $k$ Lipschitz on a ball of radius $R>0$, i.e.
\[  \forall u,v \in B_2(0,R),\quad \forall y,z\in E,\quad  \ | b(u)(y,z) -b(v)(y,z)| \leq k \| u- v\|_2 \|y\|_2 \|z\|_2   \]
and that there exists $\alpha>0$ such that
\[\forall y \in X^{\perp_1}, \quad \ b(0)(y,y) \geq \alpha \|y\|_2^2. \]

Define two constants $c_1,c_2>0$ by the explicit formulas
\[  c_1 = \max(R, \frac{\alpha }{8 k} )  \textrm{ and } c_2 = \frac{\alpha}4 \left[ \left( \sum_{j=1}^n \| x_j \|_{2} \right) \| G^{-1} \|_{\infty} \left(\frac{\alpha}2 + \| b(0) \|_2+ \frac{2 \| b(0) \|_2^2}{\alpha}\right)  \right]^{-1} .\]

If $\| u\|_2 \leq c_1$ and $\displaystyle \sup_{j=1,\dots,n} |\langle x_j,y \rangle_1 |\leq c_2 \|y\|_2$ then
\[  b(u)(y,y) \geq \frac{\alpha}8 \| y\|_2^2. \]
\end{lem}
\begin{proof}
Let $y=y_{\parallel} +y_{\perp}$ be the decomposition of $y$ associated to the algebraic decomposition $E = X \oplus X^{\perp_1}$. So, we get
\begin{align*}	b(0)(y,y)  &= b(0)(y_{\parallel} +y_{\perp}  ,y_{\parallel} +y_{\perp}) \\
					               &=  b(0)( y_{\perp}  ,y_{\perp})b(0)+ 2 b(0)(y_{\parallel}  ,y_{\perp}) + (y_{\parallel},y_{\parallel} )   \\
					               &\geq \alpha \| y_{\perp} \|_2^2 - 2 \| b(0) \|_2  \|  y_{\parallel} \|_2 \| y_{\perp}\|_2 -| b(0) \|_2  \| y_{\parallel}\|_2 \\
					               &\geq \frac{\alpha}2 \| y_{\perp} \|_2^2 -   \left(\| b(0) \|_2+ \frac{2 \| b(0) \|_2^2}{\alpha} \right) \| y_{\parallel}\|_2^2 \\
					               &\geq \frac{\alpha}2 \| y \|_2^2 - \left(\frac{\alpha}2 + \| b(0) \|_2+ \frac{2 \| b(0) \|_2^2}{\alpha}\right)  \| y_{\parallel}\|_2^2.
\end{align*}
Consequently, we just need to control $\| y_{\parallel}\|_2$ with $\| y \|_2$ to get the result when $u=0$. However, using basis linear algebra we can prove that
\[    y_{\parallel} = \sum_{j=1}^n a_j x_j \textrm{ with } (a_j)_{j=1,\dots,n} = G^{-1} ( \langle x_j,y \rangle_1 )_{j=1,\dots,n}.\]
So, we get
\[  \| y_{\parallel} \|_2 \leq c_2 \left( \sum_{j=1}^n \| x_j \|_{2} \right) \| G^{-1} \|_{\infty} \|y\|_{2}.   \]
Finally, by definition of $c_2$, we get $b(0)(y,y) \geq \frac{\alpha}4 \| y \|_2^2$. Furthermore, since $b$ is k Lipschitz on $B(0,R)$, we deduce directly that if $\| u \|_1 \leq c_1$ then $b(u)(y,y) \geq \frac{\alpha}8 \| y \|_2^2$.

\end{proof}

\begin{lem}
\label{lem_exRiesz}
Let $E$ be a Banach space of dual space $E'$. Consider a algebraic decomposition of $E$, $E=E_p\oplus E_m$, and a continuous linear application $T:E\to E'$ such that 
\begin{enumerate}[i)]
\item $\forall x,y\in E, \ \langle Tx,y \rangle_{E',E}=\langle Ty,x \rangle_{E',E}$,
\item $\exists \alpha_p>0, \ \forall x\in E_p, \ \langle Tx,x \rangle_{E',E} \geq \alpha_p \|x\|^2$,
\item $\exists \alpha_m>0, \ \forall x\in E_m, \ \langle Tx,x \rangle_{E',E} \leq - \alpha_m \|x\|^2$.
\end{enumerate}
Then $T$ is invertible and we have
\begin{equation}
\label{est_linv}
 \|T^{-1}\|\leq \left( \frac1{\alpha_p} + \frac1{\alpha_m}  +  \frac{2\|T\|}{\alpha_m \alpha_p} + \frac{\|T\|^2}{\alpha_m (\alpha_p)^2} \right) .
\end{equation}
\end{lem}
\begin{proof}
In the proof we omit the index $E',E$ for all the duality brackets. We define by restrictions  $T_{\epsilon_1\epsilon_2}\in \Lag(E_{\epsilon_2};E_{\epsilon_1})$ for $\epsilon_1,\epsilon_2\in\{p,m\}$. Then we use a direct corollary of Riesz Theorem to prove that $T_{pp}$ is invertible. This corollary is the following.
\begin{lem}
\label{lem_Riesz}
Let $E$ be a Banach space of dual $E'$. Consider a continuous linear application $T:E\to E'$ such that 
\begin{enumerate}[i)]
\item $\exists \alpha>0,\forall x\in E, \ \langle Tx,x \rangle\geq \alpha\|x\|^2$,
\item $\forall x,y\in E, \ \langle Tx,y \rangle=\langle Ty,x \rangle$,
\end{enumerate}
then $T$ is invertible and $ \|T^{-1} \|\leq \alpha^{-1}$.
\end{lem}
Now, decomposing $x=x_p+x_m$ with $x_p\in E_p$ and $x_m\in E_m$, we introduce operators $P:E\to E_p$ and $S:E_m \to E_m'$ defined by
\[   Px=x_p+T_{pp}^{-1}T_{pm}x_m \textrm{ and } S=T_{mm} - T_{mp}T_{pp}^{-1}T_{pm}.\]
Then we verify by symmetry of $T$  (with the same decomposition for $y$) that
\[ \forall x,y\in E, \ \langle Tx,y \rangle = \langle T_{pp}Px,Py\rangle +   \langle Sx_m,y_m\rangle.\]

To prove the Lemma, we have to solve,
\begin{equation}
\label{eq_to_solve}
 \forall y\in E, \  \langle Tx,y \rangle=\phi(y) \textrm{ with }  \phi\in E'.
\end{equation}

Let $z\in E_m$ and denote $y = z - T_{pp}^{-1}T_{pm}z$. First, we verify that $Py=0$. Consequently, we deduce from \eqref{eq_to_solve} that
\[  \phi(y) = \phi( z - T_{pp}^{-1}T_{pm}z  )  =    \langle Sx_m,z \rangle .\]
However, we verify that $-S$ verifies assumptions  Lemma \ref{lem_Riesz} with $\alpha=\alpha_m$. Consequently, $S$ is invertible and so we have
\[ x_m = S^{-1}\phi_{| E_m} - S^{-1}\phi T_{pp}T_{pm}. \]

Now if we apply \eqref{eq_to_solve} for $y=y_p\in E_p$, we have
\[ \phi(y) =  \langle T_{pp}Px,y\rangle = \langle T_{pp}x_p,y\rangle  +   \langle T_{pm}x_m,y\rangle.\]
Consequently, we have
\[  x_p=T_{pp}^{-1} \phi_{| E_p} - T_{pp}^{-1}T_{pm}x_m.\]

Finally, we have solved \eqref{eq_to_solve}. So $T$ is bijective and we verify \eqref{est_linv} using the estimate given by Lemma \ref{lem_Riesz}.
\end{proof}

%
%
%

\bibliographystyle{plain}
\bibliography{traveling_waves_dnls}

\end{document}